\documentclass[10pt, a4paper]{article}
\usepackage[french, english]{babel}
\usepackage[utf8,latin1]{inputenc}
\usepackage{amsmath}
\usepackage{amssymb}
\usepackage{amsthm}
\usepackage{latexsym,tikz}
\usepackage{hyperref}
\usepackage{enumerate}
\usepackage[all]{xy}
\usepackage{mathrsfs}
\usepackage{dsfont}
\usepackage{tikz}
\usetikzlibrary{decorations.pathmorphing}
\usepackage{appendix}
\usetikzlibrary{decorations.pathreplacing}
\usepackage{pdfpages}

\usepackage[nottoc]{tocbibind}

\newtheorem{thm}{Theorem}[section]
\newtheorem{cor}[thm]{Corollary}
\newtheorem{lem}[thm]{Lemma}
\newtheorem{prop}[thm]{Proposition}
\newtheorem{defn}[thm]{Definition}

\newtheorem{rem}[thm]{Remark}

\newtheorem{ex}[thm]{Example}
\newtheorem{thmdef}[thm]{Theorem and Definition}

\def\fA{\mathfrak{A}}
\def\rA{\mathrm{A}}

\def\fB{\mathfrak{B}}

\def\nC{\mathnormal{C}}

\def\cD{\mathcal{D}}

\def\det{\mathrm{det}}

\def\rF{\mathrm{F}}
\def\rG{\mathrm{G}}

\def\GL{\mathrm{GL}}

\def\Hom{\mathrm{Hom}}
\def\rH{\mathrm{H}}

\def\Ind{\mathrm{Ind}}
\def\ind{\mathrm{ind}}

\def\cK{\mathcal{K}}

\def\rL{\mathrm{L}}

\def\rM{\mathrm{M}}

\def\rN{\mathrm{N}}
\def\fo{\mathfrak{o}}

\def\rP{\mathrm{P}}

\def\rQ{\mathrm{Q}}

\def\res{\mathrm{res}}

\def\SL{\mathrm{SL}}

\def\rU{\mathrm{U}}

\def\rV{\mathrm{V}}

\def\rX{\mathrm{X}}
\def\rY{\mathrm{Y}}
\def\rZ{\mathrm{Z}}

\begin{document}





\hypersetup{							
pdfauthor = {Peiyi Cui},			
pdftitle = {representation mod l of SL_n},			
pdfkeywords = {Tag1, Tag2, Tag3, ...},	
}					
\title{$\ell$-modular representations of $p$-adic groups $\mathrm{SL_n}(F)$: maximal simple $k$-types}
\author{Peiyi Cui \footnote{peiyi.cui@uea.ac.uk, Institut de recherche math\'ematique de Rennes, Universit\'e de Rennes 1 Beaulieu, 35042 Rennes CEDEX, France; Current address: School of Mathematics, University of East Anglia, Norwich Research Park, Norwich NR4 7TJ, United Kingdom}}

\date{\vspace{-2ex}}
\maketitle

        \begin{abstract}
        Let $p$ be an arbitrary prime number and $k$ an algebraically closed field of characteristic $l\neq p$. We establish the theory of maximal simple $k$-types of Levi subgroups $\rM'$ of $\mathrm{SL}_n(F)$, where $F$ is a non-archimedean locally compact field of residual characteristic $p$. We prove the exhaustion property and the unicity property of weakly intertwining implying conjugacy for maximal simple $k$-types, extended maximal simple $k$-types and simple $k$-characters of $\rM'$.
        \end{abstract}

\tableofcontents

\section{Introduction}

\subsection{Backgrounds and motivations}
Let $F$ be a non-archimedean locally compact field whose residue field is of characteristic $p$, and $\textbf{A}$ a connected algebraic reductive group defined over $F$. We say that $\rA$ is a $p$-adic group, if $\rA$ is the group $\textbf{A}(F)$ of the $F$-rational points of $\textbf{A}$, and we endow it with the locally pro-finite topology through $F$. Let $k$ be an algebraically closed field of characteristic $\ell$ with $\ell\neq p$ and $\ell\neq 0$, $W(k)$ the ring of Witt vectors of $k$ and $\cK$ an algebraic closure of the fractional field of $W(k)$ (the characteristic of $\cK$ is equal to $0$). We call a $k$-representation of $\rA$ an $\ell$-modular representation. In this article, we consider the $\ell$-modular representations, and all representations are assumed to be smooth. We denote by $\mathrm{Rep}_k(\rA)$ the category of smooth $k$-representations of $\rA$.

The theory of $k$-representations has great similarity with the theory of $\cK$-representations, but also has important differences. For examples: a $k$-representation of a compact group need not to be semisimple in general; there exist cuspidal representations of $\rA$ which are not supercuspidal, et cetera. For these reasons, only part of the methods applied in the study of $\cK$-representations are expected to be generalised to $k$-representations, and the theory of types is one of them. The theory of types consists of two parts: maximal simple $k$-types and $\rA$-covers. The first part comes from a conjecture which has been formulated for a long time, and the aim is to construct irreducible cuspidal representations (see Definition \ref{cusp 1}) from a compact open subgroup of $\rA$. The second part is to study the cuspidal support as well as the supercuspidal support of irreducible representations. 

Now we focus on the first part. To summarise briefly, a \textbf{maximal simple $k$-type} is a pair $(K,\pi)$ consisting of an open compact subgroup $K$ of $\mathrm{A}$ and an irreducible $k$-representation $\pi$ of $K$, which are constructed under technical conditions. For each maximal simple $k$-type, it is associated to a family of pairs $(\mathbb{K},\Pi)$, where $\mathbb{K}$ is a subgroup of $\mathrm{A}$, which is open compact modulo the centre containing $K$, and $\Pi$ is an extension of $\pi$ to $\mathbb{K}$, such that the compact induction $\rho:=\ind_{\mathbb{K}}^{\mathrm{A}}\Pi$ is irreducible and cuspidal (not necessarily supercuspidal). We say that $(K,\pi)$ constructs $\rho$, and $(\mathbb{K},\Pi)$  is an \textbf{extended maximal simple type} of $\rho$. The set of maximal simple $k$-types is said to verify \textbf{the exhaustion property}, if each irreducible cuspidal $k$-representation can be constructed as above, and is said to verify \textbf{the unicity property}, if for each $\rho$ the maximal simple types that constructs $\rho$ is unique up to $\rA$-conjugation. 

The theory of maximal simple types was firstly studied for representations of characteristic zero. It has been established in \cite{BuKu} and \cite{BuKuIII} for $\mathrm{GL}_n(F)$, and in \cite{BuKuI}, \cite{BuKuII} and \cite{GoRo} for $\mathrm{SL}_n(F)$ and its Levi subgroups, and in \cite{KS} for $p$-adic classical groups. In 2001, a construction of essentially tame cuspidal representations of characteristic zero through Bruhat-Tits buildings for general connected reductive groups is given in \cite{Yu}. An equivalence of these two constructions while considering essentially tame cuspidal representations of $\mathrm{GL}_n(F)$ has been studied in \cite{Ma}. For $\ell$-modular representations, Vignéras suggested in her book \cite{V1} the method of reduction modulo $\ell$, which gives a description of maximal simple $\ell$-modular types of $\mathrm{GL}_n(F)$. In \cite{MS} the maximal simple $\ell$-modular types of the inner forms of $\mathrm{GL}_n(F)$ were constructed. In \cite{Fin} a construction of essentially tame cuspidal $\ell$-modular representations of $p$-adic connected reductive groups is given. In this article, we establish the theory of maximal simple $k$-types of Levi subgroups $\rM'$ of $\mathrm{SL}_n(F)$, which means we prove the exhaustion property as well as the unicity property. It is a generalisation of \cite{BuKuI}, \cite{BuKuII}, \cite{GoRo} and \cite{TA} to the $\ell$-modular setting.

The theory of types has been widely used in the study of representations of characteristic zero of $\rA$. One well-known result is the description of Bernstein centre of $\mathrm{Rep}_{\cK}(\rA)$. In the $\ell$-modular setting, it was applied to a wider range of topics: for examples, in \cite{AKMSS}, \cite{Se1} and \cite{CLH} on the distinguished problem by a Galois involution; in \cite{Helm} and \cite{SeSt2} on the establishment of blocks decomposition of $\mathrm{Rep}_{W(k)}(\mathrm{GL}_n(F))$ and $\mathrm{Rep}_k(\mathrm{GL}_m(D))$, where $\mathrm{GL}_m(D)$ denotes the inner forms of $\mathrm{GL}_n(F)$. In \cite{C} the author uses the results of this article to establish category decompositions of $\mathrm{Rep}_k(\rM')$. In particular, it is showed in \cite{C} that a block of $\mathrm{Rep}_k(\rM')$ is defined with respect to several inertially equivalent supercuspidal classes, which is different from the blocks of $\mathrm{Rep}_{\cK}(\rA)$ or the blocks of $\mathrm{Rep}_k(\mathrm{GL}_n(F))$. In the latter cases, the blocks are defined with respect to one inertially equivalent supercuspidal class.

\subsection{Main results}
Now we state the main results in this article with more details. We first recall the definition of cuspidal $k$-representations and supercuspidal $k$-representations.

\begin{defn}
\label{cusp 1}
Let $\rA$ be a $p$-adic group, and $\rL$ a Levi subgroup of $\rA$. Let $\pi$ be an irreducible $k$-representation of $\rA$, and $i_{\rL}^{\rA}$ be the normalised parabolic induction from $\rL$ to $\rA$. 
\begin{itemize}
\item We say that $\pi$ is \textbf{cuspidal}, if for any proper Levi subgroup $\rL$ and any irreducible $k$-representation $\rho$ of $\rL$, $\pi$ does not appear as a sub-representation or a quotient representation of $i_{\rL}^{\rA}\rho$;
\item We say that $\pi$ is \textbf{supercuspidal}, if for any proper Levi subgroup $\rL$ and any irreducible $k$-representation $\rho$ of $\rL$, $\pi$ does not appear as a sub-quotient representation of $i_{\rL}^{\rA}\rho$.
\end{itemize}
\end{defn}

From now on, we always denote $\mathrm{GL}_n(F)$ by $\rG$, and $\mathrm{SL}_n(F)$ by $\rG'$. Let $K$ be a closed subgroup of $\rG$, we always denote the intersection $K\cap\rG'$ by $K'$. In particular, let $\rM$ be a Levi subgroup of $\rG$, then $\rM'=\rM\cap\rG'$ is a Levi subgroup of $\rG'$, which gives a bijection between Levi subgroups of $\rG$ and $\rG'$. A basic idea of this article is by applying the restriction functor $\res_{\rM'}^{\rM}$. It has been firstly studied in \cite{TA} for $\cK$-representations of $\rM'$. A key observation is that, for any irreducible representation $\pi'$ of $\rM'$, there exists a $\pi$, irreducible of $\rM$, such that $\pi'$ appears as a direct component of the restriction $\pi\vert_{\rM'}$, which holds true for $\ell$-modular setting. With this relation in mind, in Section \ref{subsection 2.4.1} we construct a family of \textbf{maximal simple $k$-types} $(\tilde{J}_{\rM}',\tilde{\lambda}_{\rM}')$ of $\rM'$ from a maximal simple $k$-types $(J_{\rM},\lambda_{\rM})$ of $\rM$. In particular, let $\pi$ and $\pi'$ be irreducible and cuspidal of $\rM'$ and $\rM$ respectively, such that $\pi'\hookrightarrow \pi\vert_{\rM'}$. If $\pi$ is constructed from $(J_{\rM},\lambda_{\rM})$, then $\pi'$ is constructed from a $(\tilde{J}_{\rM}',\tilde{\lambda}_{\rM}')$. In Section \ref{subsection 2.4.2} we establish the theory of \textbf{extended maximal simple $k$-types} $(N_{\rM'}(\tilde{\lambda}_{\rM}'),\tau_{\rM'})$, which is to say that $N_{\rM'}(\tilde{\lambda}_{\rM}')$ is open compact modulo the centre and containing $\tilde{J}_{\rM}'$, and that $\tau_{\rM'}$ is an extension of $\tilde{\lambda}_{\rM}'$. 

We summerise Theorem \ref{Ltheorem 5}, Theorem \ref{thm a33}, and Definition \ref{Ldefinition 200} as below:
\begin{thmdef}[Exhaustion property]
We have:
\begin{itemize}
\item The pairs $(\tilde{J}_{\rM}',\tilde{\lambda}_{\rM}')$ are \textbf{maximal simple $k$-types} of $\rM'$.
\item The paris $(N_{\rM'}(\tilde{\lambda}_{\rM}'),\tau_{\rM'})$ are \textbf{extended maximal simple $k$-types} of $\rM'$, which means that $\tau_{\rM'}$ is an extension of $\tilde{\lambda}_{\rM'}$ and $\ind_{N_{\rM'}(\tilde{\lambda}_{\rM}')}^{\rM'}$ is irreducible and cuspidal.
\item The maximal simple $k$-types of $\rM'$ verifies the exhaustion property.
\end{itemize}
We always abbreviate $(J_{\rG},\lambda_{\rG})$ and $(\tilde{J}_{\rG}',\tilde{\lambda}_{\rG}')$ as $(J,\lambda)$ and $(\tilde{J}',\tilde{\lambda}')$ respectively.
\end{thmdef}

Since a $\cK$-representation of a compact subgroup is always semisimple, being intertwined coincides with being weakly intertwined. However, they are \textbf{not the same} for $\ell$-modular setting (see Definition \ref{prepdefn 01}). In this article, we prove the following Theorem of weakly intertwining implying conjugacy, which says that two maximal simple $k$-types of $\rM'$ are weakly intertwined if and only if they are conjugate and hence are intertwined. This theorem is essentially required in \cite{C} to establish blocks of $\mathrm{Rep}_k(\rM')$, and it implies the unicity property of simple $k$-characters (resp. maximal simple $k$-types) contained in an irreducible cuspidal $k$-representation (Theorem \ref{thm a40}).

\begin{thmdef}[Unicity property of weakly intertwining implying conjugacy, Theorem \ref{thm a30}]
\label{thmdefunicity}
Let $(\tilde{J}_i',\tilde{\lambda}_i')$ be maximal simple $k$-types of $\rM'$ for $i=1,2$.
\begin{itemize}
\item We say that $\tilde{\lambda}_1'$ is \textbf{weakly intertwined} with $\tilde{\lambda}_2'$, if there exists a $m\in\rM'$ such that $\tilde{\lambda}_1'$ appears as a sub-quotient of $\ind_{m(\tilde{J}_2')\cap\tilde{J}_1'}^{\tilde{J}_1'}\res_{m(\tilde{J}_2')\cap\tilde{J}_1'}^{m(\tilde{J}_2')}m(\tilde{\lambda}_2')$, where $m(\cdot)$ is the conjugation by $m$.
\item Suppose that $(\tilde{J}_1',\tilde{\lambda}_1')$ is weakly intertwined with $(\tilde{J}_2',\tilde{\lambda}_2')$, then they are conjugate in $\rM'$.
\end{itemize}
\end{thmdef}

\subsection{The structure of this article}
In this section, we introduce the structure of this article. We clarify the difference between characteristic zero setting and the $\ell$-modular setting, and introduce the specialities of $\rM'$ from the technical point of view. The basic references are \cite{BuKuI}, \cite{BuKuII}, \cite{GoRo} and \cite{TA}.

This article has two parts: In Section \ref{chapter 4}, we establish maximal simple $k$-types of $\rG'$ and $\rM'$, and prove the exhaustion property; in Section \ref{chapter 01} we prove the unicity properties. 

For the first part (Section \ref{chapter 4}), the fact that the derived group of $\rM'$ is non-trivial but is trivial of $\rG'$ causes important differences in the theory of maximal simple types of $\rM'$ and $\rG'$. For an example, the extended maximal simple $k$-types coincide with maximal simple $k$-types for $\rG'$, but it is not true for $\rM'$. Hence we deal with $\rM'$ separately in Section \ref{section 06}.

For both $\rM'$ and $\rG'$, the establishment of maximal simple $k$-types has significant differences comparing to that of general linear groups and classical groups. We start from the most crucial speciality. Let $(J,\lambda)$ be a maximal simple $k$-type of $\rG=\mathrm{GL}_n(F)$, and $J^1$ the pro-$p$ radical of $J$, then the quotient $J\slash J^1$ is a finite general linear group. By \cite{BuKu}, we have
\begin{equation}
\label{equation 999}
\lambda\cong\kappa\otimes\sigma,
\end{equation}
where $\sigma$ is inflated from a cuspidial $k$-representation of $J\slash J^1$, and $\kappa$ is defined from a simple character (see Section \ref{Notation}) which relates to the depth of this type. We call $\sigma$ \textbf{the depth zero part}, and $\kappa$ \textbf{the simple character part} of $\lambda$. A similar decomposition as in (\ref{equation 999}) holds for maximal simple $k$-types of $p$-adic classical groups (see \cite{KS}). Unfortunately, such decomposition does \textbf{not} exist in general for maximal simple $k$-types of $\rG'$, which causes the main difficulty in Section \ref{chapter 4}. We explain below in detail how this phenomenon influences the establishment of the theory of maximal simple $k$-types of $\rG'$.

For $\rG'$, the maximal simple $k$-types are established in Section \ref{section 08} and Section \ref{section 2.3}. In the first three sections \ref{section 02.2.1}, \ref{section 02.2.2} and \ref{section 02.2.3}, we start by studying the maximal simple $k$-types of $\rG$, which is a preparation of the rest of this section. To be more precise, Section \ref{section 02.2.1} is a generalisation of Appendix in \cite{BuKuII} to $\ell$-modular setting; Section \ref{section 02.2.3} consists of useful properties on weakly intertwining. We end the preparation part by an important decomposition of $\res_{J}^{\rG}\ind_{J}^{\rG}\lambda$ in Section \ref{section 02.2.3}, where $(J,\lambda)$ is a maximal simple $k$-type of $\rG$. 

We come back to $\rG'$ from Section \ref{section 02.2.4}, where we introduce the group of projective normaliser $\tilde{J}$, which was firstly defined in \cite{BuKuII} for characteristic zero setting. Section \ref{section 07} is a key part, of which the purpose is to show that $\ind_{\tilde{J}'}^{\rG'}\tilde{\lambda}'$ is irreducible. When $\ell=0$, the irreducibility is proved by showing that the intertwining set of $\tilde{\lambda}'$ in $\rG'$ is equal to $\tilde{J}'$. However this condition is \textbf{not} sufficient for $\ell$-modular setting, and a second condition has to be added into the criteria of irreducibility (see Lemma \ref{lem 19}), which was firstly introduced by Vign\'eras in \cite{V3}. For $\rG$ and classical groups, this second condition is verified through Equation (\ref{equation 999}), by showing that the $\kappa\vert_{J^1}$-coinvariant of $\ind_{J}^{\rG}\lambda$ is semisimple. However, since Equation (\ref{equation 999}) does not exist in general for $\rG'$, it requires more technique to verify this condition: In Section \ref{section 07}, we pass the problem from $\rG'$ to $\rG$, and verify the second condition by using the decomposition of $\res_{J}^{\rG}\ind_{J}^{\rG}\lambda$ in Section \ref{section 02.2.3}. We complete the establishment of maximal simple $k$-types of $\rG'$ in Definition \ref{defn 22}, Section \ref{section 2.3}. 

In Section \ref{section 06}, we establish the maximal simple $k$-types of $\rM'$, where some results of $\rG'$ are applied. We emphasis that extended maximal simple $k$-types \textbf{only appear} when $\rM'\neq\rG'$ as explained above. We introduce this new object in Theorem \ref{Ltheorem 5} and Definition \ref{Ldefinition 200}. By technical reason, we accomplish the establishment by showing an extended maximal simple $k$-type is an extension of a maximal simple $k$-type at the end of this article (see Theorem \ref{thm a33}).

Now we arrive at the second part. In Section \ref{chapter 01}, the main purpose is proving the unicity property of weakly intertwining implying conjugacy (Theorem \ref{thm a30}). We start by establishing the $\ell$-modular Clifford theory for $\pi\vert_{\rM'}$ (Section \ref{section 03.1} and Section \ref{section 003.2}), of which the statement is different from the characteristic zero setting. Then we deduce a decomposition of $\tilde{J}_{\rM}'$ (Theorem \ref{thm a26}), which is the key of the proof of Theorem \ref{thm a30}. In Section \ref{section intconj}, except for the unicity property of weakly intertwining implying conjugacy, we also show the unicity property of the maximal simple $k$-types (resp. the simple $k$-characters) which are contained in a fixed irreducible cuspidal $k$-representation (Theorem \ref{thm a40}). Finally in Section \ref{chapter 02}, we accomplish the establishment of extended maximal simple $k$-types, which ends this article.

It is worth noticing that we obtain an equation in Theorem \ref{thm a20}, which is an intermediate step to prove Theorem \ref{thm a30}. It is a generalisation of an equation of the length of $\pi\vert_{\rM'}$ for irreducible cuspidal $\pi$ in the characteristic zero setting. However, in the $\ell$-modular setting, the equation in Theorem \ref{thm a20} describes the $\ell$-prime part of the length of $\pi\vert_{\rM'}$. This equation has a separate interest. In $\ell$-modular setting, it is hard to compute the length of a parabolically induced representation. In Example \ref{ex a37}, we give a computation of this problem of $\mathrm{SL}_2(F)$, which is an application of this equation and answers an open question at the end of \cite{Dat}. On the other hand, the length of $\pi\vert_{\rM'}$ is equal to the size of $L$-packet containing $\pi'$ when $\ell=0$, and in \cite{CLH} it is proved to be true in the $\ell$-modular setting when $n=2$ by applying this equation, which shows potential use in the further study of $L$-packets of $\rM'$.

\section*{Acknowledgement}
The first part of this article is based on my PhD thesis. I would like to thank my advisors Anne-Marie Aubert for her patience and her suggestions of the project, and Michel Gros for his support during the writing of my thesis. I thank Vincent S\'echerre for his helpful suggestions of improvements for the early version of this article. I thank Guy Henniart and Marie-France Vign\'eras for their interests in the property of weakly intertwining implying conjugacy. I thank Guy Henniart for his suggestion on considering simple $k$-characters of $\rM'$. I thank Alberto Mínguez for his interest on Example \ref{ex a37}.

\section{Maximal simple $k$-types}
\label{chapter 4}

\subsection{Notation}
\label{Notation}
We start by recalling some basic terminology and notations of theory of types given by Bushnell and Kutzko in \cite{BuKu}.

\begin{itemize}
\item[] $F=$ a non-archimedean locally compact field, whose residue field is characteristic $p$;
\item[] $\mathfrak{o}_F=$ the ring of integers of $F$;
\item[] $\mathfrak{p}_F=$ the maximal ideal of $\mathfrak{o}_F$;
\item[] $\varpi_F=$ a uniformizer of $\mathfrak{p}_F$;
\item[] $k=$ an algebraically closed field with characteristic $\ell\neq p$;
\item[] $W(k)=$ the ring of Witt vectors of $k$, and $\mathcal{K}=$ an algebraically closed field of the fractional field of $W(k)$.
\item[] $\ind=$ the compact induction functor; $\res=$ the restriction functor.
\item[] $\mathrm{Rep}_k(\rA)=$ the category of smooth $k$-representations of $\rA$, where $\rA$ is a $p$-adic group.
\end{itemize}

In this article, we denote by $\rG$ the group $\mathrm{GL}_n(F)$, and $\rG'$ the group $\mathrm{SL}_n(\rF)$, unless otherwise specified. Let $K$ be a subgroup of $\rG$, we always denote by $K'$ the intersection of $K$ and $\rG'$.

Let $V$ be an $F$-vector space with dimension $n$. Let $A$ be the $F$-algebra $\mathrm{End}_{F}(V)\cong\mathrm{M}_n(F)$. Let $\mathcal{L}=\{L_i\}_{i\in\mathbb{Z}}$ be an $\mathfrak{o}_F$-lattice chain in $V$ such that for any $i\in\mathbb{Z}$ and any $x\in F$, $xL_i\in\mathcal{L}$. For any integer $m\in\mathbb{Z}$, we set
$$\mathfrak{P}^{m}=\mathrm{End}_{\mathfrak{o}_F}^m(\mathcal{L})=\{x\in A;xL_i\subset L_{i=m},i\in\mathbb{Z}\}.$$
A hereditary order $\fA$ is an $\mathfrak{o}_F$-lattice as well as a subring of $A$ of the form $\mathrm{End}_{\mathfrak{o}_F}^0(\mathcal{L})$ for an $\mathfrak{o}_F$-lattice chain of $V$ as above. We set 
$$\rU(\fA)=\fA^{\times},$$
which is a subgroup of $A^{\times}\cong\mathrm{GL}_n(F)$, and write
$$\rU^m(\fA)=1+\mathfrak{P}^m,m\neq 1,$$
which are pro-$p$ subgroups of $A^{\times}$. Furthermore, we have
$$\bigcup_{m\in\mathbb{Z}}\mathfrak{P}^m=A,$$
and we set
$$\nu_{\fA}(x)=\mathrm{max}\{m\in\mathbb{Z};x\in\mathfrak{P}^m\},$$
for any $x\in A$.

Let $E\slash F$ be a field extension of $F$ such that there is an inclusion $E\hookrightarrow A$. We fix such an inclusion and view $V$ as an $E$-vector space, and write $B=\mathrm{End}_{E}(V)$, then we have $B\subset A$. Assume that $E^{\times}$ normalises $\fA$, we denote by $\mathfrak{B}=B\cap\mathfrak{A}$.

A stratum in $A$ is a $4$-tuple $[\fA,n,r,\beta]$ consisting of a hereditary order $\fA$, integers $n>r$, and an element $\beta\in A$ such that $\nu_{\fA}(\beta)\geq -n$ as defined in \cite[\S 1.5.1]{BuKu}. A stratum $[\fA,n,r,\beta]$ is simple, if it is satisfied with the conditions in \cite[Definition 1.5.5]{BuKu}, in particular it is required that $E=F[\beta]$ is a field. Given a simple stratum $[\fA,n,0,\beta]$, the subgroups $H^1(\beta)\subset J^1(\beta)\subset J(\beta)$ of $\rU(\fA)$ are defined in \cite[\S 3.1]{BuKu}, with the property that the quotient $J(\beta)\slash J^{1}(\beta)$ is isomorphic to $\prod^{e}\mathrm{GL}_{f}(k_E)$, where $f,e$ are integers and $k_E$ is the residue field of the field $E=F[\beta]$ inside $A$. We abbreviate them as $H^1,J^1$ and $J$ without causing ambiguity. 

A $k$-character of $H^1$ is a \textbf{simple character} if it verifies the conditions in \cite[Definition 3.2.3]{BuKu}, and we denote by $C_{k}(\fA,0,\beta)$ the set of simple $k$-characters of $H^1$. Let $\theta$ be a simple $k$-character of $H^1$, there is an unique irreducible $k$-representation $\eta(\theta)$ containing $\theta$ of $J^1$ (\cite[\S 5.1.1]{BuKu}), we call it the \textbf{Heisenberg representation} of $\theta$ and we abbreviate it as $\eta$ without causing ambiguity. Furthermore, there are irreducible $k$-representations of $J$, which are extensions of $\eta(\theta)$ and their intertwining sets contain $B^{\times}$. We call such an irreducible $k$-representation a \textbf{$\beta$-extension} of $\eta(\theta)$, as defined in \cite[Definition 5.2.1]{BuKu} (for complex representations) and \cite[\S 2.4]{MS} (for modulo $\ell$ representations).

The definition of simple $k$-types and maximal simple $k$-types can be found in \cite[Definition 5.5.10]{BuKu} in the case of complex representations and in \cite[Definition 2.9]{MS} in the case of modulo $\ell$ representations. We give a brief summary. A simple $k$-type of $\rG$ is a pair $(J,\lambda)$, where $J$ is a compact subgroup of $\rG$ defined from a simple stratum $[\rA,n,0,\beta]$, and $\lambda$ is of the form $\kappa\otimes\sigma$, where $\kappa$ is a $\beta$-extension of an irreducible $k$-representation $\eta(\theta)$ for a simple $k$-character $\theta$ in $C_{k}(\fA,0,\beta)$, and $\sigma$ is the inflation of $\otimes^{e}\pi_0$ from $J\slash J^1$ to $J$, where $\pi_0$ is an irreducible cuspidal $k$-representation of $\mathrm{GL}_f(k_E)$. A pair $(J,\lambda)$ is a maximal-simple $k$-type of $\rG$, if it is a simple $k$-type, and $e=1$.

For a Levi subgroup $\rM$ of $\rG$, a maximal simple $k$-type of $\rM$ is isomorphic to a product of maximal simple $k$-types of general linear groups, and the same for simple $k$-characters.

\subsection{Construction of cuspidal $k$-representations of $\rG'$}
\label{section 08}

In this section, we want to prove that for any $k$-irreducible cuspidal representation $\pi'$ of $\rG'$, there exists an open compact subgroup $\tilde{J}'$ of $\rG'$ and an irreducible representation $\tilde{\lambda}'$ of $\tilde{J}'$ such that $\pi'$ is isomorphic to $\ind_{\tilde{J}'}^{\rG'}\tilde{\lambda}'$ (\ref{thm 16}, \ref{cor 21} and \ref{defn 22}).

\subsubsection{Types $(J,\lambda\otimes\chi\circ\det)$}
\label{section 02.2.1}
In this section, we assume that $\ell\neq 0$.
 Let $(J,\lambda)$ be a maximal simple $k$-type of $\rG$ and $\chi$ be a $k$-quasicharacter of $F^{\times}$. We prove that the pair $(J,\lambda\otimes\chi\circ\det)$ is also a maximal simple $k$-type of $\rG$, which will be used in the proof of Proposition \ref{prop 1}. This has been proved in appendix of \cite{BuKuII} in the case of characteristic $0$ by using the following two lemmas, of which we give a proof for the case of characteristic $\ell$.

Let $\mu_{p^{\infty}}$ be the group of roots of unities of powers of $p$. We fix an injective morphism from $\mu_{p^{\infty}}$ to $\cK^{\times}$, which gives an injective morphism from $\mu_{p^{\infty}}$ to $k^{\times}$ by applying projection from $W(k)^{\times}$ to $k^{\times}$. Now we consider the Teichmuller character $\iota_{p,k}$ from $k^{\times}$ to $\cK^{\times}$. Let $K$ be a pro-$p$ group and $\theta_{k}$ be a smooth $k$-character of $K$, then composing with $\iota_{p,k}$ gives a mapping from the set of $k$-characters and $\cK$-characters of $K$. Conversely, a $\cK$-character $\theta_{\cK}$ of $K$ takes value in $W(k)^{\times}$, and by applying projection it gives a $k$-character of $K$, which we call the reduction modulo $\ell$ of $\theta_{\cK}$. We have:

\begin{lem}
\label{lem 3}
Let $K$ be a pro-$p$ group, the composition with $\iota_{p,k}$ gives a bijection between the set of smooth $k$-characters and the set of smooth $\cK$-characters of $K$, of which the inverse is reduction modulo $\ell$. Let $\theta_{k}$ be a smooth $k$-character of $K$, we call $\theta_{\cK}=\iota_{p,k}\circ\theta_{k}$ the lifting of $\theta_{k}$.
\end{lem}

Let $[\fA,n,0,\beta]$ be a simple stratum, as defined in \cite[\S 2.2]{MS}, the set of simple characters $C_{k}(\fA,0,\beta)$ consists of the reduction modulo $\ell$ of simple $\cK$-characters of $H^1(\beta)$.

From now on we fix a continuous additive character $\psi_F$ from $F$ to $\cK^{\times}$, which is trivial on $\mathfrak{p}_{F}$ but non-trivial on $\mathfrak{o}_F$. Recall the equivalence (depending on the choice of $\psi_F$)
$$(\mathrm{U}^{[\frac{1}{2}n]+1}(\fA)/ \mathrm{U}^{n+1}(\fA))^{\wedge}\cong \mathfrak{P}^{-n}/\mathfrak{P}^{-([\frac{1}{2}n]+1)},$$
where $(\mathrm{U}^{[\frac{1}{2}n]+1}(\fA)/ \mathrm{U}^{n+1}(\fA))^{\wedge}$ denote the Pontrjagin dual. Let $\beta\in\mathfrak{P}^{-n}/\mathfrak{P}^{-([\frac{1}{2}n]+1)}$, we denote by $\psi_{\beta,\cK}$ the character on $\mathrm{U}^{[\frac{1}{2}n]+1}(\fA)/ \mathrm{U}^{n+1}(\fA)$ given through the equivalence above (or consult \cite[\S 1.6.6]{BuKu} for an explicit definition). Let $\psi_{\beta,k}$ denote the reduction modulo $\ell$ of $\psi_{\beta,\cK}$.



\begin{lem}
\label{lem 2}
Let $[\mathfrak{A},n,0,\beta]$ be a simple stratum in $\mathrm{A}$ with $\beta\notin \mathrm{F}$, $n\geq1$. Let $c\in\mathrm{F}^{\times}$, and $n_0=-\mathnormal{v}_{\fA}(c)$, $n_1=-\mathnormal{v}_{\fA}(\beta+c)$, \begin{enumerate}
\item The stratum $[\mathfrak{A},n_1,0,\beta+c]$ is a simple stratum in $\rA$, and we have $\mathfrak{H}(\beta+c, \fA)=\mathfrak{H}(\beta,\fA)$ and $\mathfrak{J}(\beta+c,\fA)=\mathfrak{J}(\beta,\fA)$.
\item Let $\chi_{\cK}$ be a $\cK$-quasicharacter of $\mathrm{F}^{\times}$ such that $\chi_{\cK}\circ\mathrm{det}$ is equivalent to $\psi_{c,\cK}$ on $\mathrm{U}^{[\frac{1}{n_0}]+1}(\fA)$. Then we have an equivalence of simple $\cK$-characters:
$$\mathnormal{C}_{\cK}(\fA,0,\beta+c)=\mathnormal{C}_{\cK}(\fA,0,\beta)\otimes\chi_{\cK}\circ\mathrm{det}.$$
\item Let $\chi_{k}$ be a $k$-quasicharacter of $F^{\times}$ such that $\chi_k\circ\det$ is equivalent to $\psi_{c,k}$ on $\mathrm{U}^{[\frac{1}{n_0}]+1}(\fA)$. Then we have an equivalence of simple $k$-characters:
$$\mathnormal{C}_{k}(\fA,0,\beta+c)=\mathnormal{C}_{k}(\fA,0,\beta)\otimes\chi_{k}\circ\mathrm{det}.$$
\end{enumerate}
\end{lem}

\begin{proof}
The first two assertions are proved in Lemma of \cite[appendix]{BuKuII}, so we only need to prove the last assertion. Let $\chi_{0,\cK}$ be the lifting of $\chi_{k}\vert_{\mathfrak{o}_F^{\times}}$ to $\cK$, then $\chi_{0,\cK}$ can be extended to $F^{\times}$ and we denote also by $\chi_{0,\cK}$ an extension. By definition, $\chi_{0,\cK}\circ\mathrm{det}$ is equivalent to $\psi_{c,\cK}$. By Part $2$, we have:
$$\mathnormal{C}_{\cK}(\fA,0,\beta+c)=\mathnormal{C}_{\cK}(\fA,0,\beta)\otimes\chi_{0,\cK}\circ\mathrm{det}.$$
Applying reduction modulo $\ell$ to both sides, by Lemma \ref{lem 3} we obtain
$$\mathnormal{C}_{k}(\fA,0,\beta+c)=\mathnormal{C}_{k}(\fA,0,\beta)\otimes\chi_{k}\circ\mathrm{det}.$$

\end{proof}

\begin{cor}
\label{cor 3}
Let $(J,\lambda)$ be a maximal simple $k$-type of $\rG$, and $\chi$ a $k$-quasicharacter of $F^{\times}$. Then the pair $(J,\lambda\otimes\chi\circ\det)$ is also a maximal simple $k$-type of $\rG$.
\end{cor}

\begin{proof}
Let $[\fA,n,0,\beta]$ be a simple stratum in the construction of $(J,\lambda)$. We have $\lambda\cong\kappa\otimes\sigma$ (see Section \ref{Notation}) where $\kappa$ is a $\beta$-extension of $\eta(\theta)$ of a simple $k$-character $\theta$ in $C_k(\fA,0,\beta)$, hence $\lambda\otimes\chi\circ\det\cong(\kappa\otimes\chi\circ\det)\otimes\sigma$. From \cite[appendix]{BuKuII} and Part $1,3$ of Lemma \ref{lem 2}, we know that there exists $c\in F^{\times}$ such that $\theta\otimes\chi\circ\det$ is a simple $k$-character of $C_k(\fA,0,\beta+c)$. Hence $\eta(\theta)\otimes\chi\circ\det$ is the unique irreducible $k$-representation of $J^1$ containing $\theta\otimes\chi\circ\det$ after restricting to $H^1$. To end this proof, it is sufficient to prove that $\kappa\otimes\chi\circ\det$ is a $\beta$-extension of $\eta(\theta\otimes\chi\circ\det)$, which follows from the fact that the intertwining set of $\kappa\otimes\chi\circ\det$ is equal to that of $\kappa$.
\end{proof}

\begin{rem}
\label{rem 33}
Let $(J_{\rM},\lambda_{\rM})$ be a maximal simple $k$-type of $\rM$, where $\rM$ is a Levi subgroup of $\rG$. Then $\lambda_{\rM}\cong\lambda_1\otimes\cdots\otimes\lambda_r$ for some $r\in\mathbb{N}^{\ast}$, where $(J_i,\lambda_i)$ are maximal simple $k$-types of $\mathrm{GL}_{n_i}(F)$. Hence for any $k$-quasicharacter $\chi$ of $F^{\times}$, the tensor product $(J_{\rM},\lambda_{\rM}\otimes\chi\circ\det)$ is also a maximal simple $k$-type of $\rM$.
\end{rem}

\subsubsection{Intertwining and weakly intertwining}
\label{section 02.2.2}
Recall that $\rM$ is a Levi subgroup of $\rG$, and $\rM'=\rM\cap\rG'$ is a Levi subgroup of $\rG'$. Let $K$ be a subgroup of $\rM$, we always denote by $K'$ the intersection $K\cap\rM'$.

\begin{prop}
\label{prop 0.1}
\label{Lprop 0.1}
Let $K$ be a compact subgroup of $\rM$, and $\rho$ an irreducible $k$-representation of $K$. Then the restriction $\res_{K'}^K\rho$ is semisimple. 
\end{prop}

\begin{proof}
Let $O$ be the kernel of $\rho$, which is a normal subgroup of $K$. The subgroup $OK'$ is also a compact open normal subgroup of $K$, hence with finite index in $K$. We deduce that the restriction $\res_{OK'}^{K}\rho$ is semisimple by Clifford theory, and that the restriction $\res_{K'}^K \rho$ is semisimple.
\end{proof}

\begin{prop}
\label{prop 0.2}
\label{Lprop 0.2}
Let $K$ be a compact open subgroup of $\rM$, $\rho$ an irreducible smooth representation of $K$, and $\rho'$ an irreducible component of the restriction $\res_{K'}^K\rho$. Let $\overline{\rho}$ be an irreducible representation of $K$ such that $\res_{K'}^{K}\overline{\rho}$ contains $\rho'$ as well. Then there exists a $k$-quasicharacter $\chi$ of $F^{\times}$ such that $\rho\cong\overline{\rho}\otimes\chi\circ \mathrm{det}$.
\end{prop}

\begin{proof}

Let $U$ be a pro-$p$ normal subgroup of $K$ contained in the kernel of $\rho$, which has finite index in $K$. The induction $\Ind_{U'}^{U}(1)$ is semisimple. Hence by Schur lemma it is a direct sum of characters of the form $(\chi\circ\mathrm{det})\vert_U$, where $\chi$ is a $k$-quasicharacter of $F^{\times}$. The fact that $\res_{U'}^{K}\rho$ contains the trivial character induces the same property for $\res_{U'}^K\overline{\rho}$. By Frobenius reciprocity, we know that $\res_{U}^K\overline{\rho}$ contains a character of the form $(\chi\circ\mathrm{det})\vert_U$, and the irreducibility implies that it is in fact a multiple of this character. We can hence assume that $\overline{\rho}$ is trivial on $U$.

By Clifford theory, the restriction of $\rho$ (resp. $\overline{\rho}$) to $K'U$ is semisimple. Hence the set $\Hom_{K'U}(\rho,\overline{\rho})$ is non-trivial. Applying Frobenius reciprocity, we see that $\rho$ is a subrepresentation of $\ind_{K'U}^{K}\res_{K'U}^K\overline{\rho}$. The latter is equivalent to $\overline{\rho}\otimes\ind_{K'U}^{K}(1)$ by \cite[\S.I, 5.2 d]{V1}, of which the Jordan-H$\ddot{\mathrm{o}}$lder factors are $\overline{\rho}\otimes\chi\circ\mathrm{det}$. We finish the proof.

\end{proof}

\begin{defn}
\label{prepdefn 01}
Let $\rA$ be a locally profinite group, and $K_i$ be an open compact subgroup of $\rA$ for $i=1,2$. Let $\rho_i$ be an $k$-representation of $K_i$, and $x$ be an element in $\rG$. Define $i_{K_1,K_2}x(\rho_2)$ to be the induced $k$-representation 
$$\ind_{K_1\cap x(K_2)}^{K_1}\res_{K_1\cap x(K_2)}^{x(K_2)}x(\rho_2),$$
where $x(\rho_2)$ is the conjugate of $\rho_2$ by $x$. 
\begin{itemize}
\item We say that an element $x\in\rA$ \textbf{weakly intertwines} $\rho_1$ with $\rho_2$, if $\rho_1$ is an irreducible subquotient of $i_{K_1,K_2}x(\rho_2)$, and that $\rho_1$ is \textbf{weakly intertwined} with $\rho_2$ in $\rA$, if $\rho_1$ is isomorphic to a subquotient of $\res_{K_1}^{\rA}\ind_{K_2}^{\rA}\rho_2$. We denote by $\mathrm{I}_{\rA}^{w}(\rho_1,\rho_2)$ the set of elements in $\rA$, which weakly intertwines $\rho_1$ with $\rho_2$. 
\item We say that an element $x\in\rA$ \textbf{intertwines} $\rho_1$ with $\rho_2$, if
$$\mathrm{H}_{x}(\rho_1,\rho_2):=\Hom_{K_1}(\rho_1,i_{K_1,K_2}x(\rho_2)) \neq 0.$$
The representation $\rho_1$ is \textbf{intertwined with} $\rho_2$ in $\rA$, if
$$\mathrm{H}_{\rA}(\rho_1,\rho_2):=\Hom_{\rA}(\ind_{K_1}^{\rA}\rho_1,\ind_{K_2}^{\rA}\rho_2)\neq 0.$$
We denote by $\mathrm{I}_{\rA}(\rho_1,\rho_2)$ the set of elements in $\rA$, which intertwine $\rho_1$ with $\rho_2$. 
\end{itemize}
When $K_1=K_2$, we use the abbreviation $i_{K_1}x(\rho_1)$. When $\rho_1=\rho_2$, we use the abbreviation $\mathrm{I}_{\rA}^{w}(\rho_1)$, $\mathrm{I}_{\rA}(\rho_1)$, $\mathrm{H}_{x}(\rho_1)$, and $\mathrm{H}_{\rA}(\rho_1)$ accordingly.
\end{defn}

When $\rho_1$ is irreducible, we deduce directly from Mackey's theory that $\rho_1$ is (weakly) intertwined with $\rho_2$ in $\rA$ if and only if there exists an element $x\in\rA$, such that $x$ (weakly) intertwines $\rho_1$ with $\rho_2$.

\begin{prop}
\label{prop 0.3}
\label{Lprop 0.3}
For $i= 1,2$, let $K_i$ be a compact open subgroup of $\rM$, and $\rho_i$ an irreducible representation of $K_i$, and $\rho_i'$ an irreducible component of $\res_{K_i'}^{K_i}\rho_i$. Let $x\in\rM$ that weakly intertwines $\rho_1'$ with $\rho_2'$. Then there exists a $k$-quasicharacter $\chi$ of $F^{\times}$ such that $x$ weakly intertwines $\rho_1$ with $\rho_2\otimes\chi\circ\mathrm{det}$.
\end{prop}

\begin{proof}
By Mackey's theory, $i_{K_1',K_2'}x(\rho_2')$ is a subrepresentation of $i_{K_1,K_2}x(\rho_2)$. Since $i_{K_1,K_2}x(\rho_2)$ has finite length, the uniqueness of Jordan-H\"older factors implies that there exists an irreducible subquotient $\rho$ of $i_{K_1,K_2}x(\rho_2)$, whose restriction to $K_1'$ contains $\rho_1'$ as a direct factors. By Proposition \ref{prop 0.2}, $\rho$ is isomorphic to $\rho_1\otimes\chi^{-1}\circ\mathrm{det}$, which implies that $\rho_1$ is weakly intertwined with $\rho_2\otimes\chi\circ\mathrm{det}$ by $x$.
\end{proof}

Now we consider the maximal simple $k$-types of $\rG=\mathrm{GL}_n(F)$.
\begin{prop}
\label{prop 1}
Let $(J,\lambda)$ be a maximal simple $k$-type of $\rG$, and $\chi$ a $k$-quasicharacter of $F^{\times}$. If $(J,\lambda\otimes\chi\circ\det)$ is weakly intertwined with $(J,\lambda)$, then they are intertwined, and conjugate under $\rU(\fA)$, which is to say that there exists an element $x\in\rU(\fA)$ such that $x(J)=J$ and $x(\lambda)\cong\lambda\otimes\chi\circ\det$.
\end{prop}

\begin{proof}
Let $\theta_1$ and $\theta_2$ be the simple character contained in $\lambda$ and $\lambda\otimes\chi\circ\det$ respectively, and let $\eta_i$ be the Heisenberg representation of $\theta_i$ for $i=1,2$. There is a surjection from $\res_{H^1}^J\lambda$ to $\theta_1$ (see Section \ref{Notation} for $K$). By Frobenius reciprocity, there is an injection from $\lambda$ to $\ind_{H^1}^J\theta_1$, and hence an injection: $\res_{H^1}^G\ind_J^G\lambda\hookrightarrow \res_{H^1}^G\ind_{H^1}^G\theta_1$. By the assumption, $\res_{H^1}^J\lambda\otimes\chi\circ\mathrm{det}$ is a subquotient of $\res_{H^1}^G\ind_{H^1}^G\theta_1$. From Corollary \ref{cor 3} we have $\mathrm{H}^1(\beta+c)=\mathrm{H}^1(\beta)$. Hence $\res_{\mathrm{H}^1}^{\rG}(\lambda\otimes\chi\circ\mathrm{det})$ is a multiple of $\theta_2$, from which we deduce that $\theta_2$ is a subquotient of $\res_{\mathrm{H}^1}^{\rG}\ind_{\mathrm{H}^1}^{\rG}\theta_1$.

Notice that $\mathrm{H}^1$ is a prop-$p$ group, and any smooth representation of $\mathrm{H}^1$ is semisimple. It follows that $\theta_2$ is a sub-representation of $\res_{\mathrm{H}^1}^{\rG}\ind_{\mathrm{H}^1}^{\rG}\theta_1$, which is equivalent to say that $\theta_2$ is intertwined with $\theta_1$ in $\rG$. The same property holds for the $\cK$-lifting of $\theta_i$ for $i=1,2$, which implies that the nonsplit fundamental strata $[\fA,n,n-1,\beta]$ and the nonsplit fundamental stratum $[\fA,m,m-1,\beta+c]$ are intertwined. We deduce that $n=m$ from \cite[2.3.4, 2.6.3]{BuKu}. Then we apply Theorem 3.5.11 of \cite{BuKu}: there exists $x\in\mathrm{U}(\fA)$ such that $x(\mathrm{H}^1)=\rH^1$, $\nC(\fA,0,\beta)=\nC(\fA,0,x(\beta+c))$ and $x(\theta_{2})=\theta_{1}$, hence $x(\theta_2)=\theta_1$. In particular, $x(J)$ is a subset of $\mathcal{I}_{\mathrm{U}(\fA)}(\theta_1)$. Meanwhile, \cite[2.3.3]{MS} and \cite[3.1.15]{BuKu} implies that $\mathcal{I}_{\rG}(\theta_1)\cap\mathrm{U}(\fA)=J$, then $x(J)=J$. Proposition 2.2 of \cite{MS} shows the uniqueness of $\eta_1$, hence $x(\eta_2)\cong\eta_1$. From  \cite[Corollary 8.4]{V3} we know that the $\eta_1$-isotypic part of $\res_J^{\rG}\ind_J^{\rG}\lambda$ can be viewed as a representation of $J$, which is a direct factor of $\res_J^{\rG}\ind_J^{\rG}\lambda$ and is a multiple of $\lambda$. Since $x(\lambda\otimes\chi\circ\det)$ is a subquotient of the $\eta_1$-isotypic part of $\res_J^{\rG}\ind_J^{\rG}(\lambda)$, we conclude that $x(\lambda\otimes\chi\circ\det)\cong\lambda$ and $\Hom_{kJ}(\lambda\otimes\chi\circ\det, \res_J^{\rG}\ind_J^{\rG}\lambda)\neq0$.
\end{proof}

\begin{cor}
\label{cor 5}
Let $g\in\rG$, if $g$ weakly intertwines $(J,\lambda\otimes\chi\circ\det)$ with $(J,\lambda)$, then $g$ intertwines $(J,\lambda\otimes\chi\circ\det)$ with $(J,\lambda)$.
\end{cor}

\begin{proof}
Recall that $\eta_1$ and $\eta_2$ is contained in $\lambda$ and $\lambda\otimes\chi\circ\det$ respectively. By Proposition \ref{prop 1}, there exists $x\in\mathrm{U}(\fA)$ such that $x(\eta_1)\cong\eta_2$, $x(\lambda)\cong\lambda\otimes\chi\circ\det$. By the assumption $\lambda\otimes\chi\circ\det$ is a subquotient of $i_{J,g(J)}g(\lambda)$, hence a subquotient of $(i_{J,g(J)}g(\lambda))^{\eta_2=x(\eta_1)}$.  The latter is a sub-representation of $(\res_J^{\rG}\ind_J^{\rG}\lambda)^{\eta_2=x(\eta_1)}$, hence a multiple of $x(\lambda)$. Then $\lambda\otimes\chi\circ\det^{x(\eta_1)}$ is a sub-representation of $(i_{J,g(J)}g(\lambda))^{x(\eta)}$, and we finish the proof.
\end{proof}

\subsubsection{Decomposition of $\res_{J}^{\rG}\ind_{J}^{\rG}\lambda$}
\label{section 02.2.3}
In this section, the main purpose is to obtain the decomposition in Theorem \ref{thm 7}, which plays a key role in the proof of Proposition \ref{prop 14}. The latter consists half of the proof of Theorem \ref{thm 16}.

\begin{thm}
\label{thm 7}
Let $(J,\lambda)$ be a maximal simple $k$-type of $\rG$. There exists an integer $m$ such that
$$\res_J^{\rG}\ind_J^{\rG}\lambda\cong(\oplus_{i=1}^{m}\ x_i(\Lambda(\lambda)))\oplus W,$$
where $x_i\in\mathrm{U}(\fA)$, and put $x_1=1$. The representation $\Lambda(\lambda)$ is a multiple of $\lambda$. For each $x_i$, the representation $x_i(\Lambda(\lambda))$ is the $x_i$-conjugation of $\Lambda(\lambda)$. The elements $x_i$'s verify that $x_i(\eta)\ncong x_j(\eta)$ for $i\neq j$ (see Section \ref{Notation} for the definition of $\eta$). Furthermore, let $\lambda'$ be an irreducible sub-representation of $\res_{J'}^{J}\lambda$, then $\lambda'$ is not equivalent to any irreducible subquotient of $\res_{J'}^{J}W$.
\end{thm}

\begin{rem}
From now on, we always denote $\oplus_{i=1}^{m}\ x_i(\Lambda(\lambda))$ by $\Lambda_{\lambda}$, and we write the decomposition in Theorem \ref{thm 7} above as:
$$\res_J^{\rG}\ind_J^{\rG}\lambda\cong\Lambda_{\lambda}\oplus W.$$
\end{rem}

Before prove Theorem \ref{thm 7}, we first introduce the following very useful lemmas:
\begin{lem}
\label{lem 8}
Let $K_1,K_2$ be two compact open subgroups of $\rG$ such that $K_1\subset K_2$. Then the compact induction $\ind_{K_1}^{K_2}$ respects infinite direct sum.
\end{lem}

\begin{proof}
This can be checked directly from the definition of the compact induction functor.
\end{proof}

\begin{lem}
\label{lem 9}
Let $K$ be a compact open subgroup of $\rM$, and $K'=K\cap\rG'$. Let $\pi$ be a $k$-representation of $K$. If $\tau'$ is an irreducible subquotient of the restriction $\res_{K'}^K\pi$, then there exists an irreducible subquotient $\tau$ of $\pi$, such that $\tau'$ is an irreducible direct factor of $\res_{K'}^{K}\tau$.
\end{lem}

\begin{proof}
Let $H$ be a pro-$p$ open compact subgroup of $K$. The representation $\res_H^K \pi$ is semisimple, which can be written as $\oplus_{i\in I}\pi_i$, where $I$ is an index set. There is an injection from $\pi$ to $\ind_H^K\res_H^K \pi$. Lemma \ref{lem 8} implies that $\ind_H^K\res_H^K \pi\cong\oplus_{i\in I}\ind_H^K \pi_i$. Let $W'$, $V'$ be two sub-representations of $\pi'=\res_{K'}^{K}\pi$, such that $\tau'\cong W'/V'$. When $\tau'$ is non-trivial, there exists $x\in W'$ such that $x\notin V'$. Then there exists a finite subset $I' \subset I$, such that $x\in \oplus_{i\in I'} \ind_H^K\pi_i$. We have:
$$(W'\cap \oplus_{i\in I'} \ind_H^K\pi_i) / (V' \cap \oplus_{i\in I'} \ind_H^K\pi_i )\cong W'/V'\cong \tau'.$$
Since the restriction $\res_{K'}^K \oplus_{i\in I'} \ind_H^K\pi_i$ has finite length, by the uniqueness of Jordan-H\"older factors, there exists an irreducible subquotient of $\oplus_{i\in I'} \ind_H^K\pi_i$, whose restriction to $K'$ is semisimple (by Proposition \ref{prop 0.1}) and contains $\tau'$ as a subrepresentation. 
\end{proof}

Now we are ready to prove Theorem \ref{thm 7}.
\begin{proof} of Theorem \ref{thm 7}:

By \cite[Corollary 8.4]{V3}, we can decompose $\res_J^{\rG}\ind_J^{\rG}\lambda\cong \Lambda(\lambda)\oplus W_1$, where an irreducible subquotient of $W_1$ is not isomorphic to $\lambda$. Let $\lambda'$ be an irreducible direct component of $\res_{J'}^{J}\lambda$. If $\lambda'$ is an irreducible subquotient of $\res_{J'}^{J}W_1$, by Lemma \ref{lem 9} and Propositon \ref{prop 0.2}, there exists a $k$-quasicharacter $\chi$ of $F^{\times}$ such that $\lambda\otimes\chi\circ\det$ is an irreducible subquotient of $W_1$, which implies that $\lambda\otimes\chi\circ\det$ and $\lambda$ are weakly intertwined in $\rG$. By Proposition \ref{prop 1}, $\lambda$ is conjugate to $\lambda\otimes\chi\circ\det$ by an element $x\in\mathrm{U}(\fA)$. The fact that $\lambda\otimes\chi\circ\det$ is a subquotient of $W_1$ implies that  $x(\eta)\ncong\eta$. Thus we can decompose $W_1$ as $W_1^{x(\eta)}\oplus W_2$, and we have:
$$W_1^{x(\eta)}\cong x((\res_J^{\rG}\ind_J^{\rG}\lambda)^{\eta}).$$
The latter is isomorphic to $x(\Lambda(\lambda))$, which is a direct sum of $x(\lambda)$. Now we obtain an isomorphism:
$$\res_J^{\rG}\ind_J^{\rG}\lambda\cong\Lambda(\lambda)\oplus x(\Lambda(\lambda))\oplus W_2,$$
where $W_2^{x(\eta)}=0$ and $W_2^{\eta}=0$. It implies an irreducible subquotient of $W_2$ is not isomorphic to $\lambda$ neither to $x(\lambda)$. Now we can repeat the above steps to $W_2$. Notice that, any irreducible representation of $J$, whose restriction to $J'$ contains $\lambda'$ as a subrepresentation, is $\mathrm{U}(\fA)$-conjugate to $\lambda$. The quotient $\rU(\fA)\backslash J$ is finite, hence the set of irreducible representations $\{ x(\lambda) \}_{x\in\rU(\fA)}$ is finite, which means after repeating finite times, we obtain the required decomposition.
\end{proof}

\subsubsection{Projective normalizer $\tilde{J}$ and its subgroups}
\label{section 02.2.4}

In this section, we apply the same the definition of projective normalizer given in \cite{BuKuII} for $\cK$-representations of $\rG'$, and we show that some properties in \cite{BuKuII} still hold true in $\ell$-modular setting.
\begin{defn}[Bushnell,Kutzko]
\label{defn 11}
Let $\fA$ be the principal order attached to a maximal simple $k$-type $(J,\lambda)$. Then the projective normalizer $\tilde{J}=\tilde{J}(\lambda)$ of $(J,\lambda)$ is defined to be the group of all $x\in \mathrm{U}(\fA)$ such that:
\begin{itemize}
\item $xJx^{-1}=J$, and
\item there exists a $k$-quasicharacter $\chi$ of $F^{\ast}$ such that $x(\lambda)\cong\lambda\otimes\chi\circ\det$.
\end{itemize}
\end{defn}

\begin{prop}
\label{prop 12}
Let $(J,\lambda)$ be a maximal simple $k$-type of $\rG$, and $\chi$ a $k$-quasicharacter of $F^{\ast}$. The following are equivalent:
\begin{enumerate}
\item $\lambda\cong\lambda\otimes\chi\circ\det,$
\item $\chi\circ\det\vert_{J^1}$ is trivial and $\sigma\otimes\chi\circ\det\vert_{\mathrm{U}(\mathfrak{B})}\cong\sigma$,
\item $\chi\circ\det\vert_{J^1}$ is trivial, and $\lambda$, $\lambda\otimes\chi\circ\det$ are intertwined in $\rG$.
\end{enumerate}
\end{prop}

\begin{proof}
The proof follows the same strategy as \cite[Proposition 2.3]{BuKuII}. We prove this proposition in the order of $2\rightarrow 1\rightarrow 3 \rightarrow 2$. 

The implication $2\rightarrow 1$ is immediate by $J/ J'\cong \rU(\mathfrak{B}/\rU^1(\mathfrak{B}))$. For $1\rightarrow 3$,  we assume that $\lambda$ is equivalent to $\lambda\otimes\chi\circ\det$. Restricting to $H^1$, we have the equivalent of the simple characters $\theta\cong\theta\otimes\chi\circ\det$, which implies $\chi\circ\det\vert_{H^1}$ is trivial. For $3\rightarrow 2$, we assume part $3$ holds. Proposition \ref{prop 1} gives an element $x\in \mathrm{U}(\fA)$ such that $x(J)=J$ and $x(\lambda)=\lambda\otimes\chi\circ\det$. Hence the uniqueness of $\eta$ (see Proposition 2.2 of \cite{MS}) implies that $x(\eta)\cong\eta$. In particular, $x\in \mathcal{I}_{\rG}(\theta)=J^1 B^{\times}J^1$ by \cite[IV.1.1]{V2}, hence $x\in J^1B^{\ast}J^1\cap\rU(\fA)=J$, then $\lambda\cong x(\lambda)\cong\lambda\otimes\chi\circ\det$. In other words, $\kappa\otimes\sigma\cong\kappa\otimes\sigma\otimes\chi\circ\det$. By Schur Lemma, we can apply the proof in \cite[Proposition 5.3.2]{BuKu}: 

Let $X$ be the representation space of $\kappa$ and $Y$ the representation space of $\sigma$, which can be identified with that of $\sigma\otimes\chi\circ\det$. Let $\phi$ be the isomorphism between $\kappa\otimes\sigma$ and $\kappa\otimes\sigma\otimes\chi\circ\det$. We may write $\phi$ as $\sum_{j}S_j\otimes T_j$ where $S_j\in\mathrm{End}_k(X)$ and $T_j\in\mathrm{End}_k(Y)$, and where $\{ T_j \}$ are linearly independent. Let $g\in J^1$, we have $\kappa\otimes\sigma(g)\circ\phi =\phi\circ(\kappa\otimes\sigma\otimes\chi\circ\det)(g)$. Since $J^1\subset \mathrm{ker}(\sigma)= \mathrm{ker}(\sigma\otimes\chi\circ\det)$, this relation reads:
$$(\eta(g)\otimes 1)\circ \sum_{j}S_j\otimes T_j=(\sum_j S_j\otimes T_j)\circ(\eta(g)\otimes1),$$
which is equivalent to say that:
$$\sum_j(\eta(g)\circ S_j-S_j\circ\eta(g))\otimes T_j=0.$$
The linearly independence of $T_j$ implies that $S_j\in\mathrm{End}_{kJ^1}(\eta)=k^{\ast}$, by the Schur lemma. Hence $\phi=\mathrm{1} \otimes \sum_{j}S_j\cdot T_j$. Now note $T= \sum_{j}S_j\cdot T_j$ and take $g\in J$, the morphism relation reads:
$$(\kappa(g)\otimes \sigma(g))\circ(1\otimes  T)=\kappa(g)\otimes(\sigma(g)\circ T)=\kappa(g)\otimes (T\circ\sigma\otimes\chi\circ\det(g))$$
$$=(1\otimes T)\circ (\kappa(g)\otimes\sigma\otimes\chi\circ\det(g)),$$
which says $T\in\mathrm{Hom}_{kJ}(\sigma,\sigma\otimes\chi\circ\det)\neq 0$. We finish the proof.
\end{proof}

\begin{cor}
\label{cor 13}
Let $x\in\tilde{J}(\lambda)$, and $\chi$ a $k$-quasicharacter of $F^{\ast}$ such that $x(\lambda)\cong\lambda\otimes\chi\circ\det$. Then:
\begin{enumerate}
\item the map $x \mapsto \chi\circ\det\vert_{J^1}$ is an injective homomorphism from $\tilde{J}/J\rightarrow (\det(J^1))^{\wedge}$. The latter is the $k$-dual group of $\det(J^1)$;
\item $\tilde{J}/ J$ is a finite abelian $p$-group, where $p$ is the residual characteristic of $F$.
\item The exponent of $\tilde{J}/J$ divides the rank $n$ of $\rG'$, where $\rG'=\mathrm{SL}_n(F)$.
\end{enumerate}
\end{cor}

\begin{proof}
For part $1$, let $x\in\tilde{J}$. Suppose that there exist two $k$-quasicharacters $\chi_1,\chi_2$ of $F^{\ast}$, such that $x(\lambda)\cong \lambda\otimes\chi_1\circ\det$ and $\lambda\otimes\chi_1\circ\det\cong\lambda\otimes\chi_2\circ\det$. This is equivalent to say that
$$\lambda\cong\lambda\otimes(\chi_1\otimes\chi_2^{-1})\circ\det.$$
Proposition \ref{prop 12} implies that $\chi_1\circ\det\vert_{J^1}\cong\chi_2\circ\det\vert_{J^1}$, which implies that the group morphism is well defined. Now suppose that $x\in\tilde{J}$ and $\chi$ is trivial on $\det(J^1)$, such that $x(\lambda)\cong\lambda\otimes\chi\circ\det$. Then by Proposition \ref{prop 12} $\lambda\cong\lambda\otimes\chi\circ\det$, and $x$ intertwines $\lambda$ to itself. Hence $x$ belongs to $J B^{\times} J\cap\rU(\fA)=J$.

Since $J^1$ is a pro-$p$ group, part $2$ is induced from part $1$.

For part $3$, let $Z'$ be the centre of $\rG'$. Since $\rU(\mathfrak{o}_F)^n\subset \mathrm{det}(J^1\cap Z')$, hence $(\chi\vert_{J^1})^n$ is trivial, which implies that $x^n\in J$ for $x\in\tilde{J}$.
\end{proof}

\begin{rem}
Part $2$ and Part $3$ of Corollary \ref{cor 13} imply that $\tilde{J}=J$, when $p$ does not divide $n$.
\end{rem}

\subsubsection{Two conditions for irreducibility}
\label{section 07}
In this section, let $(J,\lambda)$ be a maximal simple $k$-type of $\rG$. We construct a compact subgroup $M_{\lambda}$ of $\rG'$ and a family of irreducible representations $\lambda_{M_{\lambda}}'$ of $M_{\lambda}$, such that the induction $\ind_{M_{\lambda}}^{\rG'}\lambda_{M_{\lambda}}'$ is irreducible and cuspidal (Theorem \ref{thm 16}). In the next section, we will see that every irreducible cuspidal $k$-representation of $\rG'$ can be constructed in this manner.

When $\ell=0$, to show $\ind_{M_{\lambda}}^{\rG'}\lambda_{M_{\lambda}}'$ is irreducible, it is sufficient to show that $\mathrm{I}_{\rG'}(\lambda_{M_{\lambda}}')=M_{\lambda}$. However, this intertwining condition is not sufficient in the $\ell$-modular setting, and a criteria of irreducibility has been established in \cite[Lemma 4.2]{V3}. For the reason of convenience, we present it here:

\begin{lem} [Criteria of irreducibility]
\label{lem 19}
Let $K'$ be an open compact subgroup of $\rG'$, and $\pi'$ be an irreducible $k$-representation of $K'$. The induction $\ind_K'^{\rG'}\pi'$ is irreducible, when
\begin{enumerate}
\item $\mathrm{End}_{k\rG'}(\ind_{K'}^{\rG'}\pi')=k$, and
\item for any irreducible $k$-representation $\nu$ of $\rG'$, if $\pi'$ is contained in $\res_{K'}^{\rG'}\nu$ then there is a surjection from $\res_{K'}^{\rG'}\nu$ to $\pi'$.
\end{enumerate}
\end{lem}

As shown in \cite[\S I,8.3]{V1}, the first condition is equivalent to $\mathrm{I}_{\rG'}(\pi')=K'$, where $\mathrm{I}_{\rG'}(\pi')$ is the intertwining set. 

\begin{cor}
\label{cor 99}
Let $(J,\lambda)$ be a maximal simple $k$-type of $\rG$. The induction $\ind_{J}^{\tilde{J}}\lambda$ is irreducible.
\end{cor}

\begin{proof}
Lemma \ref{lem 19} can be applied by changing $\rG'$ to a locally pro-finite group. First, we compute $\mathrm{End}_{k\tilde{J}}(\ind_{J}^{\tilde{J}}\lambda)$, which is equal to $k$ since the intertwining group $\mathrm{I}_{\tilde{J}}(\lambda)=J$. For the second condition, let $\nu$ be an irreducible $k$-representation of $\tilde{J}$, such that 
$$\lambda\hookrightarrow\res_{J}^{\tilde{J}}\nu.$$
By \cite[\S I, 6.12]{V1} the restriction $\res_{J}^{\tilde{J}}\nu$ is semisimple, hence $\lambda$ is a direct component.
\end{proof}

\begin{thm}
\label{thm 10}
Let $\lambda'$ be a subrepresentation of $\res_{J'}^J\lambda$. Then $\lambda'$ verifies the second condition of irreducibility: For any irreducible $k$-representation $\pi'$ of $\rG'$, if there is an injection: $\lambda'\hookrightarrow \res_{J'}^{\rG'}\pi'$, then there is a surjection: $\res_{J'}^{\rG'}\pi'\twoheadrightarrow \lambda'.$
\end{thm}

\begin{proof}
By Mackey's theory, we have
$$\res_{\rG'}^{\rG}\ind_J^{\rG}\lambda \cong \oplus_{a\in J\backslash \rG/\rG'} \ind_{\rG'\cap a(J)}^{\rG'}\res_{\rG'\cap a(J)}^{a(J)}a(\lambda),$$
of which $\ind_{J'}^{\rG'}\res_{J'}^J\lambda$ is a direct factor. The assumption $\lambda'\hookrightarrow\res_{J'}^{\rG'}\pi'$ implies a surjection from $\ind_{J'}^{\rG'}\lambda'$ to $\pi'$ by Frobenius reciprocity. Since $\res_{J'}^J\lambda$ is semisimple, we have a surjection
$$\ind_{J'}^{\rG'}\res_{J'}^J\lambda\twoheadrightarrow \pi'.$$
Hence we obtain
$$\res_{\rG'}^{\rG}\ind_J^{\rG}\lambda\twoheadrightarrow\pi'.$$
Now consider the surjection
$$\iota:\res_{J'}^{\rG}\ind_{J}^{\rG}\lambda\twoheadrightarrow \res_{J'}^{\rG'}\pi'.$$
By Theorem \ref{thm 7}, we decompose $\res_J^{\rG}\ind_J^{\rG}\lambda\cong \Lambda_{\lambda}\oplus W$. We have $\Lambda_{\lambda}\oplus W/\mathrm{ker}(\iota)\cong\res_{J'}^{\rG'}\pi'$. 
By the definition of $W$, the image of the composed morphism: 
$$\lambda'\hookrightarrow \Lambda_{\lambda}\oplus W/\mathrm{ker}(\iota)\twoheadrightarrow \Lambda_{\lambda}\oplus W/(W+\mathrm{ker}(\iota))\cong\Lambda_{\lambda}/(\Lambda_{\lambda}\cap(W+\mathrm{ker}(\iota)))$$
is non-trivial. Since $\res_{J'}^J\Lambda_{\lambda}$ is semisimple, so is the quotient $\Lambda_{\lambda}/ (\Lambda_{\lambda}+\mathrm{ker}(\iota))$, of which $\lambda'$ is an irreducible direct factor. This implies a surjection: $\res_{J'}^{\rG'}\pi'\twoheadrightarrow \lambda'$.
\end{proof}

In the theorem above, we show that $\lambda'$ verifies the second condition of irreducibility in Lemma \ref{lem 19}. Unfortunately, $(J',\lambda')$ does not satisfy the first condition, neither when $\ell=0$ nor when $\ell\neq 0$. A natural idea is to construct an open compact subgroup bigger than $J'$, and to extend $\lambda'$ to this group. The aim of the rest of this section is to find such a group. The first step is to introduce $M_{\lambda}$ (see Definition \ref{defn 15}), of which a maximal simple $k$-types is defined. Then we prove that $M_{\lambda}=\tilde{J}'=\tilde{J}\cap\rG'$ (see Proposition \ref{prop 24} and Definition \ref{defn 22}).

\begin{lem}
\label{prop 17}
Let $L$ be a subgroup of $\rG'$ such that $J'\subset L\subset \tilde{J}'$, and $\lambda'$ an irreducible subrepresentation of $\lambda\vert_{J'}$. Then the induced representation $\ind_{J'}^{L}\lambda'$ is semisimple.
\end{lem}

\begin{proof}
By Mackey's theory, the induced representation $\ind_{J'}^{\tilde{J}'}\lambda'$ is a subrepresentation of $\res_{\tilde{J}'}^{\tilde{J}}\ind_J^{\tilde{J}}\lambda$, which is semisimple by Proposition \ref{prop 0.1} and Corollary \ref{cor 99}. Then by Clifford theory, $\res_{L}^{\tilde{J}}\ind_J^{\tilde{J}}\lambda$ is semisimple, of which $\ind_{J'}^{L}\lambda'$ is a subrepresentation, hence is semisimple.
\end{proof}

\begin{prop}
\label{prop 14}
Let $\lambda'$ be an irreducible subrepresentation of $\res_{J'}^J\lambda$, and $\lambda_L'$ an irreducible subrepresentation of $\ind_{J'}^{L}\lambda'$. Then $\lambda_L'$ verifies the second condition of irreducibility. This is to say that for any irreducible representation $\pi'$ of $\rG'$, if there is an injection $\lambda_L' \hookrightarrow \res_{L'}^{\rG'}\pi'$, then there exists a surjection $\res_{L'}^{\rG'}\pi'\twoheadrightarrow \lambda_{L}'$, where $L'=L\cap\rG'$.
\end{prop}

\begin{proof}
By Lemma \ref{prop 17}, the injection from $\lambda_L'$ to $\res_L^{\rG'}\pi'$ induces a non-trivial homomorphism $\ind_{J'}^L \lambda'\twoheadrightarrow \res_L^{\rG'}\pi'$. By Frobenius reciprocity, there is an injection from $\lambda'$ to $\res_{J'}^{\rG'}\pi'$. Thus there exists a non-trivial homomorphism $\res_{J'}^{J}\lambda\rightarrow\res_{J'}^{\rG'}\pi'$. By applying Frobenius reciprocity, we obtain a surjection:
$$\res_{J'}^{\rG'}\ind_{J'}^{\rG'}\res_{J'}^{J}\lambda\twoheadrightarrow\res_{J'}^{\rG'}\pi'.$$
By Mackey's theory, the $k$-representation $\ind_{J'}^{\rG'}\res_{J'}^{J}\lambda$ is a direct factor of $\res_{\rG'}^{\rG}\ind_{J}^{\rG}\lambda$, and hence $\res_{J'}^{\rG'}\ind_{J'}^{\rG'}\res_{J'}^{J}\lambda$ is a direct factor of $\res_{J'}^{\rG}\ind_{J}^{\rG}\lambda$. Then the surjection above implies a non-trivial homomorphism:
$$\res_{J'}^{\rG}\ind_{J}^{\rG}\lambda\twoheadrightarrow \res_{J'}^{\rG'}\pi'.$$
By Proposition \ref{thm 7}, the left hand side is isomorphic to $\res_{J'}^J\Lambda_{\lambda}\oplus\res_{J'}^J W$, where the set of isomorphism class of irreducible subquotients of $\res_{J'}^J W$ is disjoint with that $\res_{J'}^J\lambda$. Now we consider the equivalence $\res_{J'}^{\rG'}\pi'\cong(\res_{J'}^J\Lambda_{\lambda}\oplus\res_{J'}^J W )/ U$, where $U$ is the kernel of the above surjection.

Define $\tau$ be the composed morphism as below:
$$\tau:=\lambda_L'\hookrightarrow\res_{L}^{\rG'}\pi'\hookrightarrow\ind_{J'}^L\res_{J'}^{\rG'}\pi',$$
and the last factor is isomorphic to $\ind_{J'}^{L}((\res_{J'}^{J}\Lambda_{\lambda}\oplus\res_{J'}^J W)/ U)$.

Now we prove that $\ind_{J'}^L((\res_{J'}^J W +U)/ U)$ is contained in the kernel of $\tau$. Suppose to the contrary, $\tau$ induces a non-trivial morphism from $\lambda_L'$ to $\ind_{J'}^L((\res_{J'}^J W +U)/ U)$. By Frobenius reciprocity, there is a non-trivial morphism from $\res_{J'}^L\lambda_L'$ to $(\res_{J'}^J W +U)/ U$. 
Meanwhile, by the definition of $\tilde{J}$ we have
$$\res_{J'}^L \lambda_L'\hookrightarrow\res_{J'}^{\tilde{J}}\ind_J^{\tilde{J}}\lambda\cong\oplus_{J\backslash\tilde{J}/ J'}\res_{J'}^J \lambda.$$
Thus there exists an irreducible component of $\res_{J'}^J\lambda$ appearing as a subquotient of $\res_{J'}^{J}W$, which contradicts to Theorem \ref{thm 7}. So $\tau(\lambda_L')$ is contained in $\ind_{J'}^L((\res_{J'}^{J}\Lambda_{\lambda}+U)/U)$, which is a quotient $\ind_{J'}^{L}\res_{J'}^{J}\Lambda_{\lambda}$, and the latter is semisimple by Lemma \ref{prop 17} and Lemma \ref{lem 8}. Hence we conclude that $\ind_{J'}^L((\res_{J'}^{J}\Lambda_{\lambda}+U)/U)$ is semisimple, contains $\lambda_{L}'$ as a direct component. Since it is a quotient of $\res_{L}^{\rG'}\pi'$, we finish the proof.
\end{proof}

\begin{defn}
\label{defn 15}
Let $(J,\lambda)$ be a maximal simple $k$-type of $\rG$, and $\lambda'$ be an irreducible subrepresentation of $\res_{J'}^J\lambda$. Define $M_{\lambda}$ to be the subgroup of $\tilde{J}'$ consisting of the elements $x\in\tilde{J}'$, such that $x(\lambda')\cong\lambda'$.
\end{defn}

\begin{rem}
Since $M_{\lambda}$ normalizes $J$ and the intersection $M_{\lambda}\cdot J\cap \rG'$ is equal to $M_{\lambda}$, we have that $J$ normalizes $M_{\lambda}$. On the other hand, since the irreducible components of $\res_{J'}^J\lambda$ are $J$-conjugate, the group $M_{\lambda}$ is independent on the choice of $\lambda'$. We will prove in Proposition \ref{prop 24} that $M_{\lambda}=\tilde{J}'$.
\end{rem}

In Theorem \ref{thm 16}, we prove that a pair $(M_{\lambda},\lambda_{M_{\lambda}}')$ (we keep the notation in Proposition \ref{prop 14}) verifies the criteria in Lemma \ref{lem 19} and hence a maximal simple $k$-type. The first criterion has been checked in Proposition \ref{prop 14}. Now we compute its intertwining group in $\rG'$: We first show that $\mathrm{I}_{\rG'}(\ind_{J'}^{\rU(\fA)'}\lambda')\subset\rU(\fA)'$ (Proposition \ref{prop 18}), and then show that $\mathrm{I}_{\rU(\fA)'}\lambda_{M_{\lambda}}'=M_{\lambda}$ (Theorem \ref{thm 16}).

\begin{lem}
\label{lem plus 1}
The induction $\ind_{J}^{\rU(\fA)}\lambda$ is irreducible.
\end{lem}

\begin{proof}
Denote $\ind_{J}^{\rU(\fA)}\lambda$ by $\rho$. Let $\rho'$ be an irreducible subrepresentation of $\rho$, and $E$ the field extension associated with $(J,\lambda)$ as in Section \ref{Notation}. Let $\Lambda$ be an extension of $\lambda$ to $E^{\times}J$, then $\ind_{E^{\times}J}^{\rG}\Lambda$ is irreducible. Let $\omega$ be the central character of $\Lambda$, and denote by $\rho_{\omega}'$ the extension of $\rho$ to $F^{\times}\rU(\fA)$ by $\omega$. Since $E^{\times}\rU(\fA)\slash F^{\times}\rU(\fA)$ is a finite cyclic group, $\tau_{\omega}'$ can be extended to $E^{\times}\rU(\fA)$. By Frobenius reciprocity and Mackey's theory, one of such extension $\bar{\rho}_{\omega}'$ is embedded in $\ind_{E^{\times}J}^{E^{\times}\rU(\fA)}\Lambda$, which is irreducible. Hence this embedding is an equivalence. After restricting to $\rU(\fA)$, it gives an equivalence from $\rho'$ to $\rho$, and we finish the proof.
\end{proof}

\begin{prop}
\label{prop 18}
Let $\lambda'$ be an irreducible subrepresentation of $\res_{J'}^J\lambda$, then the intertwining set $\mathrm{I}_{\rG'}(\ind_{J'}^{\rU(\fA)'}\lambda')$ is contained in $\rU(\fA)'$.
\end{prop}

\begin{proof}
Let $\rho$ be as above. The induction $\ind_{J'}^{\rU(\fA)'}\lambda'$ is embedding to $\res_{\rU(\fA)'}^{\rU(\fA)}\rho$, thus is semisimple of finite length, and write $\ind_{J'}^{\rU(\fA)'}\lambda'\cong\oplus_{i\in I}\rho_i'$, where $\rho_i'$ are irreducible and $I$ is a finite. Let $g\in\rG$, we have
$$\mathrm{H}_g(\ind_{J'}^{\rU(\fA)'}\lambda')=\bigoplus_{i\in I,j\in I}\mathrm{Hom}(\rho_i',i_{\rU(\fA)'}g(\rho_j')).$$
Hence we have:
$$\mathrm{H}_{\rG'}(\ind_{J'}^{\rU(\fA)'}\lambda')= \bigoplus_{i\in I,j\in I}\mathrm{H}_{\rG'}(\rho_i',\rho_j').$$ 
Now we assume that $g\in\mathrm{I}_{\rG'}(\rho_i',\rho_j')$, and we show that $g\in\rU(\fA)'$. 

By Proposition \ref{prop 0.3}, there exists a $k$-quasicharacter $\chi$ of $F^{\times}$ such that $g$ weakly intertwines $\rho$ with $\rho\otimes\chi\circ\det$.  Hence $\res_{J}^{\rU(\fA)}\rho$ is a subquotient of $\res_{J}^{\rU(\fA)}i_{\rU(\fA),g(\rU(\fA))}\rho\otimes\chi\circ\det$, and so is $\lambda$. Applying Mackey's theory to $\res_{J}^{\rU(\fA)}i_{\rU(\fA),g(\rU(\fA))}\rho\otimes\chi\circ\det$, it is isomorphic to a finite direct sum, whose direct components are $\ind_{J\cap y(J)}^{J}\res_{J\cap y(J)}^{y(J)}y(\lambda\otimes\chi\circ\det)$ for $y\in \rU(\fA) g\rU(\fA)$. More precisely,
$$\res_{J}^{\rU(\fA)}i_{\rU(\fA),g(\rU(\fA))}(\rho\otimes\chi\circ\det)=\bigoplus_{\beta}\bigoplus_{\alpha}\ind_{J\cap y (J)}^{J}\res_{J\cap y (J)}^{\beta(\rU(\fA))\cap y (J)}y (\lambda)\otimes\chi\circ\det,$$
where $y=\beta\alpha g$, and $\beta$ (resp. $\alpha$) runs over a finite quotient of $\rU(\fA)$ (resp. $g(\rU(\fA))$).

By the uniqueness of Jordan-H\"older factors, the representation $\lambda$ is weakly intertwined with $\lambda\otimes\chi\circ\det$ by a $y_0\in \rU(\fA)g \rU(\fA)$. Hence by Corollary \ref{cor 5}, $y_0$ intertwines $\lambda$ with $\lambda\otimes\chi\circ\det$, and by Proposition \ref{prop 1} there exists $x\in\rU(\fA)$ such that $x(\lambda\otimes\chi\circ\det)\cong\lambda$. The element $y_0x^{-1}$ intertwines $\lambda$ to itself, and hence lies in $E^{\times} J$. Therefore $g\in \rU(\fA)\mathrm{E}^{\times}J \rU(\fA)\cap\rG'$. The latter is equal to $\rU(\fA)'$, since $E^{\times}$ normalises $\rU(\fA)$ and for any $e\in E^{\times}$, $\det(e)\in\mathfrak{o}_F^{\times}$ if and only $e\in\mathfrak{o}_E^{\times}$, 
where $\mathfrak{o}_F$,$\mathfrak{o}_E$ is the ring of integers of $F$,$E$ respectively. We conclude that $\mathrm{I}_{\rG'}(\ind_{J'}^{\rU(\fA)'}\lambda')=\rU(\fA)'$. 
\end{proof}

\begin{lem}
\label{lem 17}
Let $\lambda'$ be an irreducible component of $\res_{J'}^{J}\lambda$, and $x\in\mathrm{I}_{\rU(\fA)'}(\lambda')$ (see Definition \ref{prepdefn 01}), then $x\in\tilde{J}'$.
\end{lem}

\begin{proof}
For $x\in\mathrm{I}_{\rU(\fA)'}(\lambda')$, by Proposition \ref{prop 0.3} the element $x$ weakly intertwines $\lambda$ with $\lambda\otimes\chi\circ\det$ for a quasicharacter $\chi$. Then by Corollary \ref{cor 5}, $x$ intertwines $\lambda$ with $\lambda\otimes\chi\circ\det$, and hence by Proposition \ref{prop 1} there exists $y\in\rU(\fA)$ such that $y(J)=J$ and $y(\lambda)\cong\lambda\otimes\chi\circ\det$, which implies that $y\in\tilde{J}$ by definition. The element $xy^{-1}$ therefore intertwines $\lambda$, then $x\in E^{\times} J y\cap\rU(\fA)'$ by \cite[\S IV,1.1]{V2}. Since $E^{\times}\tilde{J}\cap \rU(\fA)=\tilde{J}$, we conclude that $x\in Jy\cap\rU(\fA)'\subset\tilde{J}'$.
\end{proof}

\begin{thm}
\label{thm 16}
Recall that $\lambda_{M_{\lambda}}'$ is an irreducible subrepresentation of $\ind_{J'}^{M_{\lambda}}\lambda'$. Then the induced representation $\ind_{M_{\lambda}}^{\rG'}\lambda_{M_{\lambda}}'$ is irreducible and cuspidal.
\end{thm}

\begin{proof}
It is equivalent to show that $(M_{\lambda},\lambda_{M_{\lambda}}')$ satisfies the two conditions in Lemma \ref{lem 19}, where the second condition has been verified in Proposition \ref{prop 14}. It is left to prove that $\mathrm{I}_{\rG'}(\lambda_{M_{\lambda}}')=M_{\lambda}$. We first show that $\mathrm{I}_{\rU(\fA)'}(\lambda_{M_{\lambda}}')=M_{\lambda}$, and then we show that $\mathrm{I}_{\rG'}(\lambda_{M_{\lambda}}')\subset\rU(\fA)'$.

By Lemma \ref{lem 17}, we have $\mathrm{I}_{\rU(\fA)'}\lambda'\subset \tilde{J'}$. Since $\tilde{J}'$ normalizes $J'$, then $x\in \mathrm{I}_{\rU(\fA)'}\lambda'$ meaning $x(\lambda')\cong\lambda'$. Hence $\mathrm{I}_{\rU(\fA)'}\lambda'\subset M_{\lambda}$. By \cite[\S I,8.10, Proposition 3]{V1}, let $g\in\rG'$ and $X$ a finite set of $\rG'$ such that $M_{\lambda}gM_{\lambda}=\cup_{x\in X}J'xJ'$, then there is an equivalence:
\begin{equation}
\label{equa 1}
\mathrm{H}_{g^{-1}}(\ind_{J'}^{M_{\lambda}} \lambda')\cong\oplus_{j\in X}\mathrm{H}_{(gj^{-1})}(\lambda'),
\end{equation}
which implies
\begin{equation}
\label{equa 2}
\mathrm{I}_{\rU(\fA)'}(\ind_{J'}^{M_{\lambda}} \lambda')=M_{\lambda},
\end{equation}
Hence $\mathrm{I}_{\rU(\fA)'}(\lambda_{M_{\lambda}}')=M_{\lambda}$ by the inclusion $\mathrm{I}_{\rU(\fA)'}(\lambda_{M_{\lambda}}')\subset\mathrm{I}_{\rU(\fA)'}(\ind_{J'}^{M_{\lambda}} \lambda')$.

We finish the proof by verifying the inclusion
$$\mathrm{I}_{\mathrm{\rG'}}(\lambda_{M_{\lambda}}')\subset \rU(\fA)'.$$
Notice that $\ind_{M_{\lambda}}^{\rU(\fA)'}\lambda_{M_{\lambda}}'$ is a subrepresentation of $\res_{\rU(\fA)'}^{\rU(\fA)}\rho$, where $\rho=\ind_{J}^{\rU(\fA)}\lambda$ (as in the proof of Proposition \ref{prop 18}). Since $\res_{\rU(\fA)'}^{\rU(\fA)}\tau$ is semisimple, we have
$$\mathrm{I}_{\rG'}(\ind_{M_{\lambda}}^{\rU(\fA)'}\lambda_{M_{\lambda}}')\subset\mathrm{I}_{\rG'}(\res_{\rU(\fA)'}^{\rU(\fA)}\rho).$$
Hence by Proposition \ref{prop 18}, we have
\begin{equation}
\label{equa 3} 
\mathrm{I}_{\rG'}(\ind_{M_{\lambda}}^{\rU(\fA)}\lambda_{M_{\lambda}}')\subset\rU(\fA)'
\end{equation}
As for equation (\ref{equa 1}), let $h\in\rG'$ and $Y$ a finite set of $\rG'$ such that $\rU(\fA)h\rU(\fA)=\cup_{y\in Y}M_{\lambda}yM_{\lambda}$, then there is an equivalence:
$$\mathrm{H}_{h^{-1}}(\ind_{M_{\lambda}}^{\rU(\fA)'} \lambda_{M_{\lambda}}')\cong\oplus_{s\in Y}\mathrm{H}_{(hs^{-1})}(\lambda_{M_{\lambda}}').$$
We conclude that
$$\mathrm{I}_{\rG'}(\lambda_{M_{\lambda}}')\subset\mathrm{I}_{\rG'}(\ind_{M_{\lambda}}^{\rU(\fA)'}\lambda_{M_{\lambda}}').$$
By equation \ref{equa 3}, we obtain the desired result.
\end{proof}

\subsubsection{Cuspidal $k$-representations of $\rG'$}
\label{section 02.2.6}
Recall that $\rM$ is a Levi subgroup of $\rG$, and $\rM'=\rM\cap\rG'$. In this section, we consider the restriction functor $\res_{\rM'}^{\rM}$, which has been studied in \cite{TA} when $\ell=0$. We show that an irreducible $k$-representation $\pi'$ of $\rM'$ is contained in $\res_{\rM'}^{\rM}\pi$ for an irreducible $\pi$ of $\rM$ (see Proposition \ref{prop 6}), which coincides with a result in \cite{TA}, and from which we deduce that any irreducible cuspidal $\pi'$ can be constructed as in Theorem \ref{thm 16} (see Corollary \ref{cor 21}).

\begin{lem}
\label{propver13.1}
Let $K$ be a locally pro-finite group, and $K'\subset K$ a closed normal subgroup of $K$ with finite index. Let $(\pi,V)$ be an irreducible $k$-representation of $K$, then the restriction $\res_{K'}^{K}\pi$ is semisimple of finite length.
\end{lem}

\begin{proof}
The proof of \cite[\S I,6.12]{V1} can be applied here. There is a condition in \cite[\S I,6.12]{V1} that $[K:K']$ is invertible in $k$, however it is not used in the proof.

The restriction $\res_{K'}^{K}\pi$ has finitely length, hence has an irreducible quotient. Let $V_0$ be the sub-representation such that $V\slash V_0$ is irreducible. Let $\{k_1,...,k_m\},m\in\mathbb{N}$ be a family of representatives of the quotient $K\slash K'$. Now we consider the kernel of the non-trivial projection from $\res_{K'}^{K}\pi$ to $\oplus_{i=1}^{m}V\slash k_iV_0$, which is $K$-stable, hence is equal to $0$. We deduce that $\res_{K'}^{K}\pi$ is a sub-representation of $\oplus_{i=1}^{m}V\slash k_iV_0$ hence is semisimple.
\end{proof}

\begin{prop}
\label{prop 6}
Let $\pi$ be an irreducible $k$-representation of $\rM$, then the restriction $\res_{\rM'}^{\rM}\pi$ is semisimple of finite length, and the irreducible direct components are $\rM$-conjugate. Conversely, let $\pi'$ be an irreducible $k$-representation of $\rM'$, then there exists an irreducible representation $\pi$ of $\rM$, such that $\pi'$ is a direct factor of $\res_{\rM'}^{\rM}\pi$.
\end{prop}

\begin{proof}
For the first part, the method of Silberger in \cite{Si} for the case when $\ell=0$ can be generalised to $\ell$-modular setting. At first we assume that $\pi$ is cuspidal. Let $Z$ be the center of $\rM$, and the quotient $\rM\slash Z\rM'$ is compact (it is finite when $char(F)=0$, but may be infinite when $char(F)>0$). Since the stabiliser $\mathrm{Stab}_{\rM}(v)$ is open for any vector $v$ in the representation space of $\pi$, the image of $\mathrm{Stab}_{\rM}(v)$ has finite index in the quotient group $\rM\slash Z\rM'$. Hence the restriction $\res_{\rM'}^{\rM}\pi$ is finitely generated. By \cite[\S II,2.7]{V1} the restriction is $Z'=Z\cap\rM'$-compact. Now we show that it is admissible.

Let $(v_1,...,v_m),m\in\mathbb{N}$ be a family of generators of the representation space of $\res_{\rM'}^{\rM}\pi$. For a compact open subgroup $K$ of $\rM'$, which stabilises $v_i,i=1,...,m$, we consider the maps 
$$\alpha_i:g\mapsto e_Kgv_i,i=1,...m,$$
where $e_K$ is the idempotent of the Heck algebra $\mathcal{H}(K)$ of $K$. Apparently, the space $V^{K}$ is generated by $\mathcal{L}=\{e_Kgv_i,g\in\rM',i=1,...,m\}$. Suppose that $V^{K}$ has infinite dimension. There is a subset $\mathcal{L'}$ of $\mathcal{L}$ which forms a basis of $V^{K}$. In particular, there exists $i_0\in\{1,...,m\}$ such that $\rM_{i_0}'=\{g\in\rM':e_Kgv_{i_0}\in\mathcal{L'}\}$ is an infinite set. Furthermore, since $K$ stabilises $v_{i_0}$, $\rM_{i_0}'$ is an infinite union of disjoint cosets in the form of $gK$. On the other hand, since the centre $Z'$ of $\rM'$ acts as a character on $\res_{\rM'}^{\rM}\pi$, the cosets $gZ'K,g\in\rM_{i_0}'$ are disjoint in the quotient $\rM'$. Let $v_{i_0}^{\ast}$ be a smooth $k$-linear form of $V^{K}$, which maps each element in $\mathcal{L}'$ to $1\in k$. The support of coefficient $\langle v_{i_0}^{\ast},(\cdot)v_{i_0}\rangle$ contains $\rM_{i_0}'$, which is not compact in $\rM'\slash Z'$. Hence $\res_{\rM'}^{\rM}\pi$ is admissible, and then has finite length.

Now we assume that $\pi$ is irreducible of $\rM$. We first prove that $\res_{\rM'}^{\rM}\pi$ has finite length, and then it is semisimple. For the first part, let $(\rL,\sigma)$ be a cuspidal pair in $\rM$ such that $\pi\hookrightarrow i_{\rL}^{\rM}\sigma$. Applying Theorem \ref{thm 5.2}, we have $\res_{\rM'}^{\rM}i_{\rL}^{\rM}\sigma\cong i_{\rL'=\rL\cap\rM'}^{\rM'}\res_{\rL'}^{\rL}\sigma$. By the first step and the fact that $i_{\rL'}^{\rM'}$ respects finite length, we deduce that  $\res_{\rM'}^{\rM}\pi$ has finite length. For the semi-simplicity, let $W$ be an irreducible sub-representation of $\res_{\rM'}^{\rM}\pi$, of which $gW$ is also an irreducible sub-representation for $g\in\rM$. Let $W'=\sum_{g\in\rM'}g(W)$, which is a semisimple (by the equivalence condition in \S A.VII. of \cite{Re}) sub-representation of $\res_{\rM'}^{\rM}\pi$. Obviously, $\rM$ stabilises $W'$, hence $W'=\pi$.

For the second part, we apply the proof of \cite[Proposition \S 2.2]{TA}. Let $\pi'$ be irreducible of $\rM'$, and $\mathrm{S}$ the subgroup of $\mathrm{Z}$ generated by the scalar matrix $\varpi_{F}$ (a uniformizer of $\fo_{F}$). It is clear that the intersection $\mathrm{S}\cap\rM'=\{ \mathds{1}\}$. Hence we extend $\pi'$ to $\mathrm{S}\rM'$ by acting trivially on $\mathrm{S}$, and denote it by $\tilde{\pi}'$. The quotient $\rM/\mathrm{S}\rM'$ is compact, hence the induction $\ind_{\mathrm{S}\rM'}^{\rM}\tilde{\pi}$ is admissible, and contains an irreducible subrepresentation $\pi\hookrightarrow\ind_{\mathrm{S}\rM'}^{\rM}\tilde{\pi}'$. There is a surjective morphism $\res_{\mathrm{S}\rM'}^{\rM}\pi\rightarrow\tilde{\pi}'$, defined by $f\mapsto f(\mathds{1})$, which induces a surjective morphism $\res_{\rM'}^{\rM}\pi\rightarrow\pi'$. We finish the proof.
\end{proof}

\begin{cor}
\label{cor 20}
Let $\pi$ be irreducible of $\rM$. If the restriction $\res_{\rM'}^{\rM}\pi$ contains an irreducible cuspidal subrepresentation of $\rM'$, then $\pi$ is cuspidal. In other words, an irreducible cuspidal $k$-representation of $\rM'$ is embedded in $\res_{\rM'}^{\rM}\pi$ for an irreducible cuspidal $\pi$ of $\rM$.
\end{cor}

\begin{proof}
Let $\mathrm{P}'=\mathrm{L}'\cdot\rU$ be a proper parabolic subgroup of $\rM'$, and $\mathrm{P}=\mathrm{L}\cdot\rU$ the proper parabolic of $\rM$ such that $\mathrm{P}\cap\rM'=\mathrm{P}'$ and $\mathrm{L}\cap \rM'=\mathrm{L}'$. By Proposition \ref{prop 6} the irreducible components of $\res_{\rM'}^{\rM}\pi$ are $\rM$-conjugate, hence are $\rL$-conjugate. Let $\pi_0'$ be an irreducible component. The Jacquet module $\pi_0'(\rU)=0$ if and only if the Jacquet module of each irreducible component of $\res_{\rM'}^{\rM}\pi$ is equal to $0$, if and only if the Jacquet module $\pi(\rU)=0$.
\end{proof}

\begin{cor}
\label{cor 21}
Let $\pi'$ be an irreducible cuspidal $k$-representation of $\rG'$. There exists a maximal simple $k$-type $(J,\lambda)$ of $\rG$, and a direct factor $\lambda_{M_{\lambda}}'$ of $\ind_{J'}^{M_{\lambda}}\res_{J'}^J\lambda$ (see Definition \ref{defn 15} for $M_{\lambda}$), such that $\pi'\cong\ind_{M_{\lambda}}^{\rG'} \lambda_{M_{\lambda}}'$.
\end{cor}

\begin{proof}
By Corollary \ref{cor 20}, let $\pi$ be irreducible cuspidal of $\rG'$ where $\pi'$ is embedded in. Let $(J_0,\lambda_0)$ be a maximal simple $k$-type of $\rG$ contained in $\pi$, and $(M_{\lambda_0}, \lambda_{M_{\lambda_0}}')$ as in Theorem \ref{thm 16}. Since $\pi\cong\ind_{E^{\times}J_0}^{\rG}\Lambda_0$, where $\Lambda_0$ is an extension of $\lambda$ to $E^{\times}J_0$, and the intersection $E^{\times}J_0\cap\rG'=J_0'$. By Mackey's theory $\ind_{J_0'}^{\rG'}\res_{J_0'}^J\lambda_0$ is embedded in $\res_{\rG'}^{\rG}\pi$. Hence by Proposition \ref{prop 6} $\ind_{M_{\lambda_0}}^{\rG'}\lambda_{M_{\lambda_0}}'\cong g(\pi')$ for a $g\in \rG$, then $\pi'$ contains $g^{-1}(\lambda_{M_{\lambda_0}}')$. Since $g^{-1}(M_{\lambda_0})=M_{g^{-1}(\lambda_0)}$ and $\lambda_{M_{g^{-1}(\lambda_0)}}':=g^{-1}(\lambda_{M_{\lambda_0}}')$ is a direct factor of $\ind_{g^{-1}(J')}^{M_{g^{-1}(\lambda_0)}}g^{-1}(\lambda')$, by Frobenius reciprocity and Theorem \ref{thm 16} we conclude that $\pi'\cong\ind_{M_{g^{-1}(\lambda_0)}}^{\rG'}g^{-1}(\lambda_{M_{g^{-1}(\lambda_0)}}')$ as desired.
\end{proof}

\subsection{Whittaker models and maximal simple $k$-types of $\rG'$}
\label{section 2.3}
In this section, we show $M_{\lambda}=\tilde{J}'$ in Definition \ref{defn 22}  (see Definition \ref{defn 15} for $M_{\lambda}$), which completes the establishment of maximal simple $k$-types of $\rG'$. To be more precise, we show that for $x\in\rU(\fA)$, if $x$ normalises $J$ and $x(\lambda)\cong\lambda\otimes\chi\circ\det$ for a $k$-quasicharacter $\chi$, then $x(\lambda')\cong\lambda'$, for each irreducible direct component $\lambda'$ of $\lambda\vert_{J'}$. To do so, we study Whittaker models and define derivatives for a $k$-representation of $\rM'$, which requires Appendix \ref{appendix A}, the $\ell$-modular version of geometric lemma, and can be viewed as a generalisation of \cite{BeZe} to the representations of $\rG_n'$.

\subsubsection{Uniqueness of Whittaker models}
In this section, we denote $\mathrm{GL}_n(F)$ by $\rG_n$ and $\mathrm{SL}_n(F)$ by $\rG_n'$. Let $\rU=\rU_n(F)$ be the group consisting of unipotent upper triangular matrices in $\rG$. Let $\psi$ be a non-degenerate $k$-character of $\rU$ as defined in \cite[\S III,1]{V1}. Two non-degenerate $k$-characters of $\rU$ are conjugate by a diagonal matrix in $\rG_n$. Let $P_n=\mathrm{P}_n(F)$ be the mirabolic subgroup of $\mathrm{GL}_n(F)$, and $P_n'=P_n\cap \SL_n(F)$. We denote by $V_{n-1}$ the unipotent radical of $P_n$, which is an abelian group isomorphic to the additive group $F^{n-1}$, and is also the unipotent radical of $P_n'$.

\begin{defn}
\label{defn 25}
\begin{enumerate}
\item $r_{\mathrm{id}}:\mathrm{Rep}_k(P_n)\rightarrow \mathrm{Rep}_k(\rG_{n-1})$ the functor of $V_{n-1}$-coinvariants;
\item $r_{\mathrm{id}'}: \mathrm{Rep}_k(P_n')\rightarrow \mathrm{Rep}_k(\rG_{n-1}')$ the functor of $V_{n-1}$-coinvariants;
\item $r_{\psi}:\mathrm{Rep}_k(P_n)\rightarrow \mathrm{Rep}_k(P_{n-1})$ the functor of $(V_{n-1},\psi)$-coinvariants;
\item $r_{\psi}':\mathrm{Rep}_k(P_n')\rightarrow \mathrm{Rep}_k(P_{n-1}')$ the functor of $(V_{n-1},\psi)$-coinvariants.
\end{enumerate}
\end{defn}

\begin{defn}[Bernstein and Zelevinsky]
\label{defn 26}
Let $1\leq k\leq n$, and $\pi\in \mathrm{Mod}_{k}P_n$, $\pi'\in\mathrm{Mod}_k P_n'$. Define $\pi^{(k)}:=r_{\mathrm{id}}r_{\psi}^{k-1}\pi$ to be the $k$-th derivative of $\pi$, and $\pi'^{(\psi,k)}:=r_{\mathrm{id}}'r_{\psi}'^{k-1}\pi'$ to be the $k$-th derivative of $\pi'$ relative to $\psi$.
\end{defn}

\begin{rem}
\label{rem 27}
There is an equivalence $\res_{\rG_{n-k}'}^{\rG_{n-k}}\pi^{(k)}\cong(\res_{\rG_{n}'}^{\rG_n}\pi)^{(k)}$, where $\pi\in\mathrm{Mod}_n\rG_n$.
\end{rem}

\begin{prop}
\label{prop 23}
Let $\pi$ be a cuspidal $k$-representation of $\rG$, then the restriction $\res_{\rG'}^{\rG}\pi$ is multiplicity free.
\end{prop}

\begin{proof}
By Proposition \ref{prop 6}, we have $\res_{\rG_n'}^{\rG_n}\pi\cong\oplus_{i=1}^m \pi_i'$, where $m\in\mathbb{N}$ and $\pi_i'$ are irreducible and cuspidal of $\rG_n'$. By \cite[\S III,1.7]{V1}, we have that $\mathrm{dim}(\pi^{(n)})=1$. By Remark \ref{rem 27}, we then have
$$\mathrm{dim}(\res_{\rG_n'}^{\rG_n}\pi)^{(n)}=\oplus_{i=1}^m\mathrm{dim}(\pi_i'^{(\psi,n)})=1,$$
which implies that there exists one unique $\pi_{i_0}'$, where $1\leq i_0\leq m$, such that $\pi_{i_0}'^{(\psi,n)}$ is non-trivial. Hence $\res_{\rG'}^{\rG}\pi$ is multiplicity free by considering the $\rG_n$-conjugation of $\psi$.
\end{proof}

\begin{cor}
Let $\pi'$ be an irreducible cuspidal $k$-representation of $\rG'$. Then there exists a non-degenerate character $\psi$ of $\rU$, such that $\mathrm{dim}(\pi'^{(g(\psi),n)})=1$.
\end{cor}

\begin{proof}
It is deduced from Corollary \ref{cor 20} and Proposition \ref{prop 23}.
\end{proof}

\subsubsection{Maximal simple $k$-types of $\rG'$}
We complete the establishment of maximal simple $k$-types of $\rG'$. 
\begin{prop}
\label{prop 24}
The subgroup $M_{\lambda}$ in Definition \ref{defn 15} is equal to $\tilde{J}'$.
\end{prop}

\begin{proof}
Let $\Lambda$ be an extension of $\lambda$ to $E^{\times}J$. Then $\ind_{E^{\times}J}^{\rG}\Lambda$ is irreducible and cuspidal of $\rG$, and we denote it by $\pi$. The restriction $\res_{\rG'}^{\rG}\pi$ is semisimple and its direct components are cuspidal. By Theorem \ref{thm 16}, a component $\pi'$ is isomorphic to $\ind_{M_{\lambda}}^{\rG'}\lambda_{M_{\lambda}}'$, for a $\lambda_{M_{\lambda}}'$ as defined in Proposition \ref{prop 14}. In the proof of Theorem \ref{thm 16}, we show that $\mathrm{I}_{\rG'}(\lambda_{M_{\lambda}}')=M_{\lambda}$. Suppose $\tilde{J}'\neq M_{\lambda}$. Let $x\in\tilde{J'}\backslash M_{\lambda}$, we have that $x(\lambda_{M_{\lambda}}')\ncong\lambda_{M_{\lambda}}'$. Meanwhile, since $x(\pi')\cong\pi'$, then
\begin{equation}
\label{equa 5}
\ind_{M_{\lambda}}^{\rG'}x(\lambda_{M_{\lambda}}')\cong\pi'.
\end{equation}
On the other hand, by the fact that $\res_{J'}^{M_{\lambda}}x(\lambda_{\rM_{\lambda}}')\hookrightarrow x(\lambda')$, and the definition of $\tilde{J}$, we have that $x(\lambda')\hookrightarrow\res_{J'}^{J}\lambda$. Hence $x(\lambda_{\rM_{\lambda}}')\hookrightarrow \ind_{J'}^{M_{\lambda}}\res_{J'}^{J}\lambda$.

By Mackey's theory
$$\ind_{J'}^{\rG'}\res_{J'}^{J}\lambda\hookrightarrow\res_{\rG'}^{\rG}\ind_{E^{\times}J}^{\rG}\Lambda.$$
We deduce that $\ind_{M_{\lambda}}^{\rG}\lambda_{M_{\lambda}}'$ and $\ind_{M_{\lambda}}^{\rG}x(\lambda_{M_{\lambda}}')$ are two irreducible components of $\res_{\rG'}^{\rG}\pi$. Since both of them are isomorphic to $\pi$, it is contradicted with Proposition \ref{prop 23}.
\end{proof}

\begin{defn}
\label{defn 22}
Let $(J,\lambda)$ be a maximal simple $k$-type of $\rG$ and $\tilde{J}'=\tilde{J}\cap \rG'$ (see Definition \ref{defn 11}). Let $\tilde{\lambda}'$ be an irreducible direct component of $\ind_{J'}^{\tilde{J}'}\res_{J'}^{J}\lambda$. We define a pair in the form of $(\tilde{J'},\tilde{\lambda}')$ to be a maximal simple $k$-type of $\rG'$. By Corollary \ref{cor 21} and Proposition \ref{prop 24}, for an irreducible cuspidal $k$-representation of $\rG'$, there exists a maximal simple $k$-type $(\tilde{J'},\tilde{\lambda}')$ such that $\pi'\cong\ind_{\tilde{J}'}^{\rG'}\tilde{\lambda}'$.
\end{defn}

\subsection{Maximal simple $k$-types for Levi subgroups of $\rG'$}
\label{section 06}
Since the structure of a Levi subgroup $\rM'$ of $\rG'$ is not a multiple of $p$-adic special linear groups, the establishment of maximal simple $k$-types of $\rG'$ can not be applied directly to $\rM'$. In this section, we establish the maximal simple $k$-types of $\rM'$. We show that the existence for extended maximal simple $k$-types of proper $\rM'$ (they coincide with maximal simple $k$-types when $\rM'=\rG'$), and that a cuspidal $k$-representation of $\rM'$ is obtained from an extended maximal simple $k$-type rather than a maximal simple $k$-type, which is different from the case of $\rG'$.

\subsubsection{Intertwining and weakly intertwining}
\label{subsection 2.4.1}
In this section, let $\rM\cong\mathrm{GL}_{n_1}\times\cdots\times\mathrm{GL}_{n_r}$ be a Levi subgroup of $\rG$. For a closed subgroup $H$ of $\rG$, we always denote $H\cap\rG'$ as $H'$. As in Section \ref{section 08}, we start by considering the maximal simple $k$-types of $\rM$. In particular, Proposition \ref{Lprop 0.1}, Proposition \ref{Lprop 0.2}, Definition \ref{prepdefn 01} and Proposition \ref{Lprop 0.3} will be applied in this section. 

\begin{prop}
\label{Lprop 1}
Let $(J_{\rM},\lambda_{\rM})$ be a maximal simple $k$-type of $\rM$, and $\chi$ a $k$-quasicharacter of $F^{\times}$. If $(J_{\rM},\lambda_{\rM}\otimes\chi\circ\det)$ is weakly intertwined with $(J_{\rM},\lambda_{\rM})$, then they are intertwined in $\rM$, and there exists an element $x\in\rU(\fA_{\rM})=\rU(\fA_1)\times\cdots\times\rU(\fA_r)$ such that $x(J_{\rM})=J_{\rM}$ and $x(\lambda_{\rM})\cong\lambda_{\rM}\otimes\chi\circ\det$, where $\fA_i$ is a hereditary order associated to $(J_i,\lambda_i)$ ($i=1,...,r$). Furthermore, for a $g\in\rG$, if $g$ weakly intertwines $(J_{\rM},\lambda_{\rM}\otimes\chi\circ\det)$ with $(J_{\rM},\lambda_{\rM})$, then $g$ intertwines $(J_{\rM},\lambda_{\rM}\otimes\chi\circ\det)$ with $(J_{\rM},\lambda_{\rM})$.

\end{prop}

\begin{proof}

By definition, $J_{\rM}=J_1\times\cdots\times J_r$ and $\lambda_{\rM}\cong\lambda_1\times\cdots\times\lambda_r$, where $(J_i,\lambda_i)$ are $k$-maximal cuspidal simple type of $\mathrm{GL}_{n_i}$ for $i\in\{ 1,\ldots,r\}$. The group $\rU(\fA_{\rM})=\rU(\fA_1)\times\cdots\times\rU(\fA_r)$. Hence the two results are directly deduced by \ref{prop 1} and \ref{cor 5}.

\end{proof}

\begin{defn}
\label{Ldefinition 21}
Let $(J_{\rM},\lambda_{\rM})$ be a $k$-maximal cuspidal simple type of $\rM$. We define the group of projective normalizer $\tilde{J}_{\rM}$ a subgroup of $\rU(\fA_{\rM})$, where $\fA_{\rM}=\fA_1\times\cdots\times\fA_r$, such that $x\in\tilde{J}_{\rM}$, if and only if $x(J_{\rM})=J_{\rM}$ and $x(\lambda_{\rM})\cong \lambda_{\rM}\otimes\chi\circ\mathrm{det}$ for a $k$-quasicharacter $\chi$ of $F^{\times}$.  
\end{defn}

Since $\tilde{J}_{\rM}$ is subgroup of $\tilde{J}_1\times\cdots\times\tilde{J}_r$, the quotient $\tilde{J}_{\rM}\slash J_{\rM}$ is a finite abelian $p$-group. The proof of Corollary \ref{cor 99} can be applied to $(J_{\rM},\lambda_{\rM})$, hence the induction $\tilde{\lambda}_{\rM}:=\ind_{J_{\rM}}^{\tilde{J}_{\rM}}\lambda_{\rM}$ is irreducible. By Proposition \ref{Lprop 0.1}, the restriction $\res_{\tilde{J}_{\rM}'}^{\tilde{J}_{\rM}}\tilde{\lambda}_{\rM}$ is semisimple, of which $\tilde{\lambda}_{\rM}'$ is an irreducible component. 

\begin{lem}
\label{Llemma 3}
Let $(J_{\rM},\lambda_{\rM})$ be a maximal simple $k$-type of $\rM$, and $\tilde{\lambda}_{\rM,1}',\tilde{\lambda}_{\rM,2}'$ two irreducible components of $\res_{\tilde{J}_{\rM}'}^{\tilde{J}_{\rM}}\tilde{\lambda}_{\rM}$. Then :
$$\mathrm{I}_{\rM'}^{w}(\tilde{\lambda}_{\rM,1}',\tilde{\lambda}_{\rM,2}')=\{ m\in \rM':m(\tilde{\lambda}_{\rM,1}')\cong\tilde{\lambda}_{\rM,2}' \},$$
hence $\mathrm{I}_{\rM'}^{w}(\tilde{\lambda}_{\rM,1}',\tilde{\lambda}_{\rM,2}')=\mathrm{I}_{\rM'}(\tilde{\lambda}_{\rM,1}',\tilde{\lambda}_{\rM,2}')$. In particular, $\mathrm{I}_{\rM'}(\tilde{\lambda}_{\rM,2}')$ is the normaliser group of $\tilde{\lambda}_{\rM,2}'$ in $\rM'$. Moreover, the group $\mathrm{I}_{\rM'}(\tilde{\lambda}_{\rM,2}')$ dependents on $(J_{\rM},\lambda_{\rM})$ but not the choice of $\tilde{\lambda}_{\rM,2}'$.
\end{lem}

\begin{proof}
Let $m\in\rM'$ weakly intertwines $\tilde{\lambda}_{\rM,2}'$ with $\tilde{\lambda}_{\rM,1}'$, then by Proposition \ref{Lprop 0.3}, $m$ weakly intertwines $\tilde{\lambda}_{\rM}$ with $\tilde{\lambda}_{\rM}\otimes\chi\circ\mathrm{det}$ for a $\chi$. By Mackey's theory
$$\tilde{\lambda}_{\rM}\vert_{J_{\rM}}\cong\oplus_{x\in \tilde{J}_{\rM}/J_{\rM}}x(\lambda_{\rM})\cong \oplus_{x\in\tilde{J}_{\rM}/J_{\rM}}\lambda_{\rM}\otimes\chi_x\circ\mathrm{det},$$
where $\chi_x$ is a $k$-quasicharacter of $F^{\times}$. Hence we have $\tilde{\lambda}_{\rM}\otimes\chi_x\circ\mathrm{det}\cong\tilde{\lambda}_{\rM}$ for every $x\in\tilde{J}_{\rM}/J_{\rM}$. It follows that for a $g\in \tilde{J}_{\rM}$, the element $gm$ weakly intertwines $\lambda_{\rM}$ with $\lambda_{\rM}\otimes\chi_x\cdot\chi\circ\mathrm{det}$ for a $x\in \tilde{J}_{\rM}/J_{\rM}$. By Proposition \ref{Lprop 1}, the element $gm$ intertwines $\lambda_{\rM}$ with $\lambda_{\rM}\otimes\chi_x\cdot\chi\circ\mathrm{det}$, and $y(\lambda_{\rM})\cong \lambda_{\rM}\otimes\chi_x\cdot\chi\circ\mathrm{det}$ for a $y\in\tilde{J}_{\rM}$. By induction to $\tilde{J}_{\rM}$, we obtain an isomorphism $\tilde{\lambda}_{\rM}\cong\tilde{\lambda}_{\rM}\otimes\chi\circ\mathrm{det}$, hence $m\in\mathrm{I}_{\rM'}(\tilde{\lambda}_{\rM})$.

Furthermore, the intertwining set $\mathrm{I}_{\rM}(\lambda_{\rM})=N_{\rM}(\lambda_{\rM})$, and the latter is the normaliser group of $\lambda_{\rM}$, which also normalises $\rU(\fA_{\rM})$, hence normalises $\tilde{J}_{\rM}$. We deduce that $\mathrm{I}_{\rM}(\tilde{\lambda}_{\rM})=\tilde{J}_{\rM} N_{\rM}(\lambda_{\rM})$. Then an element of $\mathrm{I}_{\rM'}^{w}(\tilde{\lambda}_{\rM,2}',\tilde{\lambda}_{\rM,1}')$ normalises $\tilde{\lambda}_{\rM}$ and $\tilde{J}_{\rM}'$, which implies the first two assertions.

To prove the last assertion that $\mathrm{I}_{\rM'}(\tilde{\lambda}_{\rM,2}')$ is independent of the choice of $\tilde{\lambda}_{\rM,2}'$, we observe at first that the irreducible components of $\tilde{\lambda}_{\rM}\vert_{\tilde{J}_{\rM}'}$ are $\tilde{J}_{\rM}$-conjugate. We have to show therefore that $\tilde{J}_{\rM}$ normalises $N_{\rM'}(\tilde{\lambda}_{\rM,2}')$.

The quotient group $N_{\rM}(\tilde{\lambda}_{\rM})/\tilde{J}_{\rM}\cong N_{\rM}(\tilde{\lambda}_{\rM})/\tilde{J}_{\rM}$ is abelian, from which we deduce that  $N_{\rM'}(\tilde{\lambda}_{\rM,2}')$ is a subgroup of $N_{\rM}(\tilde{\lambda}_{\rM})\cap\rM'$. Let $x\in \tilde{J}_{\rM}$ and $y\in N_{\rM'}(\tilde{\lambda}_{\rM,2}')$, we have $x^{-1}yx=y\cdot m$ for some $m\in\tilde{J}_{\rM}\cap\rM'=\tilde{J}_{\rM}'$. Therefore:
$$x^{-1}yx(\tilde{\lambda}_{\rM,2}')\cong y(\tilde{\lambda}_{\rM,2}')\cong\tilde{\lambda}_{\rM,2}',$$
as required.
\end{proof}

\begin{rem}
\label{Lrem aa1}
To be more detailed, we proved that the intertwining group $\mathrm{I}_{\rM'}(\tilde{\lambda}_{\rM}')$ is the stabiliser group $N_{\rM'}(\tilde{\lambda}_{\rM}')$, which is a subgroup of $E_1^{\times}\tilde{J}_1\times\cdots\times E_r^{\times}\tilde{J}_r \cap \rM'$, where $E_i$'s are finite field extensions of $F$, hence a compact group modulo centre.
\end{rem}

\subsubsection{Maximal simple $k$-types of $\rM'$}
\label{subsection 2.4.2}
In this section, we construct maximal simple $k$-types of $\rM'$, and extended maximal simple $k$-types (Definition \ref{Ldefinition 200}). We show that every irreducible cuspidal $k$-representation of $\rM'$ is constructed from an extended maximal simple $k$-type. We follow the same strategy as for the case of $\rG'$, which is by checking the two conditions of irreducibility in Lemma \ref{lem 19}.

\begin{lem}
\label{Llemma 2}
We have a decomposition:
$$\res_{J_{\rM}}^{\rM}\ind_{J_{\rM}}^{\rM}\lambda_{\rM}\cong \Lambda_{\lambda_{\rM}}\oplus W_{\rM},$$
where $\Lambda_{\lambda_{\rM}}$ is semisimple, of which each irreducible component is isomorphic to $\lambda_{\rM}\otimes\chi\circ\mathrm{det}$ for a $k$-quasicharacter $\chi$ of $F^{\times}$, and non of irreducible subquotients of $W_{\rM}$ is contained in $\Lambda_{\lambda_{\rM}}$.
\end{lem}

\begin{proof}
This is directly deduced from the decomposition in Theorem \ref{thm 7}.
\end{proof}

\begin{prop}
\label{Lproposition 4}
Let $\tilde{\lambda}_{\rM}'$ be an irreducible subrepresentation of $\res_{\tilde{J}_{\rM}'}^{\tilde{J}_{\rM}}\tilde{\lambda}_{\rM}$, where $\tilde{J}_{\rM}'=\tilde{J}_{\rM}\cap\rM'$, then $\tilde{\lambda}_{\rM}'$ satisfies the second condition of irreducibility in Lemma \ref{lem 19}. In other words, for an irreducible representation $\pi'$ of $\rM'$, if there is an embedding $\tilde{\lambda}_{\rM}'\hookrightarrow \res_{\tilde{J}_{\rM}'}^{\rM'}\pi'$, then there exists a surjection $\res_{\tilde{J}_{\rM}'}^{\rM'}\pi'\twoheadrightarrow\tilde{\lambda}_{\rM}'$.
\end{prop}

\begin{proof}
The proof of Proposition \ref{prop 14} can be applied here. 
\end{proof}

\begin{prop}
\label{Lproposition 3}
Let $\tau_{\rM'}$ be an irreducible representation of $N_{\rM'}(\tilde{\lambda}_{\rM}')$ which contains $\tilde{\lambda}_{\rM}'$. Then $\tau_{\rM'}$ satisfies the second condition of irreducibility.
\end{prop}

\begin{proof}
Denote $N_{\rM'}(\tilde{\lambda}_{\rM}')$ by $N_{\rM'}$, by Mackey's theory we have
\begin{equation}
\label{equation L+1}
\res_{N_{\rM'}}^{\rM'}\ind_{N_{\rM'}}^{\rM'}\tau_{\rM'}\cong\oplus_{N_{\rM'}\backslash\rM' / N_{\rM'}}\ind_{N_{\rM'}\cap a(N_{\rM'})}^{N_{\rM'}}\res_{N_{\rM'}\cap a(N_{\rM'})}^{a(N_{\rM'})}a(\tau_{\rM'}).
\end{equation}
Notice that $\tilde{J}_{\rM}'$ is the unique maximal open compact subgroup of $N_{\rM'}$. Hence for $b,a\in\rM'$ we have $\tilde{J}_{\rM}'\cap ba(N_{\rM'})=\tilde{J}_{\rM}'\cap ba(\tilde{J}_{\rM}')$. Then
$$\res_{\tilde{J}_{\rM}'}^{N_{\rM'}}\ind_{N_{\rM'}\cap a(N_{\rM'})}^{N_{\rM'}}\res_{N_{\rM'}\cap a(N_{\rM'})}^{a(N_{\rM'})}a(\tau_{\rM'})$$
$$\cong\oplus_{b\in N_{\rM'}\cap a(N_{\rM'})\backslash N_{\rM'} / \tilde{J}_{\rM}'} \ind_{\tilde{J}_{\rM}'\cap ba(N_{\rM'})}^{\tilde{J}_{\rM}'}\res_{\tilde{J}_{\rM}'\cap ba(N_{\rM'})}^{ba(N_{\rM'})}ba(\tau_{\rM'})$$
$$\cong\oplus_{b \in N_{\rM'}\cap a(N_{\rM'})\backslash N_{\rM'} / \tilde{J}_{\rM}'} \ind_{\tilde{J}_{\rM}'\cap ba(\tilde{J}_{\rM}')}^{\tilde{J}_{\rM}'}\res_{\tilde{J}_{\rM}'\cap ba(\tilde{J}_{\rM}')}^{ba(\tilde{J}_{\rM}')}ba(\oplus\tilde{\lambda}_{\rM}'),$$
where $\oplus\tilde{\lambda}_{\rM}'$ is a finite multiple of $\tilde{\lambda}_{\rM}'$.

For $a\notin N_{\rM'}$, by Lemma \ref{Llemma 3}, we have $ba\notin \mathrm{I}_{\rM'}^{w}(\tilde{\lambda}_{\rM}')$. This implies that non of irreducible subquotients of $\ind_{\tilde{J}_{\rM}'\cap ba(N_{\rM'})}^{\tilde{J}_{\rM}'}\res_{\tilde{J}_{\rM}'\cap ba(N_{\rM'})}^{ba(N_{\rM'})}ba(\tau_{\rM'})$ is isomorphic to $\tilde{\lambda}_{\rM}'$. Now combining with (\ref{equation L+1}), we obtain that
$$\res_{N_{\rM'}}^{\rM'}\ind_{N_{\rM'}}^{\rM'}\tau_{\rM'}\cong\tau_{\rM'}\oplus W_{N_{\rM'}},$$
where non of irreducible subquotient of $W_{N_{\rM'}}$ is isomorphic to $\tau_{\rM'}$.

Now we verify the second condition in Lemma \ref{lem 19} for $\tau_{\rM'}$. Let $\pi'$ be irreducible of $\rM'$. If there is an embedding $\tau_{\rM'}\hookrightarrow \res_{N_{\rM'}}^{\rM'}\pi'$, then $\res_{N_{\rM'}}^{\rM'}\pi'$ is isomorphic to a quotient of $\res_{N_{\rM'}}^{\rM'}\ind_{N_{\rM'}}^{\rM'}\tau_{\rM'}$ and we denote it by $(\res_{N_{\rM'}}^{\rM'}\ind_{N_{\rM'}}^{\rM'}\tau_{\rM'})/ W_0$. The image of the composed morphism
$$\tau_{\rM'}\hookrightarrow \res_{N_{\rM'}}^{\rM'}\pi'\cong(\res_{N_{\rM'}}^{\rM'}\ind_{N_{\rM'}}^{\rM'}\tau_{\rM'})/ W_0$$
is not contained in $(W_{N_{\rM'}}+W_0)/ W_0$ as explained above. Hence $\tau_{\rM'}$ is a quotient of $\res_{N_{\rM'}}^{\rM'}\pi'$.
\end{proof}

\begin{lem}
\label{Llemma 0.3}
Let $\rA$ be a locally pro-finite group, and $K_1,K_2$ two open subgroups of $\rA$, such that $K_1$ is the unique maximal open compact subgroup of $K_2$. Let $\pi$ be an irreducible $k$-representation of $K_2$, and $\tau$ an irreducible $k$-representation of $K_1$, such that $\pi\vert_{K_1}$ is a multiple of $\tau$. If $x\in\mathrm{I}_{\rA}(\pi)$ (resp. $x\in\mathrm{I}_{\rA}^w(\pi)$), then there exists an element $y\in K_2$ such that $yx\in\mathrm{I}_{\rA}(\tau)$ (resp. $yx\in\mathrm{I}_{\rA}^w(\tau)$).
\end{lem}

\begin{proof}
Since $\pi$ is isomorphic to a subquotient of $\ind_{K_2\cap x(K_2)}^{K_2}\res_{K_2\cap x(K_2)}^{x(K_2)}x(\pi)$, the restriction $\res_{K_1}^{K_2}\pi$ is isomorphic to a subquotient of $R(\pi):=\res_{K_1}^{K_2}\ind_{K_2\cap x(K_2)}^{K_2}\res_{K_2\cap x(K_2)}^{x(K_2)}x(\pi)$. Applying Mackey's theory, $R(\pi)$ is isomorphic to
$$\bigoplus_{a\in K_2\cap x(K_2)\backslash K_2\slash K_1}\ind_{K_1\cap ax(K_2)}^{K_1}\res_{K_1\cap ax(K_2)}^{ax(K_2)}ax(\pi).$$
Since $K_1\cap ax(K_2)$ is open compact, and the maximal open compact subgroup of $ax(K_2)$ is unique, which implies that  $K_1\cap ax(K_2)=K_1\cap ax(K_1)$. Meanwhile, since $\res_{K_1}^{K_2}\pi$ is a multiple of $\tau$, and the functors $\ind,\res$ can change order with infinite direct sum, $R(\pi)$ is isomorphic to a multiple of
$$\bigoplus_{a\in K_2\cap x(K_2)\backslash K_2\slash K_1}\ind_{K_1\cap ax(K_1)}^{K_1}\res_{K_1\cap ax(K_1)}^{ax(K_1)}ax(\tau).$$
As in the proof of Lemma \ref{lem 9}, this implies that there exists at least one $y\in K_2$ such that $\tau$ is an subquotient of $\ind_{K_1\cap yx(K_1)}^{K_1}\res_{K_1\cap yx(K_1)}^{yx(K_1)}yx(\tau)$.
\end{proof}

\begin{thm}
\label{Ltheorem 5}
The induction $\ind_{N_{\rM'}(\tilde{\lambda}_{\rM}')}^{\rM'}\tau_{\rM'}$ is cuspidal and irreducible. Conversely, an irreducible cuspidal $k$-representation $\pi'$ of $\rM'$ is equivalent to $\ind_{N_{\rM'}(\tilde{\lambda}_{\rM}')}^{\rM'}\tau_{\rM'}$, for a pair $(N_{\rM'}(\tilde{\lambda}_{\rM}'),\tau_{\rM'})$ as in Proposition \ref{Lproposition 3} defined from a maximal simple $k$-type $(J_{\rM},\lambda_{\rM})$ of $\rM$.
\end{thm}

\begin{proof}
For the first part, we first verify the two conditions in Lemma \ref{lem 19}. The second condition has been checked in \ref{Lproposition 3}. By \ref{Llemma 0.3} and \ref{Llemma 3}, the first condition is satisfied. Denote the irreducible induction $\ind_{N_{\rM'}(\tilde{\lambda}_{\rM}')}^{\rM'}\tau_{\rM'}$ by $\pi'$, and let $\pi$ be irreducible of $\rM$ as in \ref{prop 6}. We deduce from \ref{lem 9} that $\lambda_{\rM}\otimes\chi\circ \det)$ is a subquotient of $\pi\vert_{J_{\rM}}$ for a $\chi$. Hence $\pi$ is cuspidal (by \ref{prop 1} and \ref{cor 3}) which implies that $\pi'$ is cuspidal. 

For the second part, let $\pi'$ be irreducible and cuspidal of $\rM'$, and $\pi$ be irreducible cuspidal of $\rM$ as in \ref{prop 6}. Then there exists a maximal simple $k$-type $(J_{\rM},\lambda_{\rM})$, and an extension $\Lambda_{\rM}$ of $\lambda_{\rM}$ to $N_{\rM}(\lambda_{\rM})$ such that $\pi\cong\ind_{N_{\rM}(\lambda_{\rM})}^{\rM}\Lambda_{\rM}$. Let $\tilde{\lambda}_{\rM}=\ind_{J_{\rM}}^{\tilde{J}_{\rM}}\lambda_{\rM}$, and $N_{\rM}(\tilde{\lambda}_{\rM})$ be the group of normalisers of $\tilde{\lambda}_{\rM}$ in $\rM$. By the transitivity of induction:
$$\pi\cong\ind_{N_{\rM}(\tilde{\lambda}_{\rM})}^{\rM}\circ\ind_{N_{\rM}(\lambda_{\rM})}^{N_{\rM}(\tilde{\lambda}_{\rM})}\Lambda_{\rM}.$$
Denote $\ind_{N_{\rM}(\lambda_{\rM})}^{N_{\rM}(\tilde{\lambda}_{\rM})}\Lambda_{\rM}$ by $\tilde{\Lambda}_{\rM}$, which is irreducible containing $\tilde{\lambda}_{\rM}$.

Till the end of this proof, we denote by $\tilde{\lambda}_{\rM}'$ a direct component of $\tilde{\lambda}_{\rM}\vert_{\tilde{J}_{\rM}'}$, and denote $N_{\rM}(\tilde{\lambda}_{\rM})$ by $N$, $N\cap \rM'$ by $N'$, and $N_{\rM'}(\tilde{\lambda}_{\rM}')$ by $N_{\rM'}$. Let $K$ be the kernel of $\tilde{\lambda}_{\rM}$, and $Z$ the centre of $\rM$, $Z'=Z\cap\rM'$. Since the quotient $(ZN')\slash N$ is compact and $K$ is open, we deduce that $ZN'K$ is a normal subgroup with finite index of $N$. Hence $\res_{ZN'K}^{N}\tilde{\Lambda}_{\rM}$ is semisimple of finite length as in the proof of \ref{prop 6}. We deduce that $\res_{N'}^{N}\tilde{\Lambda}_{\rM}$ is semisimple of finite length as well. By a conjugation of $\rM$, we can assume that $\pi'$ contains a direct factor of $\res_{N'}^{N}\tilde{\Lambda}_{\rM}$, which we denote by $\tilde{\Lambda}_{\rM}'$. By Mackey's theory $\res_{\tilde{J}_{\rM}'}^{N}\ind_{\tilde{J}_{\rM}}^{N}\tilde{\lambda}_{\rM}$ is a multiple of $\res_{\tilde{J}_{\rM}'}^{\tilde{J}_{\rM}}\tilde{\lambda}_{\rM}$, and we can assume that $\tilde{\Lambda}_{\rM}'$ contains a $\tilde{\lambda}_{\rM}'$. On the other hand, $N_{\rM'}$ is a normal subgroup with finite index in $N'$. In fact, $N_{\rM'}$ contains $Z'\tilde{J}_{\rM}'$. We have showed in \ref{Lrem aa1} that $N_{\rM'}$ is a subgroup of $E_1^{\times}\tilde{J}_1\times\cdots\times E_r^{\times} \tilde{J}_r\cap\rM'$ which is compact modulo the centre. Hence $\res_{N_{\rM'}}^{N'}\tilde{\Lambda}_{\rM}'$ is semisimple of finite length, and there must be a direct component $\tau_{\rM'}$ that contains $\tilde{\lambda}_{\rM}'$. Hence we have
$$\pi'\cong\ind_{N_{\rM'}}^{\rM'}\tau_{\rM'}.$$
\end{proof}

\begin{defn}
\label{Ldefinition 200}
\begin{itemize}
\item Let $(J_{\rM},\lambda_{\rM})$ be a maximal simple $k$-type of $\rM$, and $\tilde{\lambda}_{\rM}'$ be an irreducible component of $\res_{\tilde{J}_{\rM}'}^{\tilde{J}_{\rM}}\tilde{\lambda}_{\rM}$, where $\tilde{J}_{\rM}$ and $\tilde{\lambda}_{\rM}$ are defined as in \ref{Ldefinition 21}. We define the couples in forms of $(\tilde{J}_{\rM}',\tilde{\lambda}_{\rM}')$ are \textbf{the maximal simple $k$-types} of $\rM'$.
\item Let $N_{\rM'}(\tilde{\lambda}_{\rM}')$ be the normaliser group of $\tilde{\lambda}_{\rM}'$ in $\rM'$, and $\tau_{\rM'}$ an irreducible $k$-representation of $N_{\rM'}(\tilde{\lambda}_{\rM}')$ containing $\tilde{\lambda}_{\rM}'$. We say a pair of the form $(N_{\rM'}(\tilde{\lambda}_{\rM}'),\tau_{\rM'})$ is an \textbf{extended maximal simple $k$-type} of $\rM'$.
\end{itemize}
\end{defn}

\begin{rem}
We will prove in Theorem \ref{thm a33} that $\tau_{\rM'}$ is an extension of $\tilde{\lambda}_{\rM}'$, which is a parallel result for complex case.
\end{rem}

\section{Intertwining and conjugacy}
\label{chapter 01}
In this section, we prove the unicity property of weakly intertwining implying conjugacy (Theorem \ref{thm a30}), and the unicity property of maximal simple $k$-types (resp. simple $k$-characters) contained in an irreducible cuspidal $k$-representation of $\rM'$ (Theorem \ref{thm a40}). The first one is a well-known result for representations of characteristic zero of $\rM'$, and was proved in \cite{BuKuII} for $\rM'=\rG'$. The philosophy can be generalised to our case, with some technical difficulties that arise from the $\ell$-modular Clifford theory for $\pi\vert_{\rG'}$ in Section \ref{section 03.1}. The main results are proved in Section \ref{section intconj}.

\subsection{A formula of the length of $\pi\vert_{\rG'}$}
\label{section 03.1}

Recall that $\rG=\mathrm{GL}_n(F)$ and $\rG'=\mathrm{SL}_n(F)$. Let $(\pi, V)$ be an irreducible $k$-representation of $\rG$, we have seen in Proposition \ref{prop 6} that the restriction of $\pi$ to $\rG'$ is a finite direct sum of irreducible representations, that are conjugate under $\rG$, and let $(\pi',V')$ be one of them. We begin with a technical section to deal with the main difficulty while applying the method in \cite{BuKuII} to our case, that is the inclusion in \cite[Proposition 1.5]{BuKuII} which is not true in the $\ell$-modular setting. It leads to the failure of \cite[Corollary 1.6]{BuKuII}, which is the base stone in the proof of \cite[Theorem 5.3]{BuKuII}. Fortunately, we obtain a generalisation of this inclusion, which gives an interpretation of the $\ell$-prime part of the length of $\pi\vert_{\rG'}$, as the $\ell$-modular version of \cite[Corollary 1.6]{BuKuII}, and which is the key to prove the property of weakly intertwining implying conjugacy.

Write
\begin{equation}
\label{equa a12}
\mathcal{S}(\pi)=\{ x\in\rG: x(\pi')\cong\pi'  \}.
\end{equation}
The index of $\mathcal{S}(\pi)$ in $\rG$ is finite. In fact, the group $\mathcal{S}(\pi)$ contains $\rG'$ and the centre $\mathrm{Z}(\rG)$ of $\rG$. On the other hand, denoting by $V'$ the representation space of $\pi'$, for any non-zero vector $v'\in V'$, since $\rG'$ is normal in $\rG$ and $v'$ generalises $V'$ as $\rG'$-representation, hence a stabiliser of $v'$ in $\rG$ also normalises $V'$, then it is contained in $\mathcal{S}(\pi)$, which implies that $\mathcal{S}(\pi)$ is open with finite index in $\rG$.

Write
$$\mathcal{G}(\pi)=\{\chi:\pi\otimes\chi\circ\det\cong\pi\},$$
where $\chi$ ranges over the set of $k$-quasicharacters of $F^{\times}$.

Write
$$\mathcal{T}(\pi)=\bigcap_{\chi\in\mathcal{G}(\pi)}\mathrm{Ker}(\chi\circ\det).$$
The group $\mathcal{T}(\pi)$ is closed in $\rG$.

\begin{prop}[Proposition $1.4$ in \cite{BuKuII}]
The group $\mathcal{G}(\pi)$ is finite. The subgroup $\mathcal{T}(\pi)$ of $\rG$ contains $Z\rG'$, and is open of finite index, where $Z$ is the centre of $\rG$.
\end{prop}

\begin{proof}
The original proof can be applied directly.
\end{proof}

Let $A$ be a finite abelian group. Let $A_{\ell}$ be the subgroup of $A$ consisting of elements whose order are powers of $\ell$, and we say that $A_{\ell}$ is the \textbf{$\ell$-power part} of $A$. Meanwhile, let $A_{\ell'}$ be the subgroup of $A$ consisting of elements whose orders are prime to $\ell$, and we say it is the \textbf{$\ell$-prime part} of $A$. 

\begin{lem}
\label{lem a4}
Let $A$ be a finite abelian group, we consider the $\ell$-dual group $A^{\wedge}$ of $A$, which is the group of $k$-characters of $A$. Then $A^{\wedge}$  is isomorphic to ($A_{l'})^{\wedge}$.
\end{lem}

\begin{proof}
Since $A\slash A_{\ell}\cong A_{\ell'}$, and $A_{\ell}$ is contained in the kernel of each $k$-character of $A$, we deduce the result.
\end{proof}

Composition with the determinant induces an isomorphism
$$\mathcal{G}(\pi)\cong(\rG\slash\mathcal{T}(\pi))^{\wedge}.$$
The lemma above implies that the $\ell$-power part $(\rG\slash\mathcal{T}(\pi))_{\ell}$ of the quotient group $\rG\slash\mathcal{T}(\pi)$ is trivial.

\begin{lem}
\label{lem a11}
Let $\rG$ be a locally pro-finite group and $\rH$ a normal subgroup of $\rG$. Let $(\pi,V)$ be an irreducible $k$-representation of $\rH$. Assume that $\rG\slash\rH$ is finite cyclic and $\rG$ normalises $(\pi,V)$, then $(\pi,V)$ can be extended to $\rG$.
\end{lem}

\begin{proof}
This is a well-known result. For those who might be interested in, the construction in the first paragraph of the proof of Proposition $5.2.4$ in \cite{BuKu} can be applied.
\end{proof}

Let $(\pi,V)$ be an irreducible $k$-representation of $\rG$, and $(\pi',V')$ an irreducible sub-representation of $(\pi\vert_{\rG'},V)$. Let $d$ be the multiplicity of $\pi'$ in $\pi\vert_{\rG'}$, which is independent of the choice of $\pi'$. Let $\rH(\pi',V')$ be the subgroup of $\rG$, consisting of elements $g$ such that $\pi(g)V'=V'$. In other words, $\rH(\pi',V')$ is the group of $\rG$-stabilizer of $(\pi',V')$. Hence $\rH(\pi',V')$ is an open subgroup of $\mathcal{S}(\pi)$. Moreover, the quotient $\mathcal{S}(\pi)\slash\rH(\pi',V')$ is finite and abelian. Since the restriction of $\pi$ to any normal subgroup with finite quotient in $\rG$ is semisimple of finite length, the method of \cite[\S 1.16]{BuKuI} can be applied, hence by choosing $V'$, we can assume that $\rH(\pi',V')$ is maximal, and write $\pi_1$ to be the $\rH(\pi',V')$-representation on $V'$. Then we have,
$$\pi\cong\ind_{\rH(\pi',V')}^{\rG}\pi_1,$$
which implies that $d=(\mathcal{S}(\pi):\rH(\pi',V'))$. The group $\rH(\pi',V')$ is independent of the choice of $\pi'$ and $V'$, and we denote it as $\rH(\pi)$.

\begin{rem}
\label{rem add01}
Let $\chi$ be a $k$-quasicharacter of $F^{\times}$ which is trivial on $\det(\rH(\pi))$, and $\pi_1$ as above, then $\pi_1\otimes\chi\circ\det\cong\pi_1$. It follows that $\pi\otimes\chi\circ\det\cong\pi$. 
\end{rem}

The purpose of Lemma \ref{lem a9} and Lemma \ref{lem a2} below is to prove Proposition \ref{prop a1}, which is a generalisation of \cite[Proposition 1.5]{BuKuII} and the main result of this section. The complexity of Lemma \ref{lem a9} arises from the fact that the $\ell$-dual set of a non-trivial finite abelian group can be trivial, which is not possible for $\cK$-quasicharacters. Hence, the group $\mathcal{T}(\pi)$ given from $k$-characters will be larger than the setting of characteristic zero. However, the following two lemmas are sufficient for the later use.

\begin{lem}
\label{lem a9}
There is an equivalence between the $\ell$-dual groups $(\rG\slash\mathcal{T}(\pi)\cap\rH(\pi))^{\wedge}$ and $(\rG\slash\mathcal{T}(\pi))^{\wedge}$. Let $\rH(\pi)_{\ell'}$ (resp. $\rH(\pi)_{\ell}$) be a subgroup of $\rH(\pi)$, consisting with the elements whose images by projection belong to $(\rH(\pi)\slash\mathcal{T}(\pi)\cap\rH(\pi))_{\ell'}$ (resp. $(\rH(\pi)\slash\mathcal{T}(\pi)\cap\rH(\pi))_{\ell}$). Then $\rH(\pi)_{\ell'}$ is equal to $\rH(\pi)$.
\end{lem}

\begin{proof}
First, we consider the quotient group 
$$\overline{\mathcal{T}}=\mathcal{T}(\pi)\slash\mathcal{T}(\pi)\cap\rH(\pi)\cong\mathcal{T}(\pi)\rH(\pi)\slash\rH(\pi).$$
The latter is a subgroup of the finite abelian group $\rG\slash\rH(\pi)$. Hence a $k$-quasicharacter $\theta$ of $\overline{\mathcal{T}}$ can be extended to $\rG\slash\rH(\pi)$. Let $\tilde{\theta}$ be such an extension. The equivalence given in Remark \ref{rem add01} implies that $\tilde{\theta}\in\mathcal{G}(\pi)$, hence $\theta$ is trivial by the definition of $\mathcal{T}(\pi)$. We deduce from Lemma \ref{lem a4} that $\overline{\mathcal{T}}$ only has $\ell$-power part, in other words, $\overline{\mathcal{T}}_{\ell}$ is equal to $\overline{\mathcal{T}}$. We conclude that, if a $k$-quasicharacter $\chi$ of $F^{\times}$ is trivial on $\det(\mathcal{T}(\pi)\cap\rH(\pi))$, then it is trivial on $\det(\mathcal{T}(\pi))$, hence $\chi\circ\det$ belongs to $(\rG\slash\mathcal{T}(\pi))^{\wedge}$.

Now consider the second part of this lemma. For any $\chi\circ\det\in(\rG\slash\mathcal{T}(\pi)\cap\rH(\pi))^{\wedge}$, we have an inclusion $\rH(\pi)_{\ell}\subset\mathrm{ker}(\chi\circ\det)$, which implies that 
$$\rH(\pi)_{\ell}\subset\mathcal{T}(\pi).$$ 
Hence $\rH(\pi)_{\ell}$ is equal to $\mathcal{T}(\pi)\cap\rH(\pi)$.
\end{proof}

Write $\pi_0=\pi_1\vert_{\mathcal{T}(\pi)\cap\rH(\pi)}$. The representation $\pi_0$ is an extension of $\pi'$, hence is irreducible.
\begin{lem}
\label{lem a2}
Let $\tau$ be an extension of $\pi'$ to $\mathcal{T}(\pi)\cap\rH(\pi)$. Assume that $\tau$ has the same central character as $\pi_0$. Then there exists a $k$-quasicharacter $\phi$ of $F^{\times}$ such that $\tau$ is equivalent to $\pi_0\otimes\phi\circ\det$.
\end{lem}

\begin{proof}
Let $\pi_0'$ be $\pi_0\vert_{Z\rG'}$, since the quotient group $\rG\slash Z\rG'$ is compact, the assumption implies that $\tau$ is a sub-representation of $\ind_{Z\rG'}^{\mathcal{T}(\pi)\cap\rH(\pi)}\pi_0'$, which is equivalent to the tensor product $\pi_0\otimes\ind_{Z\rG'}^{\mathcal{T}(\pi)\cap\rH(\pi)}\mathds{1}$. The induction $\ind_{\mathrm{Z}\rG'}^{\mathcal{T}(\pi)\cap\rH(\pi)}\mathds{1}$ is isomorphic to a direct sum of $k$-representations of finite length. Moreover, an irreducible sub-quotient of $\ind_{\mathrm{Z}\rG'}^{\mathcal{T}(\pi)\cap\rH(\pi)}\mathds{1}$ is isomorphic to $\phi\circ\det$, where $\phi$ is a $k$-quasicharacter of $F^{\times}$. 
We conclude that $\ind_{Z\rG'}^{\mathcal{T}(\pi)\cap\rH(\pi)}\pi_0'\cong\oplus_{i\in I}\tau_i$, where $\tau_i$ has finite length, and any irreducible sub-quotient of $\tau_i$ is equivalent to $\pi_0\otimes\phi\circ\det$, which implies the result.
\end{proof}

\begin{prop}
\label{prop a1}
Let $(\pi,V)$ be an irreducible $k$-representation of $\rG$ and $\pi'$ an irreducible sub-representation of $\pi\vert_{\rG'}$. Let $d$ denote the multiplicity of $\pi'$ in $\pi$. Then the intersection $\mathcal{T}(\pi)\cap\rH(\pi)$ is contained in $\mathcal{S}(\pi)$ with index $d^2$.
\end{prop}

\begin{proof}
It is clear that $\pi$ contains $\pi_0$. Assume that $\pi$ contains an extension $\pi_2$ of $\pi'$ to $\mathcal{T}(\pi)\cap\rH(\pi)$, and is different from $\pi_0$, then $\pi_2$ has the same central character as $\pi_0$ and $\pi_2$ is isomorphic to $\pi_0\otimes\phi\circ\det$ by Lemma \ref{lem a2}. Now adjusting $\phi$ by a $k$-quasicharacter of $F^{\times}$ which is trivial on $\det(\mathcal{T}(\pi)\cap(\rH(\pi))$, we deduce from Frobenius reciprocity that $\pi$ contains $\pi_1\otimes\phi\circ\det$, which is an extension of $\pi_2\otimes\phi\circ\det$. Then by the definition of $\rH(\pi)$, we have $\pi\cong\pi\otimes\phi\circ\det$, but $\phi\circ\det$ is non-trivial on $\mathcal{T}(\pi)\cap\rH(\pi)$, which contradicts with the definition of $\mathcal{T}(\pi)$. Hence $\pi_0$ is the unique extension of $\pi'$ occurring in $\pi$.

Now we consider
$$\ind_{\mathcal{T}(\pi)\cap\rH(\pi)}^{\rH(\pi)}\pi_0\cong\pi_1\otimes\ind_{\mathcal{T}(\pi)\cap\rH(\pi)}^{\rH(\pi)}\mathds{1}.$$
By Lemma \ref{lem a9}, the orders of elements inside the quotient group $\rH(\pi)\slash\mathcal{T}(\pi)\cap\rH(\pi)$ are prime to $\ell$. Hence the $k$-representation on the right-hand side is semisimple, and we have the equivalence
$$\ind_{\mathcal{T}(\pi)\cap\rH(\pi)}^{\rH(\pi)}\pi_0\cong\bigoplus_{\chi\in(\rH(\pi)\slash\mathcal{T}(\pi)\cap\rH(\pi))^{\wedge}}\pi_1\otimes\chi\circ\det.$$
Since by the first part of Lemma \ref{lem a9} each $\chi$ can be viewed as the restriction of some element in $\mathcal{G}(\pi)$ and $\ind_{\rH(\pi)}^{\rG}\pi_1$ is equivalent to $\pi$, the induced representation $\ind_{\mathcal{T}\cap\rH(\pi)}^{\rG}\pi_0$ is a multiple of $\pi$ with multiplicity equal to the cardinality $\vert(\rH(\pi)\slash\mathcal{T}(\pi)\cap\rH(\pi))^{\wedge}\vert$, which is equal to $(\rH(\pi):\mathcal{T}(\pi)\cap\rH(\pi))$ by Lemma \ref{lem a9}. The multiplicity $d$ of $\pi'$ in $\pi$ is equal to the dimension of $\mathrm{Hom}_{\mathcal{T}(\pi)\cap\rH(\pi)}(\pi,\pi_0)$, and the Frobenius reciprocity implies that
$$\mathrm{Hom}_{\mathcal{T}(\pi)\cap\rH(\pi)}(\pi,\pi_0)\cong\mathrm{Hom}_{\rG}(\pi,\ind_{\mathcal{T}(\pi)\cap\rH(\pi)}^{\rG}\pi_0).$$
We conclude that $d=(\rH(\pi):\mathcal{T}(\pi)\cap\rH(\pi))$.

\end{proof}

The Corollary below is an equation of the $\ell$-prime part of the length of $\pi\vert_{\rG'}$.

\begin{cor}
\label{cor a10}
Let $\pi$ be an irreducible $k$-representation of $\rG$.
\begin{enumerate}
\item The length of $\pi\vert_{\rG'}$ is equal to $(\rG:\mathcal{S}(\pi))\cdot(\mathcal{S}(\pi):\mathcal{T}(\pi)\cap \rH(\pi))^{\frac{1}{2}}$,
\item The restriction $\pi\vert_{\rG'}$ is multiplicity-free if and only if $\mathcal{S}(\pi)\subset\mathcal{T}(\pi)$,
\item Assume the restriction $\pi\vert_{\rG'}$ is multiplicity-free, then we have an equation $lg(\pi\vert_{\rG'})_{\ell'}=\vert\mathcal{G}(\pi)\vert$, where $lg(\pi\vert_{\rG'})$ denotes the length of $\pi\vert_{\rG'}$.
\item If $\pi'$ is an irreducible sub-representation of $\pi\vert_{\rG'}$, then there is a unique irreducible representation $\pi_0$ of $\mathcal{T}(\pi)\cap \rH(\pi)$ which contains $\pi'$ on $\rG'$ and occurs in $\pi\vert_{\mathcal{T}(\pi)\cap\rH(\pi)}$. Moreover, $\pi_0\vert_{\rG'}\cong\pi'$.
\end{enumerate}
\end{cor}

\begin{proof}
We have shown that the multiplicity $d=(\mathcal{S}(\pi):\rH(\pi))$. Parts $1,2,4$ are deduced from Proposition \ref{prop a1}. For part $3$, since $\pi\cong\ind_{\rH(\pi)}^{\rG}\pi_1$, the length of $\pi\vert_{\rG'}$ is equal to $(\rG:\rH(\pi))$. When $d=1$, $\mathcal{S}(\pi)=\rH(\pi)$ and by Proposition \ref{prop a1} we have $\mathcal{S}(\pi)\subset\mathcal{T}(\pi)$. We conclude that when $d=1$,
$$lg(\pi\vert_{\rG'})=(\rG:\rH(\pi))=(\rG:\mathcal{T}(\pi))\cdot(\mathcal{T}(\pi):\rH(\pi)).$$
Lemma \ref{lem a4} and Lemma \ref{lem a9} imply that $lg(\pi\vert_{\rG'})_{\ell'}=(\rG:\mathcal{T}(\pi))$ and $lg(\pi\vert_{\rG'})_{\ell}=(\mathcal{T}(\pi):\mathcal{S}(\pi))$.

\end{proof}

In Section \ref{section 003.2}, when $\pi$ is cuspidal, we need the results above to give an equation on the $\ell$-prime part of the length of $\pi\vert_{\rG'}$ through type theory as in \cite{BuKuII}. Now we consider a similar results for pro-finite groups. The results below is required in the study of the the maximal simple $k$-types contained in $\pi$. 

\begin{lem}
\label{lem a13}
Let $A$ be a pro-finite abelian group, and $K$ a closed subgroup of $A$. Let $\chi$ be a smooth $k$-character of $K$, then $\chi$ can be extended to $A$
\end{lem}

\begin{proof}
The kernel $\mathrm{Ker}(\chi)$ is open in $K$. Hence there exists an open subgroup $K_0$ of $A$ such that $K_0\cap K\subset \mathrm{Ker}(\chi)$. We have $K_1=\mathrm{Ker}(\chi)\cdot K_0$ is an open subgroup of $A$ such that $K\cap K_1=\mathrm{Ker}(\chi)$. By the equivalence $K\slash\mathrm{Ker}(\chi)\cong K\cdot K_1\slash K_1$, $\chi$ can be viewed as a smooth character of $K\cdot K_1\slash K_1$. Notice that $A\slash K_1$ is a finite abelian group, which is a direct sum of finite cyclic groups, and each element in $A$ normalises $\chi$. By repeating Lemma \ref{lem a11} we obtain the result.
\end{proof}

\begin{prop}
\label{prop a14}
Let $\rG$ be a profinite group, and $\rN$ a closed normal subgroup of $\rG$ such that $\rG\slash\rN$ is abelian. Assume that there exists an open subgroup of $\rG$, whose pro-order is invertible in $k^{\ast}$. Let $\rho$ be an irreducible smooth representation of $\rG$, and $\rho'$ an irreducible component of $\rho\vert_{\rN}$. Define
$$\mathcal{G}(\rho)=\{\phi\in(\rG\slash\rN)^{\wedge}:\rho\otimes\phi\cong\rho\},$$
$$\mathcal{T}(\rho)=\bigcap_{\phi\in(\rG\slash\rN)^{\wedge}}\mathrm{Ker}(\phi).$$
The subgroup $\mathcal{T}(\rho)$ is open in $\rG$. Define $\rH(\rho)$ by the same manner of $\rH(\pi)$ as above. Then $\rH(\rho)$ is open of $\rG$ containing $\rN$, such that
$$(\mathcal{T}(\rho)\slash\mathcal{T}(\rho)\cap\rH(\rho))_{\ell}=\mathcal{T}(\rho)\slash\mathcal{T}(\rho)\cap\rH(\rho),$$
$$(\rH(\rho)\slash\mathcal{T}(\rho)\cap\rH(\rho))_{\ell'}=\rH(\rho)\slash\mathcal{T}(\rho)\cap\rH(\rho).$$
Furthermore, there exists an irreducible $k$-representation $\pi_1$ of $\rH(\pi)$, such that $\ind_{\rH(\pi)}^{\rG}\rho_1$, and $\res_{\rN}^{\rH(\rho)}\rho_1\cong\rho'$. Suppose the restriction $\rho\vert_{\rN}$ is multiplicity-free, then we have an inclusion $\rH(\rho)\subset\mathcal{T}(\rho)$ and $\rho_1$ is the unique extension of $\rho'$ occurring in $\rho\vert_{\rN}$. We have $lg(\rho\vert_{\rN})_{\ell'}=\vert\mathcal{G}(\rho)\vert$, where $lg(\rho\vert_{\rN})$ denotes the length of $\rho\vert_{\rN}$.
\end{prop}

\begin{proof}
The proof is basically the same as the case above. We list some details which need to be modified while applying the proofs. We can define $\mathcal{S}(\rho)$ the same way as in (\ref{equa a12}). Similarly we can show that $\mathcal{S}(\rho)$ and $\rH(\rho)$ are open with finite index. Moreover, $\mathcal{T}(\rho)$ is open with finite index. In fact, the $k$-representation $\rho$ is finitely dimensional. There exists an open subgroup $K$ of $\rG$ stabilising $\rho$, which implies that $\mathcal{G}(\rho)\subset(\rG\slash K)^{\wedge}$. By applying Lemma \ref{lem a11}, we can find an extension $\rho_1$ of $\rho'$ to $\rH(\rho)$. 

Define $\rho_0$ as $\rho_1\vert_{\mathcal{T}(\rho)\cap\rH(\rho)}$, and Propositon \ref{prop a1} is correct. By the assumption, there exists an open normal subgroup $K_0$ of $\rG$ whose pro-order in invertible in $k^{\ast}$, then $K_0\rN$ is an open subgroup, who is the inverse image of $K_0\slash (K_0\cap\rN)$, hence $\rG\slash\rN$ contains an open subgroup whose pro-order is invertible in $k^{\ast}$. Then Lemma \ref{lem a2} can be applied to our case by replacing $\phi\circ\det$ to $\chi\in(\mathcal{T}(\rho)\cap\rH(\rho)\slash\rN)^{\wedge}$, furthermore, we can assume $\chi\in(\rG\slash\rN)^{\wedge}$ according to Lemma \ref{lem a13}. We deduce that $\rho_0$ is the unique extension of $\rho'$ containing in $\rho$. The proof of Lemma $\ref{lem a9}$ is valid, so is the proof of Proposition \ref{prop a1}, of which this proposition can be deduced as Corollary \ref{cor a10}.
\end{proof}

We end this section with an example, which is an application of the length formula in Corollary \ref{cor a10}. The length of parabolic induction is always a difficult problem while considering $\ell$-modular representations. Especially to estimate the upper bound, which is a difference comparing to the case of representations of characteristic zero. While in small ranks, many examples can be computed. In the example at the end of \cite{Dat}, the last case is left, and we give an answer of this problem.

\begin{ex}
\label{ex a37}
Let $q$ be the cardinality of the residual field of $F$, and $\alpha$ a square root of $q^{-1}$. Let $\rG'=\mathrm{SL}_2(F)$ and $\rG=\mathrm{GL}_2(F)$. Let $\rM$ be the diagonal maximal torus in $\rG$, and $\rM'=\rG\cap\rM$, $\rP$ the group of upper triangular matrices, with $\rP=\rM\rU$. Let $\psi$ be a $k$-quasicharacter on $\rM'$, such that for $m'=(\varpi_{F},\varpi_{F}^{-1})$ ($\varpi_F$ is an uniformizer of $F$, $m'\in\rM'$) $\psi(m)=q^{-1}$ and $q\equiv -1$ (mod $\ell$). When $\ell\neq 2$, the induced representation $i_{\rM'}^{\rG'}\psi$ has length $4$, and when $\ell=2$, the length is $6$. 
\end{ex}

\begin{proof}

As explained in the last example of \cite{Dat}, the induced representation $i_{\rM'}^{\rG'}\psi$ should have length $4$ or $6$. We compute the length in two ways: first we apply Corollary \ref{cor a10} with a condition that $\ell\neq 2$, second we compute through the theory of types.

$\mathbf{Method}$ $\mathbf{1}$ ($\ell\neq 2$): Assume that $\ell\neq 2$, and $q^{-1}\equiv -1$ (mod $\ell$). Let $\delta_{\rP}$ be the $\mathrm{mod}$ character of $\rP$, and $\mu$ be an $k$-quasicharacter of $\rM$, such that for an diagonal element $m=(a,b)\in\rM$, $\mu(m)=\alpha^{\mathrm{val}(a\slash b)}$, which can be extended to $\rP$ through $\rP\rightarrow\rM$. We have $\mu^{2}=\delta_{\rP}$ and $\psi=\mu\vert_{\rM'}$. Vign\'eras proved that $i_{\rM}^{\rG}\mu$ has length $3$, with a quotient character and a sub-character. By geometric lemma, we have
$$i_{\rM'}^{\rG'}\psi\cong \res_{\rG'}^{\rG}i_{\rM}^{\rG}\mu.$$
Let $\pi$ be the cuspidal sub-quotient of $i_{\rM}^{\rG}\mu$, the length of $\res_{\rG'}^{\rG}\pi$ should have length $2$ or $4$ as explained in \cite{Dat}. Now we apply Corollary \ref{cor a10} to $\pi$. The restriction $\res_{\rG'}^{\rG}\pi$ is multiplicity-free, hence $\vert\mathcal{G}(\pi)\vert=lg(\res_{\rG'}^{\rG}\pi)_{\ell'}$. Under the assumption $\ell\neq 2$, the equation above should be $\vert\mathcal{G}(\pi)\vert=lg(\res_{\rG'}^{\rG}\pi)$. 

It's left to compute $\mathcal{G}(\pi)$. In fact, for a $k$-quasicharacter $\chi$ of $F^{\times}$, $\pi\otimes\chi\circ\det$ is a sub-quotient of $i_{\rM}^{\rG}\mu\otimes\chi\circ\det$. If $\pi\cong\pi\otimes\chi\circ\det$, by the uniqueness of supercuspidal support, we obtain that $(\rM,\mu)$ and $(\rM,\mu\otimes\chi\circ\det)$ should belong to the same $\rG$-conjugacy class. There are only two element in the Weyl group $W_{\rG}$ relative to $\rM$. We have $W_{\rG}=\{1,w\}$, where $w(a,b)=(b,a)$, for any element $(a,b)\in\rM$. Hence $(\rM,\mu)$ and $(\rM,\mu\otimes\chi\circ\det)$ belong to the same $\rG$-conjugacy class, either $\mu\cong\mu\otimes\chi\circ\det$ or $w(\mu)\cong\mu\otimes\chi\circ\det$. For the first case, we deduce that $\chi=\mathds{1}$. For the second case, we have $w(\mu)=\mu^{-1}$. Let $\theta$ be a $k$-quasicharacter of $\rM$, such that $w(\mu)\cong\mu\otimes\theta$, we have $\theta(a,b)=q^{-\mathrm{val}(b)+\mathrm{val}(a)}=q^{\mathrm{val}(ab)}$, since $q^{-1}\equiv q\equiv -1$ (mod $\ell$). Hence $\theta$ factors through determinant. We deduce that $\vert\mathcal{G}(\pi)\vert=2$. Since the Jordan-H\"older components of $i_{\rM}^{\rG}\mu$ are two characters as well as an infinitely dimensional representation $\pi$, we conclude that the length of $i_{\rM'}^{\rG'}\psi$ is $4$.

$\mathbf{Method}$ $\mathbf{2}$ ($p\neq 2$): Now we consider the types: $\pi$ contains $(\mathrm{GL}_2(\mathfrak{o}_F),\mathrm{St})$, where $\mathrm{St}$ is the cuspidal Steinberg representation of $\mathrm{GL}_2(k_{F})$, $\mathfrak{o}_F$ is the ring of integers of $F$, and $k_F$ is the residual field of $F$. By Frobenius reciprocity, we have
\begin{equation}
\label{equationex7}
\res_{\rG'}^{\rG}\pi\cong\ind_{\mathrm{SL}_2(\mathfrak{o}_F)}^{\rG'}\mathrm{St}\vert_{\mathrm{SL}_2(k_F)}\oplus\ind_{\alpha(\mathrm{SL}_2(\mathfrak{o}_F))}^{\rG'}\alpha(\mathrm{St}\vert_{\mathrm{SL}_2(k_F)}),
\end{equation}
where $\alpha=(\omega_{F},1)\in\rM$. We determine the length of $\mathrm{St}\vert_{\mathrm{SL}_2(k_F)}$ by looking at the character table in \cite{BonII} which requires $p\neq 2$.

When $\ell=2$ (hence $p\neq 2$), by \cite[\S 9.4.4]{BonII}, we know that $\mathrm{St}\vert_{\mathrm{SL}_2(k_F)}$ is cuspidal and semisimple with length two, containing two irreducible components $\overline{\mathrm{St}}_{+}^k$ and $\overline{\mathrm{St}}_{-}^k$ (the same notation as in \cite{BonII}), which implies that $(\mathrm{SL}_{2}(\mathfrak{o}_F),\overline{\mathrm{St}}_{\pm}^k)$ are maximal simple $k$-types of $\rG'$ (see Definition \ref{defn 22}). Hence the length $lg(\pi\vert_{\rG'})=4$, and the length of $i_{\rM'}^{\rG'}\psi$ is $6$.

It is worth noticing that when $\ell\neq 2$ and $p\neq 2$, the length of $\pi\vert_{\rG'}$ can be computed through the same equation \ref{equationex7} by applying \cite[\S 9.4.4]{BonII}, which coincides with the computation in Method $1$.

We determine the length of $i_{\rM'}^{\rG'}\psi$ by combining the two methods above.
\end{proof}

\subsection{The index of $\mathcal{G}(\pi)$}
\label{section 003.2}
The third part of Corollary \ref{cor a10} gives an equation between $lg(\pi\vert_{\rG'})_{\ell'}$ and the cardinality $\vert \mathcal{G}(\pi)\vert$, while the latter will be expressed through type theory in this section (Theorem \ref{thm a20}), which will be applied to prove Theorem \ref{thm a26}. 

Let $(J,\lambda)$ be a $k$-maximal simple type of $\rG$, and $t(\lambda)$ the length of $\lambda\vert_{J'}$. We divide the characters in $\mathcal{G}(\pi)$ into two parts: those which are non-trivial on $\mathfrak{o}_{F}^{\times}$ (Proposition \ref{prop a15}, Proposition \ref{prop a16}, Corollary \ref{cor a17}), and those which are unramified (Proposition \ref{prop a18}, Corollary \ref{cor a19}), and we compute the size of these two parts individually.

\begin{prop}
\label{prop a15}
Let $(J,\lambda)$ be a maximal simple $k$-type of $\rG$. There are $(\tilde{J}:J)t(\lambda)_{\ell'}$ distinct $k$-quasicharacters of the group $\det(J)$, such that $\lambda$ and $\lambda\otimes\chi\circ\det$ are weakly intertwine in $\rG$ (Definition \ref{prepdefn 01}).
\end{prop}

\begin{proof}
Let $\mathcal{I}(\lambda)$ be the set of $k$-quasicharacters $\chi$ of $\det(J)$ such that $\lambda$ and $\lambda\otimes\chi\circ\det$ are weakly intertwined in $\rG$. By Proposition \ref{prop 1}, let $\chi\in\mathcal{I}(\lambda)$, then there exists an element $x\in\rU(\rA)$ such that $x$ normalises $J$ and $x(\lambda)\cong\lambda\otimes\chi\circ\det$. Let $\mathcal{G}(\lambda)$ consists with the $k$-quasicharacter $\chi$ such that $\lambda$ is isomorphic to $\lambda\otimes\chi\circ\det$, which is a subgroup of $\mathcal{I}(\lambda)$. We compute firstly the cardinality of $\mathcal{G}(\lambda)$.  Since $\ind_{J'}^{\rG'}\res_{J'}^{J}\lambda$ is a subrepresentation of $\res_{\rG'}^{\rG}\pi$, and the latter is multiplicity-free, we have that $\res_{J'}^{J}\lambda$ is multiplicity-free. Hence by Proposition \ref{prop a14}, the cardinality of $\mathcal{G}(\lambda)$ is equal to $t(\lambda)_{\ell'}$. Proposition \ref{prop 12} and Corollary \ref{cor 13} imply that $(\mathcal{I}(\lambda):\mathcal{G}(\lambda))=(\tilde{J}:J)$. We conclude that $\vert\mathcal{I}(\lambda)\vert=(\tilde{J}: J)t(\lambda)_{\ell'}$.
\end{proof}

\begin{prop}
\label{prop a16}
Let $(J,\lambda)$ be a maximal simple $K$-type in $\rG$, and $\fA$ the attached principal order (see Section \ref{Notation}). Let $\chi$ be a $k$-quasicharacter of $F^{\times}$.
\begin{itemize}
\item Let $\pi$ be an irreducible $k$-representation of $\rG$ containing $\lambda$. Suppose that $\pi\otimes\chi\circ\det\cong\pi$. Then there exists $x\in\tilde{J}(\lambda)$ such that $x(\lambda)\cong\lambda\otimes\chi\circ\det$.
\item If $\lambda$ and $\lambda\otimes\chi\circ\det$ are weakly intertwined in $\rG$, then there exists an unramified $k$-quasicharacter $\phi$ of $F^{\times}$, such that for any irreducible $k$-representation $\pi$ of $\rG$ which contains $\lambda$, we have $\pi\cong\pi\otimes (\chi\phi)\circ\det$.
\end{itemize}
\end{prop}

\begin{proof}
If $\pi$ contains $\lambda$ and $\lambda\otimes\chi\circ\det$, they are intertwined in $\rG$. By Proposition \ref{prop 1}, there exists $x\in\tilde{J}$, such that $x(\lambda)\cong\lambda\otimes\chi\circ\det$.

For the second part, by assumption there exists $x\in\tilde{J}$ such that $x(\lambda)\cong\lambda\otimes\chi\circ\det$. Suppose that $\Lambda$ is an extension of $\lambda$ to $E^{\times}J$ such that $\ind_{E^{\times}J}^{\rG}\Lambda\cong\pi$. We have $\ind_{E^{\times}J}^{\rG}x(\Lambda)\cong\pi$ and $\res_{J}^{E^{\times}J}x(\Lambda)\cong\lambda\otimes\chi\circ\det$. Hence $x(\Lambda)$ and $\Lambda\otimes\chi\circ\det$ are two extensions of $\lambda\otimes\chi\circ\det$. It is sufficient to construct an unramified $k$-quasicharacter $\phi$ of $F^{\times}$ such that $\Lambda\otimes(\chi\phi)\circ\det\cong x(\Lambda)$. 

Let $i$ be the canonical isomorphism of $F^{\times}$ to the centre of $\rG$, and $\chi_0$ the central character of $\pi$. We have $i(F^{\times})\cap J\cong i(\mathfrak{o}_{F}^{\times})$, and $x(\Lambda)\vert_{i(\mathfrak{o}_{F}^{\times})}$ is a multiple of $\chi_0$. Meanwhile the restriction $\Lambda\otimes\chi\circ\det\vert_{i(\mathfrak{o}_F^{\times})}$ is a multiple of $\chi_0\otimes\chi\circ\det$, which implies that $\chi_0\otimes\chi\circ\det\vert_{i(\mathfrak{o}_F^{\times})}\cong\chi_0\vert_{i(\mathfrak{o}_F^{\times})}$. By modifying an unramified $k$-character $\theta_1$ of $F^{\times}$, we obtain an extension $\Lambda\otimes (\chi\theta_1)\circ\det$ of $\lambda\otimes\chi\circ\det$, such that $x(\Lambda)\vert_{F^{\times}J}\cong\Lambda\otimes (\chi\theta_1)\circ\det\vert_{F^{\times}J}$.

Furthermore, we have 
$$\Lambda\otimes(\chi\theta_1)\circ\det\hookrightarrow\ind_{F^{\times}J}^{E^{\times}J}x(\Lambda)\vert_{F^{\times}J}\cong x(\Lambda)\otimes\ind_{F^{\times}J}^{E^{\times}J}\mathds{1}.$$
We have an equivalence $E^{\times}J\slash F^{\times}J\cong E^{\times}\slash F^{\times}\mathfrak{o}_{E}^{\times}$, which is a cyclic group of order $e(E\vert F)$. The inclusion above implies that $\Lambda\otimes (\chi\theta_1)\circ\det\cong x(\Lambda)\otimes\theta_2$, where $\theta_2$ is a $k$-quasicharacter of $E^{\times}\slash F^{\times}\mathfrak{o}_{E}^{\times}$. We have that $\theta_2$ factors through $\det$. In fact, the intersection $E^{\times}J\cap\rG'$ is equal to $J'$. Hence $\theta_2$ can be extended to an unramified $k$-quasicharacter of $\rG$ which factors through $\det$. 

We conclude that $\Lambda\otimes(\chi\theta_1\theta_2^{-1})\circ\det\cong x(\Lambda)$, hence $\pi\otimes (\chi\theta_1\theta_2^{-1})\circ\det\cong\pi$, where by construction $\theta_1\theta_2^{-1}$ satisfies the condition for $\phi$ required in the statement.
\end{proof}

\begin{cor}
\label{cor a17}
Let $(J,\lambda)$ be a maximal simple $k$-type of $\rG$. Let $\pi$ be an irreducible $k$-representation containing $\lambda$, and $\mathcal{G}_0(\pi)$ a subset of $\mathcal{G}(\pi)$ consisting of unramified $k$-quasicharacters of $F^{\times}$. Then
$$(\mathcal{G}(\pi):\mathcal{G}_0(\pi))=(\mathfrak{o}_F^{\times}:\det(J))_{\ell'}(\tilde{J}(\lambda):J)t(\lambda)_{\ell'}.$$
\end{cor}

\begin{proof}
First we construct a map
$$\tau_1: (\mathfrak{o}_F^{\times}\slash\det{J})^{\wedge}\rightarrow\mathcal{G}(\pi)\slash\mathcal{G}_0(\pi).$$
Let $\chi\in(\fo_F^{\times}\slash\det(J))^{\wedge}$, which belongs to $(F^{\times})^{\wedge}$ such that $\chi(\mathfrak{p}_F)=1$. We have $\lambda\cong\lambda\otimes\chi\circ\det$. Consider $\tau_1(\chi)$ as the image of $\chi\phi\circ\det$ in $\mathcal{G}(\pi)\slash\mathcal{G}_0(\pi)$, where $\phi$ is an unramified $k$-quasicharacter of $F^{\times}$ satisfying the requirement in Proposition \ref{prop a16}. The map $\tau_1$ is an injective group morphism. Now we consider
$$\mathcal{G}_1:=(\mathcal{G}(\pi)\slash\mathcal{G}_0(\pi))\slash\tau_1((\fo_F^{\times}\slash\det(J))^{\wedge}).$$
Recall that $\mathcal{I}(\lambda)$ be the set of $k$-quasicharacters $\chi$ of $\det(J)$, such that $\lambda\otimes\chi\circ\det$ and $\lambda$ are weakly intertwined in $\rG$. We consider a map $\tau_2:\mathcal{I}(\lambda)\rightarrow\mathcal{G}_1$, which maps $\chi$ to the image of $\chi\phi\circ\det$, then $\tau_2$ is a group isomorphism. In fact, the restriction to $\det(J)^{\times}$ induces a left inverse of $\tau_2$, which must be injective. On the other hand, the second statement of Proposition \ref{prop a16} implies that $\tau_2$ is surjective. We conclude that 
$$(\mathcal{G}(\pi):\mathcal{G}_0(\pi))=\vert(\fo_{F}\slash\det(J))^{\wedge}\vert \cdot\vert\mathcal{I}(\lambda)\vert.$$
Proposition \ref{prop a15} implies the equation we desire.
\end{proof}

Now we consider the set $\mathcal{G}_0(\pi)$. In \cite[\S III,5.7]{V1}, there is a support preserving isomorphism of Hecke algebras:
$$\rH(\rG,\lambda)\cong\rH_k(1,q_E^{f}),$$
where $q_E^{f}$ is the cardinal of the residual field of the field $E$. The algebra $\rH(1,q_E^f)$ is commutative, generated as a $k$-algebra by $[\zeta]$, where $[\zeta]$ corresponding to the characteristic function supported on $J\mathfrak{p}_E^{-1}  J$ in $\rH(\rG,\lambda)$ and $\mathfrak{p}_E$ is a uniformiser of the field extension $E$ associated to the maximal simple $k$-type $(J,\lambda)$. Let $M$ be a simple $\rH(1,q_E^f)$-module, the action of $\rH(1,q_E^f)$ is uniquely defined by $[\zeta]$, which acts on $M$ by a scalar. Let $\alpha\in k$, define $M_{\alpha}$ the simple $\rH(1,q_E^f)$-module, which is isomorphic to $M$ as $k$-vector space through $i$, such that $[\zeta]i(m)=\alpha[\zeta]m$ for any $m\in M$.

Theorem $4.2$ and Proposition $4.8$ in \cite{MS} gives a bijection between the isomorphic classes of irreducible $k$-representations $\pi$ of $\rG$ such that $\Hom_{J}(\lambda,\res_{J}^{\rG}\pi)\neq 0$, and the isomorphic classes of simple $\rH(\rG,\lambda)$-modules. We occupy the notations as \cite{BuKuII}. Denote by $\rM(\pi)$ the corresponding simple $\rH(1,q_E^{f})$-module.

\begin{prop}
\label{prop a18}
Let $\pi$ be an irreducible $k$-representation of $\rG$ containing the maximal simple $k$-type $(J,\lambda)$. Fix a support-preserving algebra isomorphism $\rH(\rG,\lambda)\cong\rH(1,q_E^f)$, and write $M(\pi)$ for the simple $\rH(1,q_E^f)$-module corresponding to $\pi$. Let $\chi$ be an unramified $k$-quasicharacter of $F^{\times}$, and $\mathfrak{p}_F$ a uniformiser of $F$. Then
$$M(\pi\otimes\chi\circ\det)\cong M(\pi)_{\alpha},$$
where $\alpha=\chi(\mathfrak{p}_F)^{n\slash e(E\vert F)}$.
\end{prop}

\begin{proof}
By definition, we have an equivalence of $k$-vector spaces
$$M(\pi)=\Hom_{\rG}(\ind_{J}^{\rG}\lambda,\pi)\cong\Hom_{\rG}(\lambda,\res_{J}^{\rG}\pi\cong\res_{J}^{\rG}\pi\otimes\chi\circ\det)=M(\pi\otimes\chi\circ\det).$$
Let $\iota$ be the above canonical equivalence. We determine the action of $[\zeta]$ on $M(\pi\otimes\chi\circ\det)$. Let $V$ be the representation space of $\lambda$, and $i_v$ the function of value $v\in V$ at the element $1\in\rG$. Denote by $[\zeta]$ the morphism in $\Hom_{\rG}(\ind_J^{\rG}\lambda,\ind_J^{\rG}\lambda)$, determined by $[\zeta] i_v=\mathfrak{p}_E i_v$.

Let $h\in\Hom_{\rG}(\ind_{J}^{\rG}\lambda,\pi)$, and $H$ the element in $\Hom_{J}(\lambda,\res_J^{\rG}\pi)$. We have 
$$(h\circ[\zeta])(i_v)=h(\mathfrak{p}_E i_v)=\pi(\mathfrak{p}_E) h(i_v),$$
where $\pi(\mathfrak{p}_E)h(i_v)=ah(i_v)$, with a scalar $a\in k^{\times}$.
In other words,
$$(H\circ[\zeta])v=aH(v).$$
Let $\chi$ be an unramified $k$-quasicharacter of $F^{\times}$. We have
$$(\iota(H)\circ[\zeta])(i_v)=\chi(\det(\mathfrak{p}_E))a\iota(H)v,$$
which implies that $M(\pi\otimes\chi\circ\det)=M(\pi)_{\chi(\det(\mathfrak{p}_E))}$. Meanwhile, we have an equation $\det(\mathfrak{p}_E)=\mathfrak{p}_F^{n\slash e(E\vert F)}$, which ends the proof.
\end{proof}

\begin{cor}
\label{cor a19}
Let $\chi$ be an unramified $k$-quasicharacter of $F^{\times}$. Then $\chi\in\mathcal{G}(\pi)_0$, if and only if $\chi(\mathfrak{p}_F)^{n\slash e(E\vert F)}=1$.
\end{cor}

\begin{lem}
\label{lem a21}
Let $\fA$ be a hereditary $\mathfrak{o}_F$-order in $\mathrm{End}_F(V)$ and $\mathfrak{K}(\fA)$ be the set of normaliser of $\fA$ in $\rG$, where $V$ is a $n$-dimensional $F$-vector space. Suppose that $E^{\times}\subset\mathfrak{K}(\fA)$, then $\mathfrak{K}(\fA)=\mathfrak{K}(\fB)\rU(\fA)$.
\end{lem}

\begin{proof}
Let $\mathcal{L}$ be the $\mathfrak{o}_F$-lattice chain corresponding to $\fA$, and $L_0,...,L_{e-1}\in\mathcal{L}$, where $e=e(\fA)$ (see Section \ref{Notation}). The group $\mathfrak{K}(\fA)$ is generated by $F^{\times},\rU(\fA)$ and a family of elements $g_1,...,g_{e-1}\in\rG$ such that $g_iL_0=L_i$ for $1\leq i\leq e-1$. Since the field extension $E^{\times}\subset\mathfrak{K}(\fA)$, by \cite[\S 1.2.1]{BuKu} each element in $\mathcal{L}$ is an $\mathfrak{o}_E$-lattice. Let $B=\mathrm{End}_E(V)$, then $\mathfrak{B}=B\cap\fA$ is the hereditary $\mathfrak{o}_E$-order corresponding to $\mathcal{L}$ (viewed as an $\mathfrak{o}_E$-lattice chain). Hence there exists a family of elements $b_i$ in $\mathfrak{K}(\fB)$ such that $b_iL_0=L_i$ for $0\leq i\leq e-1$, which implies that $\mathfrak{K}(\fA)\subset\mathfrak{K}(\fB)\rU(\fA)$. The inverse inclusion is trivial.
\end{proof}

In our case, since $(J,\lambda)$ is maximal simple, we have 
$$\mathfrak{K}(\fA)=E^{\times}\rU(\fA),$$
and
$$\det(\mathfrak{K}(\fA))=\mathfrak{p}_F^{n\slash e(E\vert F)}\fo_F^{\times}.$$
Corollary \ref{cor a19} above implies that $\chi\in\mathcal{G}_0(\pi)$, if and only if $\ker(\chi)\subset\det(\mathfrak{K}(\fA))$, which implies $\mathcal{G}_0(\pi)$ is in bijection with $(F^{\times}\slash\det(\mathfrak{K}(\fA)))^{\wedge}$. We conclude that
\begin{equation}
\label{equa a22}
\vert \mathcal{G}_0(\pi)\vert=(F^{\times}:\det(\mathfrak{K}(\fA)))_{\ell'}.
\end{equation}

\begin{thm}
\label{thm a20}
Let $\pi$ be an irreducible cuspidal $k$-representation of $\rG$. We have the equation
$$\vert\mathcal{G}(\pi)\vert=(F^{\times}:\det(\mathfrak{K}(\mathfrak{B})J))_{\ell'}(\tilde{J}:J)t(\lambda)_{\ell'}.$$
\end{thm}

\begin{proof}
By Corollary \ref{cor a17} and Equation (\ref{equa a22}), we have
\begin{equation}
\label{equa a23}
\vert\mathcal{G}(\pi)\vert=(F^{\times}:\det(\mathfrak{K}(\fA)))_{\ell'}(\rU(\fo_F):\det(J))_{\ell'}(\tilde{J}:J)t(\lambda)_{\ell'}.
\end{equation}
By Definition \S $3.1.8$ and $3.1.14$ of \cite{BuKu}, we have $\rU(\fB)\subset J$. Hence $\det(\rU(\fA))\cap\det(E^{\times}J)$ is equal to $\det(J)$. Lemma \ref{lem a21} implies
$$(\rU(\fo_F):\det(J))=(\det(\mathfrak{K}(\fA)):\mathfrak{K}(\fB)J).$$
We rewrite Equation \ref{equa a23}:
\begin{equation}
\label{equa a24}
\vert\mathcal{G}(\pi)\vert=(F^{\times}:\det(\mathfrak{K}(\fB)J))_{\ell'}(\tilde{J}:J)t(\lambda)_{\ell'}.
\end{equation}
\end{proof}

\subsection{A decomposition of  $\tilde{J}$}
\label{section 003.3}

The following theorem is a generalisation of \cite[Theorem 4.1]{BuKuII} when $(J,\lambda)$ is maximal simple, which is applied in the proof of Theorem \ref{thm a30}.
\begin{thm}
\label{thm a26}
Let $(J,\lambda)$ be a simple type in $\rG$, with $\lambda=\kappa\otimes\sigma$. Let $t(\lambda)$ to be the length of $\lambda\vert_{J'}$. Then
\begin{enumerate}
\item the representation $\lambda\vert_{J'}$ is multiplicity-free;
\item $\tilde{J}=J\tilde{J}'$;
\item an irreducible component $\lambda'$ of $\lambda\vert_{J'}$ extends in $(\tilde{J}:J)$ distinct ways to $\tilde{J}'$.
\end{enumerate}
\end{thm}

\begin{proof}
Part $1$ follows from the proof of Proposition \ref{prop a15}. Let $\pi$ be an irreducible cuspidal $k$-representation of $\rG$ containing $(J,\lambda)$. To prove $2$ and $3$, we need to compute the length of $\pi\vert_{\rG'}$. By Corollary \ref{cor a10} Part $3$ and Theorem \ref{thm a20}, we have $lg(\pi\vert_{\rG'})_{\ell'}=\vert \mathcal{G}(\pi)\vert$.

Let $\rG_1=\mathfrak{K}(\fA)\rG'$. Since the restriction $\pi\vert_{\rG'}$ is multiplicity-free, so is $\pi\vert_{\rG_1}$. We have an inequality $lg(\pi\vert_{\rG_1})\leq (\rG:\rG_1)=(F^{\times}:\det(\mathfrak{K}(\fA)))$. Now we prove that $lg(\pi\vert_{\rG_1})$ divides $(\rG:\rG_1)$. By Corollary \ref{cor a10} we have $\rH(\pi)=\mathcal{S}(\pi)\subset\mathcal{T}(\pi)$. The fact that $\pi\cong\ind_{\rH(\pi)}^{\rG}\pi_1$ (see Remark \ref{rem add01}), implies that $\pi\vert_{\mathcal{T}(\pi)\rG_1}$ is semisimlpe with length $(\rG:\mathcal{T}(\pi)\rG_1)$. For any irreducible component $\tau$ of the restriction $\pi\vert_{\mathcal{T}(\pi)\rG_1}$, its restriction to $\rG_1$ is multiplicity-free (since $\pi\vert_{\rG_1}$ is multiplicity-free). Then the length of $\tau\vert_{\rG_1}$ divides $(\mathcal{T}(\pi)\rG_1:\rG_1)$, in fact if we put $\mathcal{T}_0=\{g\in\mathcal{T}(\pi)\rG_1,x(\tau')\cong\tau'\}$ where $\tau'$ a direct factor of $\tau\vert_{\rG_1}$, then the length $\tau\vert_{\rG_1}=(\mathcal{T}(\pi)\rG_1:\mathcal{T}_0)$. On the other hand, the direct components of $\pi\vert_{\mathcal{T}(\pi)\rG_1}$ are $\rG$-conjugate, hence their length are equivalent after restricting to $\rG_1$. We conclude that $lg(\pi\vert_{\rG_1})=(\rG:\mathcal{T}(\pi)\rG_1)(\mathcal{T}(\pi)\rG_1:\mathcal{T}_0)$, which divides $(\rG:\rG_1)$.  We need the lemma below to continue the proof.
\end{proof}

\begin{lem}
We occupy the notations as before, then we have $(\mathcal{T}(\pi)\rG_1:\rG_1)_{\ell'}=\{\mathds{1}\}$.
\end{lem}

\begin{proof}
In other words, we need to proof the cardinality of the quotient group $\mathcal{T}(\pi)\rG_1\slash\rG_1$ is an $\ell$-power. It is equivalent to prove its $k$-dual is trivial. In fact, let $\chi\circ\det\in(\mathcal{T}(\pi)\rG_1\slash\rG_1)^{\wedge}$. Since $\chi$ is trivial on $\det(E^{\times}J)$, we have $\pi\otimes\chi\circ\det\cong\pi$. Then the kernel of $\chi\circ\det$ contains $\mathcal{T}(\pi)$, hence it is trivial on the quotient $\mathcal{T}(\pi)\rG_1\slash\rG_1$.
\end{proof}

\begin{proof}[Continue the proof of Theorem \ref{thm a26}]
The lemma above implies an equation below:
\begin{equation}
\label{equa a27}
lg(\pi\vert_{\rG_1})_{\ell'}=(\rG:\mathcal{T}(\pi)\rG_1)_{\ell'}=(\rG:\rG_1)_{\ell'}.
\end{equation}
Let $\Lambda$ be an extension of $\lambda$ to $E^{\times}J$ which is contained in $\pi$. The irreducible induction $\pi_1:=\ind_{E^{\times}J}^{\rG_1}\Lambda$ is an irreducible component of $\pi\vert_{\rG_1}$.

Now we need to know the length of the restriction $\pi_1\vert_{\rG'}$. We write:
$$T=\ind_{E^{\times}J}^{\mathfrak{K}(\fA)}\Lambda,$$
$$\tau=\ind_{J}^{\rU(\fA)}\lambda.$$
The representation $T$ is irreducible, and the following lemma will show that $\tau$ is irreducible as well.
\end{proof}

\begin{lem}
The induction $\tau=\ind_{J}^{\rU(\fA)}\lambda$ is irreducible.
\end{lem}

\begin{proof}
The intertwining group $\mathcal{I}_{\rU(\fA)}(\lambda)$ is equal to $J$. It is sufficient to prove that $\lambda$ verifies the second condition of irreducibility (Lemma \ref{lem 19}). Let $\nu$ be an irreducible $k$-representation of $\rU(\fA)$, and $\lambda$ a sub-representation of $\nu\vert_{J}$. We need to show that $\lambda$ is also a quotient of $\nu\vert_{J}$. By Frobenius reciprocity and the exactness of the restriction functor, there is a surjection
\begin{equation}
\label{equa a25}
\res_{J}^{\rU(\fA)}\ind_{J}^{\rU(\fA)}\lambda\rightarrow\nu\vert_{J}.
\end{equation}
We have
$$\res_{J}^{\rU(\fA)}\ind_{J}^{\rU(\fA)}\lambda\cong\bigoplus_{\alpha\in J\backslash\rU(\fA)\slash J}\ind_{J\cap\alpha(J)}^{J}\res_{J\cap\alpha(J)}^{\alpha(J)}\alpha(\lambda).$$
By Corollary \ref{cor 5}, when $\alpha\neq 1$, there is not irreducible sub-quotient of $\ind_{J\cap\alpha(J)}^{J}\res_{J\cap\alpha(J)}^{\alpha(J)}\alpha(\lambda)$ which is isomorphic to $\lambda$. Hence the kernel of the morphism in Equation \ref{equa a25} is contained in the sub-representation $\oplus_{\alpha\neq 1}\ind_{J\cap\alpha(J)}^{J}\res_{J\cap\alpha(J)}^{\alpha(J)}\alpha(\lambda)$, which implies the existence of a surjection from $\nu\vert_{J}$ to $\lambda$.
\end{proof}

\begin{proof}[Continue the proof of Theorem \ref{thm a26}]
In our case, we have $\mathfrak{K}(\fA)=E^{\times}\rU(\fA)$ and $E^{\times}\rU(\fA)\cap\rG'=\rU(\fA)\cap\rG'$, and since we also have $T\vert_{\rU(\fA)}=\tau$, we conclude that
$$\res_{\rG'}^{\rG_1}\pi_1\cong\ind_{\rU(\fA)\cap\rG'}^{\rG'}\res_{\rU(\fA)\cap\rG'}^{\rU(\fA)}\tau.$$
From now till the end of this proof, we abbreviate $\rU(\fA)$ as $\rU$ and $\rU(\fA)\cap\rG'$ as $\rU'$. We have
$$\ind_{\rU'}^{\rG'}\res_{\rU'}^{\rU}\tau\cong\bigoplus_{x\in J\backslash\rU\slash\rU'}x(\ind_{J'=J\cap\rG'}^{\rG'}\res_{J'}^{J}\lambda).$$
The cardinality of the index set $J\backslash\rU\slash\rU'$ is equal to $(\rU(\mathfrak{o}_{F}):\det(J))$. Let $t(\lambda)$ be the length of $\lambda\vert_{J'}$, and $\tilde{\lambda}'$ an irreducible component of $\ind_{J'}^{\tilde{J}'=\tilde{J}\cap\rG'}\lambda'$, then $\ind_{\tilde{J}'}^{\rG'}\tilde{\lambda}'$ is irreducible by Theorem \ref{thm 16}. Hence $lg(\ind_{\rU'}^{\rG'}\res_{\rU'}^{\rU}\tau)$ is equal to $lg(\res_{\rU'}^{\rU}\tau)$ and $lg(\ind_{J'}^{\rG'}\res_{J'}^{J}\lambda)$ is equal to $lg(\ind_{J'}^{\tilde{J}'}\res_{J'}^{J}\lambda)$. Recall that $t(\lambda)$ is the length of $\res_{J'}^J\lambda$, and the irreducible components of the latter are $J$-conjugate. Let $\lambda'$ be one of them. We have $lg(\ind_{J'}^{\tilde{J}'}\res_{J'}^{J}\lambda)$ is equal to $lg(\ind_{J'}^{\tilde{J}'}\lambda')t(\lambda)$. Meanwhile, we have
$$(\tilde{J}':J')=lg(\res_{J'}^{\tilde{J}'}\ind_{J'}^{\tilde{J}'}\lambda').$$
Notice that $\ind_{J'}^{\tilde{J}'}\lambda'$ is a sub-representation of $\res_{\tilde{J}'}^{\tilde{J}}\ind_{J}^{\tilde{J}}\lambda$, whose irreducible components are $\tilde{J}$-conjugate, hence the length of $\res_{J'}^{\tilde{J}'}\tilde{\lambda}'$ is independ of the choice of irreducible component $\tilde{\lambda}'$. We then obtain that
$$(\tilde{J}':J')=lg(\ind_{J'}^{\tilde{J}'}\lambda')lg(\res_{J'}^{\tilde{J}'}\tilde{\lambda}'),$$
which implies $lg(\ind_{J'}^{\tilde{J}'}\lambda')\vert(\tilde{J}':J)$. The equation holds if and only if $\tilde{\lambda}'$ is an extension of a $\lambda'$. We conclude that
$$lg(\tau\vert_{\rU'})\vert(\rU(\mathfrak{o}_F):\det(J))(\tilde{J}':J')t(\lambda).$$
Combining this with Equation \ref{equa a27}, we have
$$lg(\pi\vert_{\rG'})_{\ell'}\vert(F^{\times}:\det(\mathfrak{K}(\fA)))_{\ell'}(\rU(\mathfrak{o}_F):\det(J))_{\ell'}(\tilde{J}':J')t(\lambda)_{\ell'}.$$
On the other hand, by Corollary \ref{cor a10} 3) and Equation \ref{equa a23} we have
\begin{equation}
\label{equa a28}
lg(\pi\vert_{\rG'})_{\ell'}=(F^{\times}:\det(\mathfrak{K}(\fA)))_{\ell'}(\rU(\fo_F):\det(J))_{\ell'}(\tilde{J}:J)t(\lambda)_{\ell'}.
\end{equation}
Hence we deduce that $(\tilde{J}':J')=(\tilde{J}:J)$, and $\tilde{J}=\tilde{J}'J$. Then the analysis above implies that $lg(\ind_{J'}^{\tilde{J}'}\lambda')$ is equal to $(\tilde{J}':J)$, hence $\tilde{\lambda}'$ is an extension of $\lambda'$. The distinction property follows from $\res_{\tilde{J}'}^{\tilde{J}}\tilde{\lambda}$ being multiplicity-free by Proposition \ref{prop 23}.
\end{proof}

A same decomposition of $\tilde{J}_{\rM}$ can be deduced when $\rM$ is a Levi subgroup, which will be used in the proof of Theorem \ref{thm a30}.

\begin{cor}
\label{cor a32}
Let $\rM=\prod_{i=1}^m\mathrm{GL}_{n_i}(F)$ be a Levi subgroup of $\rG$, and $(J_{\rM},\lambda_{\rM})$ a maximal simple $k$-type of $\rM$. Then we have $\tilde{J}_{\rM}=\tilde{J}_{\rM}'J_{\rM}$, and $\tilde{J}_{\rM}'\subset J_{\rM}'\prod_{i=1}^m\tilde{J}_i'$.
\end{cor}

\begin{proof}
This can be deduced from Theorem \ref{thm a26} Part $2$ in the same way as \cite[\S 1.4.1, \S 1.4.2]{GoRo}.
\end{proof}

\subsection{Intertwining implies conjugacy}
\label{section intconj}
Recall that $\rM$ is a Levi subgroup of $\rG$ and $\rM'=\rM\cap\rG'$ a Levi subgroup of $\rG'$. In this section, we prove the unicity property of weakly intertwining implying conjugacy (Theorem \ref{thm a30}). In the $\ell$-modular setting, there is a difference between intertwining and weakly intertwining. The unicity property of simple characters of $\rM'$ will be considered (Theorem \ref{thm a40}), which is a new result for both $\ell$-modular and complex setting. 

This section contains two parts. We study maximal simple $k$-types in Proposition \ref{prop a29} and Theorem \ref{thm a30} (for complex setting see \cite{BuKuII} for $\rG'$, and see \cite{GoRo} for $\rM'$). We study simple $k$-characters since Theorem \ref{thm a37}.

\begin{prop}
\label{prop a29}
Let $(J_{\rM},\lambda_{\rM})$ be a maximal simple $k$-type of $\rM$. Let $\tilde{\lambda}_1'$ and $\tilde{\lambda}_2'$ be two irreducible components of $\tilde{\lambda}_{\rM}\vert_{\tilde{J}_{\rM}'}$, where $\tilde{\lambda}_{\rM}=\ind_{J_{\rM}}^{\tilde{J}_{\rM}}\lambda_{\rM}$. Then $\tilde{\lambda}_1'$ and $\tilde{\lambda}_2'$ are weakly intertwined in $\rM'$, if and only if they are conjugate in $\rM'$. In particular, if $\rM=\rG$, then $\tilde{\lambda}_1'$ and $\tilde{\lambda}_2'$ are distinct, and they are never weakly intertwined in $\rG'$.
\end{prop}

\begin{proof}
If $x\in\rM'$ weakly intertwines $\tilde{\lambda}_1'$ with $\tilde{\lambda}_2'$, then $x$ weakly intertwines $\tilde{\lambda}_{\rM}$ with $\tilde{\lambda}_{\rM}\otimes\theta\circ\det$ for a $k$-quasicharacter $\chi$ of $F^{\times}$. By Mackey's theory, we have
$$\res_{J_{\rM}}^{\tilde{J}_{\rM}}\ind_{\tilde{J}_{\rM}\cap x(\tilde{J}_{\rM})}^{\tilde{J}_{\rM}}\res_{\tilde{J}_{\rM}\cap x(\tilde{J}_{\rM})}^{x(\tilde{J}_{\rM})}\ind_{x(J_{\rM})}^{x(\tilde{J}_{\rM})}\lambda_{\rM}\otimes\chi\circ\det$$
$$\cong\bigoplus_{\alpha,\beta}\ind_{J_{\rM}\cap\beta\alpha x(J)_{\rM}}^{J_{\rM}}\res_{J_{\rM}\cap\beta\alpha x(J)_{\rM}}^{\beta\alpha x(J_{\rM})}\beta\alpha x(\lambda_{\rM})\otimes\chi\circ\det,$$
where $\alpha\in J_{\rM}\backslash\tilde{J}_{\rM}\slash\tilde{J}_{\rM}\cap x(\tilde{J}_{\rM})$ and $\beta\in\alpha x(J)_{\rM}\backslash\alpha x(\tilde{J}_{\rM})\slash J_{\rM}\cap\alpha x(\tilde{J}_{\rM})$. This implies that $\beta\alpha x$ weakly intertwines $\lambda_{\rM}$ with $\lambda_{\rM}\otimes\chi\circ\det$. Then by Proposition \ref{prop 1}, there exists $g\in\tilde{J}_{\rM}$ such that $g(\lambda_{\rM})\cong\lambda_{\rM}\otimes\chi\circ\det$ and by Proposition \ref{prop 1} $\beta\alpha x g$ intertwines $\lambda_{\rM}$ to itself, which means $\beta\alpha xg\in E_{\rM}^{\times}J_{\rM}$ where $E_{\rM}^{\times}=E_1^{\times}\times\cdots\times E_{m}^{\times}$ and $\rM=\mathrm{GL}_{n_1}\times\cdots\times\mathrm{GL}_{n_m}$. We deduces that $x\in E_{\rM}^{\times}\tilde{J}_{\rM}\cap\rM'$, which normalises $\tilde{J}_{\rM}$, hence $\tilde{\lambda}_1'\cong x(\tilde{\lambda}_2')$. In particular, when $\rM=\rG$, we have $E_{\rM}^{\times}\tilde{J}_{\rM}\cap\rM'=\tilde{J}_{\rM}'$, which implies that $\tilde{\lambda}_1'\cong\tilde{\lambda}_2'$.
\end{proof}

\begin{thm}
\label{thm a30}
Let $(J_i,\lambda_i)$ be an maximal simple $k$-type of $\rM$, $(\tilde{J}_i',\tilde{\lambda}_i')$ be an irreducible component of $\tilde{\lambda}_i\vert_{\tilde{J}_i'}$ for $i=1,2$ respectively. Suppose that $(\tilde{J}_i',\tilde{\lambda}_i')$ are weakly intertwined in $\rM'$, then they are conjugate in $\rM'$.
\end{thm}

\begin{proof}
If $x\in\rM'$ weakly intertwines $(\tilde{J}_1',\tilde{\lambda}_1')$ with $(\tilde{J}_2',\tilde{\lambda}_2')$, we can assume that (up to twist a $k$-quasicharacter of $F^{\times}$ on $\tilde{\lambda}_2$) $x$ weakly intertwines $(\tilde{J}_1,\tilde{\lambda}_1)$ with $(\tilde{J}_2,\tilde{\lambda}_2)$ . As in the proof of Proposition \ref{prop a29}, this implies that $\beta\alpha x$ weakly intertwines $\lambda_1$ with $\lambda_2$, for $\alpha\in\tilde{J}_1$ and $\beta\in\alpha x(\tilde{J}_2)$. Hence there is an element $g\in\rM$ such that $g(J_2)=J_1$ and $g(\lambda_2)\cong \lambda_1$ by Proposition $\mathrm{IV},1.6 $ $2)$ in \cite{V2}. Furthermore, $g^{-1}\beta\alpha x$ weakly intertwines $\lambda_2$ to itself, hence $g^{-1}\beta\alpha x$ intertwines $\lambda_2$ to itself. Then $g^{-1}\beta\alpha x\in E_{\rM}^{\times}J_2$ (see the proof of Proposition \ref{prop a29} for $E_{\rM}^{\times}$), and $x\in \tilde{J}_1gE_{\rM}^{\times}\tilde{J}_2\cap\rM'=gE_{\rM}^{\times}\tilde{J}_2 \cap\rM'$. By Corollary \ref{cor a32} we have $\tilde{J}_2=\tilde{J}_2'J_2$. After adjusting an element in $E_{\rM}^{\times}J_2$ to the right of $g$, we can assume that $g\tilde{J}_2'\cap\rM'$ is non-trivial, hence $g\in\rM'$. Then we have $g(\tilde{J}_2)=\tilde{J}_1$ and $g(\tilde{\lambda}_2)\cong\tilde{\lambda}_1$. Therefore $g(\tilde{\lambda}_2')$ is an irreducible component of $\tilde{\lambda}_1\vert_{\tilde{J}_1'}$, and $\tilde{\lambda}_1'$ is weakly intertwined with $g(\tilde{\lambda}_2')$ in $\rM'$. By Proposition \ref{prop a29}, this ensures that $\tilde{\lambda}_2'$ and $\tilde{\lambda}_1'$ are conjugate in $\rM'$.
\end{proof}

\begin{defn}
A simple $k$-character in $\rM'$ is a $k$-character of $H^{1'}=H^1\cap\rM'$ of the form  $\theta_{\rM}\vert_{H^{1'}}$, where $(H^1,\theta_{\rM})$ is a simple $k$-character in $\rM$.
\end{defn}

The simple $k$-characters in $\rM$ is in bijection with the simple $\cK$-characters in $\rM$, and the simple $k$-characters of $\rM'$ are given by the former. The above two relations are the base to study simple $k$-characters.

\begin{thm}
\label{thm a37}
Let $\pi$ be an irreducible cuspidal $k$-representation of $\rG=\mathrm{GL}_n(F)$, and $(H_i^{m_i+1},\theta_i)$ for $i=1,2$ be two simple $k$-characters contained in $\pi$, where $[\fA_i,n_i,m_i,\beta_i]$ be the corresponding simple strata such that $\theta_i\in\mathcal{C}(\fA_i,m_i,\beta_i)$. We have that $\fA_1\cong\fA_2$ as $\mathfrak{o}_{F}$-hereditary orders, $m_i=0$ and $(H_i^{1},\theta_i)$ are $\rG$-conjugate.
\end{thm}

\begin{proof}
We give a proof by passing to the complex setting, and we require the parallel result for complex representations in \cite{BuKu}. For this reason, we need the following lemma on a property of lifting and reduction modulo $\ell$.
\end{proof}

\begin{lem}
\label{lem a38}
Let $\pi$ be an irreducible $k$-representation of $\rG$ containing a simple $k$-character $(H^{m+1},\theta)$. Let $\pi_{\cK}$ be an irreducible $\cK$-representation of $\rG$ which is $\ell$-integral, and its reduction modulo $\ell$ contains $\pi$ as a sub-quotient. Let $(H^{m+1},\theta_{\cK})$ be the $\cK$-lifting of $(H^{m+1},\theta)$. Then $\theta_{\cK}$ is a direct component of $\pi_{\cK}\vert_{H^{m+1}}$.
\end{lem}

\begin{proof}
Since $H^{m+1}$ is a pro-$p$ subgroup for $m\in\mathbb{N}$, the reduction modulo $\ell$ gives a bijection between the set of simple $\cK$-characters and the set of simple $k$-characters in $\rG$. Let $O_{W(k)}$ be a $W(k)[\rG]$-lattice contained in $\pi_{\cK}$. Since $H^{m+1}$ is a pro-$p$ compact subgroup of $\rG$, there is a surjection from $O_{W(k)}\vert_{H^{m+1}}$ to $\theta$. On the other hand, let $\theta_{\cK}$ be the $\cK$-lifting of $\theta$, and $O_0$ be a $W(k)[H^{m+1}]$-lattice inside $\theta_{\cK}$. The lattice $O_0$ is projective in the category of smooth $W(k)[H^{m+1}]$-modules. 
Hence there is a non-trivial morphism from $O_0$ to $O_{W(k)}\vert_{H^{m+1}}$, which is injective since $O_0$ is rank $1$ and $O_{W(k)}$ is torsion-free. We conclude that 
$$\theta_{\cK}\cong O_0\otimes_{W(k)}\cK\hookrightarrow O_{W(k)}\vert_{H^{m+1}}\otimes_{W(k)}\cK\cong\pi_{\cK}\vert_{H^{m+1}}$$
as we desired.
\end{proof}

\begin{proof}[Continue the proof of Theorem \ref{thm a37}]
The result can be deduced from the lemma above and \cite[\S 8]{BuKu}. As in \cite[\S III,5.10]{V1}, there is an irreducible $\ell$-integral cuspidal $\pi_{\cK}$, of which the reduction modulo $\ell$ is isomorphic to $\pi$. Let $\theta_{i,\cK}$ be the 
$\cK$-lifting of $\theta_i$ for $i=1,2$. According to \cite[\S 3]{BuKu}, a simple $\cK$-character $\theta_{i,\cK}$ is defined over a simple stratum $[\fA_i,n_i,m_i,\beta_i]$. By the proof of Theorem 8.5.1 of \cite{BuKu}, if $m_i\neq 0$, then $\pi_{\cK}$ contains a split type, which contradicts with the cuspidality of $\pi_{\cK}$ according to \cite[8.2.5,8.3.3]{BuKu}. Since for both $i=1,2$ we have $m_i=0$. Let $J(\beta_i)$ be as in Section \ref{Notation}, then $\pi_{\cK}\vert_{J(\beta_i)}$ contains an irreducible $\cK$-representation $\tau_i$ such that $\tau_i\vert_{J^{1}(\beta_i)}$ contains the Heisenberg representation of $\theta_i$. Then $\tau$ must be of the form $\kappa_i\otimes\xi_i$, where $\kappa_i$ is an $\beta$-extension of $\theta_i$ and $\xi_i$ is inflated from an irreducible $\cK$-representation $\overline{\xi_i}$ of $J(\beta_i)\slash J^1(\beta_i)$ which is isomorphic to a direct product of finite groups of $\mathrm{GL}(k_j^i),k_j^i\in\mathbb{N}$. If $\overline{\xi_i}$ is not cuspidal, by the proof of \cite[Theorem 8.1.5]{BuKu}, $\pi_{\cK}$ contains a split type or a simple type $(J(\beta_3),\lambda_3)$ defined over a simple stratum $[\fA_3,n_3,0,\beta_3]$, such that $J(\beta_3)\slash J^1(\beta_3)$ is a proper Levi of a finite group of $GL_{k_3},k_3\in\mathbb{Z}$. We have explain above the it is impossible to contain a split type. Since $\pi_{\cK}$ contain a maximal simple $\cK$-type, \cite[Theorem 6.2.4]{BuKu} deduces a contradiction with the non-maximality of $(J(\beta_3),\lambda_3)$, which implies that $\overline{\xi}_i$ are both cuspidal. Hence $(J(\beta_i),\kappa_i\otimes\xi_i)$ are both simple types. By \cite[Lemma 6.2.5, Theorem 6.2.4]{BuKu}, we have $\fA_1\cong\fA_2$ as $\mathfrak{o}_F$-hereditary orders, and there exists $g\in\rG$ such that $g(J(\beta_1))=J(\beta_2)$ and $g(\kappa_1\otimes\xi_1)\cong\kappa_2\otimes\xi_2$. Then we obtain that $g(\theta_{1,\cK})\cong\theta_{2,\cK}$, which is equivalent to $g(\theta_1)\cong\theta_2$ after reduction modulo $\ell$.
\end{proof}

\begin{rem}
\label{rem a39}
\begin{enumerate}
\item
Let $\pi$ be irreducible and cuspidal of $\rG$. Let $(J_1,\lambda_1)$ and $(J_2,\lambda_2)$ be two maximal simple $k$-types appear as sub-quotients of $\pi$. The result of Theorem \ref{thm a37} combining with the parahoric restriction of  of \cite[Corollary 8.4,(1)]{V3} gives another proof of \cite[Proposition IV,1.6,(2)]{V2}.
\item Let $\pi_{\rM}$ be an irreducible cuspidal $k$-representation of $\rM$, a Levi subgroup of $\rG$, and $(H_{i,\rM},\theta_{i,\rM})$ for $i=1,2$ be two simple $k$-characters contained in $\pi_{\rM}$. Let $\fA_{i,\rM}$ be the direct product of $\mathfrak{o}_{F}$-hereditary orders in the definition of $\theta_{i,\rM}$. We deduce from Theorem \ref{thm a37} that $\fA_{1,\rM}\cong\fA_{2,\rM}$, and $(H_{i,\rM},\theta_{i,\rM})$ are $\rM$-conjugate.
\item Let $\pi_{\rM}$ be as above, and $(J_{i,\rM},\lambda_{i,\rM})$ be simple $k$-types of $\rM$ appears as sub-quotients of $\pi_{\rM}$ for $i=1,2$. Part $2$ implies that they are maximal and belong to the same $\rM$-conjugacy class.
\end{enumerate}
\end{rem}

Now we move on to the case of $\rM'$ and prove the desired result, which is the unicity of simple $k$-characters contained in a fixed cuspidal $k$-representation of $\rM'$.
\begin{thm}
\label{thm a40}
Let $\pi'$ be an irreducible cuspidal $k$-representation of $\rM'$ which is a Levi subgroup of $\rG'=\mathrm{SL}_n(F)$. 
\begin{enumerate}
\item
Let $i=1,2$. Assume that $(\tilde{J}_i',\tilde{\lambda}_i')$ is a maximal simple $k$-type which appears as a sub-quotient of $\pi'\vert_{\tilde{J}_i'}$, then there exists $g\in\rM'$, such that $g(\tilde{J}_1')=\tilde{J}_2'$ and $g(\tilde{\lambda}_1')\cong\tilde{\lambda_2'}$.
\item
Assume $(H_i',\theta_i')$ for $i=1,2$ be two simple $k$-characters contained in $\pi'$, then there exists $h\in\rM'$, such that $h(H_1')=H_2'$ and $h(\theta_1')\cong\theta_2'$.
\end{enumerate}
\end{thm}

\begin{proof}
For part $1$. Since $\pi'$ must contains a maximal simple $k$-type $(\tilde{J}_0',\tilde{\lambda}_0')$ as a sub-representation. By Frobenius reciprocity and the exactness of the restriction functor, there is a surjection from $\res_{\tilde{J}_i'}^{\rM'}\ind_{\tilde{J}_0'}^{\rM'}\tilde{\lambda}_0'$ to $\res_{\tilde{J}_i'}^{\rM'}\pi'$, which implies that $\tilde{\lambda}_i'$ is weakly intertwined with $\tilde{\lambda}_0'$ in $\rM'$ for $i=1,2$. They are conjugate in $\rM'$ by Theorem \ref{thm a30}.

For part $2$, our strategy is to prove that there exists $k$-simple type $(\tilde{J}_i',\tilde{\lambda}_i')$ for $i=1,2$ of $\rM'$, which contains $\theta_i'$ and appears in $\pi'$ as a sub-quotient, then we obtain the result by applying part $1$. Let $\pi$ be an irreducible cuspidal $k$-representation of $\rM$, whose restriction on $\rM'$ contains $\pi'$ as a sub-representation. Since there is an irreducible sub-quotient $\tau_i'$ of $\pi'\vert_{\tilde{J}_i'}$ such that $\tau_i'\vert_{H_i'}$ contains $\theta_i'$. By Lemma \ref{lem 9}, there is an irreducible sub-quotient $\tau_i$ of $\pi\vert_{\tilde{J}_i}$ such that $\tau_i\vert_{\tilde{J}_i'}$ contains $\tau_i'$. Hence $\tau_i\vert_{H_i}$ must contains an $k$-character $\theta_i$, furthermore $(H_i,\theta_i)$ is a simple $k$-character in $\rM$ whose restriction on $H_i'$ is isomorphic to $\theta_i'$. There is an irreducible sub-quotient $\lambda_i$ of $\tau_i\vert_{J_i}$ whose restriction on $H_i$ contains $\theta_i$. By \cite[\S III, 4.18]{V1}, a maximal simple $\cK$-type is always $\ell$-integral, and its reduction modulo $\ell$ is a maximal simple $k$-type in $\rM$, which implies that the proof of Theorem \ref{thm a37} can be applied here, and we conclude that$(J_i,\lambda_i)$ is a maximal simple $k$-type in $\rM$. Since $\pi$ must contains a maximal simple $k$-type as a sub-representation, the first part of Remark \ref{rem a39} implies that $(J_i,\lambda_i)$ appears as a sub-representation of $\pi\vert_{J_i}$ for $i=1,2$. Applying \cite[Corollary 8.4]{V3}, we know that $\lambda_i$ is a sub-representation of $\tau_i\vert_{J_i}$. Hence we deduced from Corollary \ref{cor 99} that $\tau_i\cong\ind_{J_i}^{\tilde{J}_i}\lambda_i=\tilde{\lambda}_i$. We conclude that $\tau_i'\hookrightarrow\tilde{\lambda}_i\vert_{\tilde{J}_i'}$ is a maximal simple $k$-type in $\rM$. We end the proof by applying part $1$. 

\end{proof}

\section{Extended maximal simple $k$-types of $\rM'$}
\label{chapter 02}
Recall that $\rM$ is a Levi subgroup of $\rG$. Write $\rM\cong\mathrm{GL}_{n_1}\times\cdots\times\mathrm{GL}_{n_m}$, for $m\in\mathbb{Z}$. Let $(J_{\rM},\lambda_{\rM})$ be a maximal simple $k$-type of $\rM$, and $\tilde{J}_{\rM}$ the group of projective normalisers of $(J,\lambda)$. Let $\tilde{\lambda}_{\rM}'$ be an irreducible component of $\tilde{\lambda}_{\rM}=\ind_{J_{\rM}}^{\tilde{J}_{\rM}}\lambda_{\rM}$, and $N_{\rM'}(\tilde{\lambda}_{\rM}')$ be the normaliser of $\tilde{\lambda}_{\rM}'$ in $\rM'$. We have, as in Remark \ref{Lrem aa1}, $N_{\rM'}(\tilde{\lambda}_{\rM}')$ is a subgroup of $E_{\rM}^{\times}\tilde{J}_{\rM}\cap\rM'$.

\begin{lem}
\label{lem a35}
Let $(J_{\rM},\lambda_{\rM})$ be as above. Then 
\begin{enumerate}
\item $E_{\rM}^{\times}\tilde{J}_{\rM}\cap\rM'$ is compact modulo centre. Furthermore, the quotient $E_{\rM}^{\times}\tilde{J}_{\rM}\cap\rM'\slash N_{\rM'}(\tilde{\lambda}_{\rM}')$ is finite and abelian.

\item Denote $\ind_{E_{\rM}^{\times}J_{\rM}}^{E_{\rM}^{\times}\tilde{J}_{\rM}}\Lambda_{\rM}$ by $\tilde{\Lambda}_{\rM}$. Then $\tilde{\Lambda}_{\rM}$ is an extension of $\tilde{\lambda}_{\rM}$.
\item The group $N'=N_{\rM'}(\tilde{\lambda}_{\rM}')$ is normal in $E_{\rM}^{\times}\tilde{J}_{\rM}$. In particular, $N'$ is independent of the choice of irreducible component $\tilde{\lambda}_{\rM}'$. Furthermore, for any irreducible $k$-representation of $N'$, if it contains $\tilde{\lambda}_{\rM}'$ as a sub-representation, then it is a multiple of $\tilde{\lambda}_{\rM}'$ after restricted to $\tilde{J}_{\rM}'$.
\item An irreducible $k$-representation of $E_{\rM}^{\times}\tilde{J}_{\rM}$ that contains $\tilde{\lambda}_{\rM}$ is an extension of $\tilde{\lambda}_{\rM}$. In particular, an extension of $\tilde{\lambda}_{\rM}$ can be written as $\ind_{E_{\rM}^{\times}J_{\rM}}^{E_{\rM}^{\times}\tilde{J}_{\rM}}\Lambda_{\rM}$.
\end{enumerate}
\end{lem}

\begin{proof}
For Part $1$, the property that $E_{\rM}^{\times}\tilde{J}_{\rM}\cap\rM'$ is compact modulo centre follows from the facts that $E_{\rM}^{\times}\tilde{J}_{\rM}\slash Z_{\rM}$ is compact and that $E_{\rM}^{\times}\tilde{J}_{\rM}\cap\rM' \slash Z_{\rM}\cap\rM'$ is closed inside $E_{\rM}^{\times}\tilde{J}_{\rM}\slash Z_{\rM}$, where $Z_{\rM}$ is the centre of $\rM$. Recall that the centre $Z_{\rM'}$ of $\rM'$ is equal to $Z_{\rM}\cap\rM'$. Since 
$$E_{\rM}^{\times}\tilde{J}_{\rM}\cap\rM'\slash N_{\rM'}(\tilde{\lambda}_{\rM}')\cong (E_{\rM}^{\times}\tilde{J}_{\rM}\cap\rM'\slash Z_{\rM'})\slash (N_{\rM'}(\tilde{\lambda}_{\rM}')\slash Z_{\rM'}),$$
combining with the fact that $N_{\rM'}\tilde{J}_{\rM}'$ is open, we know that the quotient $E_{\rM}^{\times}\tilde{J}_{\rM}\cap\rM'\slash N_{\rM'}(\tilde{\lambda}_{\rM}')$ is finite, which is abelian, since $E_{\rM}^{\times}\tilde{J}_{\rM}\cap\rM'\slash \tilde{J}_{\rM}'\subset E_{\rM}^{\times}\tilde{J}_{\rM}\slash\tilde{J}_{\rM}$, and the latter is abelian.

For Part $2$, we consider $\res_{\tilde{J}_{\rM}}^{E_{\rM}^{\times}\tilde{J}_{\rM}}\ind_{E_{\rM}^{\times}J_{\rM}}^{E_{\rM}^{\times}\tilde{J}_{\rM}}\Lambda_{\rM}$, and the result can be deduced by Mackey's theory.

For Part $3$, since $\prod_{i=1}^m\tilde{J}_i'\subset N'$, to show that $N'$ is normal in $E_{\rM}^{\times}\tilde{J}_{\rM}$, it is sufficient to show that the quotient $\prod_{i=1}^m E_i^{\times}\tilde{J}_i\slash\tilde{J}_i'$ is abelian, which is deduced from the fact that $E_i^{\times}\tilde{J}_i\cap\mathrm{SL}_{n_i}(F)=\tilde{J}_i'$. Since the irreducible components of $\tilde{\lambda}_{\rM}\vert_{\tilde{J}_{\rM}'}$ are $\tilde{J}_{\rM}$-conjugate, the second part can be deduced from the first part. The last part comes from the definition of $N'$.

For Part $4$, we start from the second part. Let $\tilde{\Lambda}_0$ be an extension of $\tilde{\lambda}_{\rM}$ to $E_{\rM}^{\times}\tilde{J}_{\rM}$. We know that $\res_{E_{\rM}^{\times}J_{\rM}}^{E_{\rM}^{\times}\tilde{J}_{\rM}}\tilde{\Lambda}_0$ is semisimple of finite length, and write it as $\oplus_{s\in S}\Lambda_s$, where $S$ is a finite index set. Since $\res_{J_{\rM}}^{E_{\rM}^{\times}\tilde{J}_{\rM}}\tilde{\Lambda}_0$ is semisimple of finite length, containing $\lambda_{\rM}$ as a sub-representation, and 
$$\Hom(\lambda_{\rM},\bigoplus_{s\in S}\Lambda_s)\cong\bigoplus_{s\in S}\Hom(\lambda_{\rM},\Lambda_s),$$
we deduce that there is an $s_0\in S$ such that $\res_{J_{\rM}}^{E_{\rM}^{\times}J_{\rM}}\Lambda_{s_0}$ contains $\lambda_{\rM}$ as a sub-representation, hence $\Lambda_{s_0}$ is an extension of $\lambda_{\rM}$. Meanwhile $\Lambda_s$ are $E_{\rM}^{\times}\tilde{J}_{\rM}$-conjugate, hence the induction $\ind_{E_{\rM}^{\times}J_{\rM}}^{E_{\rM}^{\times}\tilde{J}_{\rM}}\Lambda_s$ is independent of the choice of $\Lambda_s$. Since $\tilde{\Lambda}_0\hookrightarrow\oplus_{s\in S}\ind_{E_{\rM}^{\times}J_{\rM}}^{E_{\rM}^{\times}\tilde{J}_{\rM}}\Lambda_s$, we conclude that $\tilde{\Lambda}_0\cong\ind_{E_{\rM}^{\times}J_{\rM}}^{E_{\rM}^{\times}\tilde{J}_{\rM}}\Lambda_{s_0}$, where $\Lambda_{s_0}$ is an extension of $\lambda_{\rM}$. We deduce the result from Part $2$.

For the first part, let $\pi$ be an irreducible $k$-representation of $E_{\rM}^{\times}\tilde{J}_{\rM}$ and $\tilde{\lambda}_{\rM}$ appears as a sub-representation. It is left to prove $\pi$ is an extension of $\tilde{\lambda}_{\rM}$. Since $E_{\rM}^{\times}J_{\rM}$ is normal in $E_{\rM}^{\times}\tilde{J}_{\rM}$, the restriction $\res_{E_{\rM}^{\times}J_{\rM}}^{E_{\rM}^{\times}\tilde{J}_{\rM}}\pi$ is semisimple and has finite length, and we write it as $\oplus_{i\in I}\pi_i$. For each $i$, there is a $k$-quasicharacter $\chi$ of $F^{\times}$ such that $\pi_i$ contains $\lambda_{\rM}\otimes\chi\circ\det$ as a sub-representation. In fact, by Frobenius reciprocity, there is a surjection
$$\res_{J_{\rM}}^{E_{\rM}^{\times}\tilde{J}_{\rM}}\ind_{J_{\rM}}^{E_{\rM}^{\times}\tilde{J}_{\rM}}\lambda_{\rM}\rightarrow\res_{J_{\rM}}^{E_{\rM}^{\times}\tilde{J}_{\rM}}\pi.$$
By Mackey's theory, $\res_{J_{\rM}}^{E_{\rM}^{\times}\tilde{J}_{\rM}}\ind_{J_{\rM}}^{E_{\rM}^{\times}\tilde{J}_{\rM}}\lambda_{\rM}$ is semisimple of finite components $\lambda_{\rM}\otimes\chi\circ\det$ for a family of $k$-quasicharacters $\chi$ of $F^{\times}$. We deduce that $\pi_i$ is an extension of $\lambda_{\rM}\otimes\chi\circ\det$. Meanwhile, there is an injection 
$$\pi\hookrightarrow\oplus_{i\in I}\ind_{E_{\rM}^{\times}J_{\rM}}^{E_{\rM}^{\times}\tilde{J}_{\rM}}\pi_i.$$
By the unicity of Jordan-H\"older components, there is an $i\in I$ such that $\pi\cong\ind_{E_{\rM}^{\times}J_{\rM}}^{E_{\rM}^{\times}\tilde{J}_{\rM}}\pi_i$, and the latter is an extension of $\tilde{\lambda}_{\rM}\otimes\chi\circ\det$ by Part $2$. Notice that $\lambda_{\rM}\otimes\chi\circ\det$ and $\lambda_{\rM}$ are $\tilde{J}_{\rM}$-conjugate, hence $\tilde{\lambda}_{\rM}\otimes\chi\circ\det\cong\tilde{\lambda}_{\rM}$.
\end{proof}

\begin{prop}
\label{prop a34}
The restriction $\res_{N_{\rM'}(\tilde{\lambda}_{\rM'})}^{E_{\rM}^{\times}\tilde{J}_{\rM}}\ind_{E_{\rM}^{\times}J_{\rM}}^{E_{\rM}^{\times}\tilde{J}_{\rM}}\Lambda_{\rM}$ is semisimple and multiplicity-free. Furthermore, each irreducible component of $\res_{N_{\rM'}(\tilde{\lambda}_{\rM'})}^{E_{\rM}^{\times}\tilde{J}_{\rM}}\ind_{E_{\rM}^{\times}J_{\rM}}^{E_{\rM}^{\times}\tilde{J}_{\rM}}\Lambda_{\rM}$ is an extension of an irreducible component of $\tilde{\lambda}_{\rM}\vert_{\tilde{J}_{\rM}'}$.
\end{prop}

\begin{proof}
In this proof, we denote $N_{\rM'}(\tilde{\lambda}_{\rM'})$ by $N'$. Let $\Lambda_{\rM}$ be an extension of $(J_{\rM},\lambda_{\rM})$ to $E_{\rM}^{\times}J_{\rM}$. By Frobenius reciprocity and Corollary \ref{cor a32}, we have
$$\res_{\tilde{J}_{\rM}'}^{E_{\rM}^{\times}\tilde{J}_{\rM}}\ind_{E_{\rM}^{\times}J_{\rM}}^{E_{\rM}^{\times}\tilde{J}_{\rM}}\Lambda_{\rM}\cong\ind_{J_{\rM}'}^{\tilde{J}_{\rM}'}\res_{J_{\rM}'}^{J_{\rM}}\lambda_{\rM},$$
while the latter is semisimple and multiplicity-free. In fact 
$$\ind_{J_{\rM}'}^{\tilde{J}_{\rM}'}\res_{J_{\rM'}}^{J_{\rM}}\lambda_{\rM}\cong\res_{\tilde{J}_{\rM}'}^{\tilde{J}_{\rM}}\ind_{J_{\rM}}^{\tilde{J}_{\rM}}\lambda_{\rM},$$
by Corollary \ref{cor a32}, and the induced representation $\ind_{J_{\rM}}^{\tilde{J}_{\rM}}\lambda_{\rM}$ is irreducible by Corollary \ref{cor 99}. Hence $\res_{\tilde{J}_{\rM}'}^{\tilde{J}_{\rM}}\ind_{J_{\rM}}^{\tilde{J}_{\rM}}\lambda_{\rM}$ is semisimple by Proposition \ref{prop 0.1}. For the multiplicity-free part, first we consider the case when $\rM=\rG$. In this case $\ind_{\tilde{J}_{\rM}'}^{\rM'}\res_{\tilde{J}_{\rM}'}^{\tilde{J}_{\rM}}\tilde{\lambda}_{\rM}$ is a sub-representation of $\res_{\rM'}^{\rM}\ind_{E_{\rM}^{\times}\tilde{J}_{\rM}}^{\rM}\tilde{\Lambda}_{\rM}$ which is semisimple and multiplicity-free, hence the same for $\res_{\tilde{J}_{\rM}'}^{\tilde{J}_{\rM}}\tilde{\lambda}_{\rM}$, as well as a sub-representation of $\res_{\tilde{J}_{\rM}'}^{\tilde{J}_{\rM}}\tilde{\lambda}_{\rM}$. For Levi subgroup $\rM$, we write $J_{\rM}=\prod_{i=1}^m J_i$, $\lambda_{\rM}=\prod_{i=1}^m\lambda_i$ where $(J_i,\lambda_i)$ a maximal simple $k$-type of $\mathrm{GL}_{n_i}(F)$. We deduce from the case of $\rG$, that $\res_{\prod_i\tilde{J}_i'}^{\prod_i\tilde{J}_i}\ind_{J_{\rM}}^{\prod_{i}\tilde{J}_i}\lambda_{\rM}$ is semisimple and multiplicity-free. Since $\tilde{J}_{\rM}\subset\prod_{i}\tilde{J}_i$ and $\prod_i\tilde{J}_i'\subset\tilde{J}_{\rM}'$, we obtain the same property for $\res_{\tilde{J}_{\rM}'}^{\tilde{J}_{\rM}}\ind_{J_{\rM}}^{\tilde{J}_{\rM}}\lambda_{\rM}$.

For the second part, we denote the irreducible induction $\ind_{E_{\rM}^{\times}J_{\rM}}^{E_{\rM}^{\times}\tilde{J}_{\rM}}\Lambda_{\rM}$ by $\tilde{\Lambda}_{\rM}$ in this proof. The restriction $\res_{N'}^{E_{\rM}^{\times}\tilde{J}_{\rM}}\tilde{\Lambda}_{\rM}$ is semisimple of finite length, which is deduced from the fact that $\res_{(E_{\rM}^{\times}\tilde{J}_{\rM})'}^{E_{\rM}^{\times}\tilde{J}_{\rM}}\tilde{\Lambda}_{\rM}$ is semisimple. We deduce that $\res_{N'}^{E_{\rM}^{\times}\tilde{J}_{\rM}}\tilde{\Lambda}_{\rM}$ is semisimple of finite length by Lemma \ref{lem a35} Part $1$ and Clifford theory. By Lemma \ref{lem a35} Part $2$, we have $\res_{\tilde{J}_{\rM}'}^{E_{\rM}^{\times}\tilde{J}_{\rM}}\tilde{\Lambda}_{\rM}\cong\res_{\tilde{J}_{\rM}'}^{\tilde{J}_{\rM}}\tilde{\lambda}_{\rM}$, which is semisimple and multiplicity-free as discussed above. On the other hand, the restriction to $\tilde{J}_{\rM}'$ of an irreducible component of $\res_{N_{\rM'}(\tilde{\lambda}_{\rM'})}^{E_{\rM}^{\times}\tilde{J}_{\rM}}\tilde{\Lambda}_{\rM}$ contain an irreducible component of $\tilde{\lambda}_{\rM}\vert_{\tilde{J}_{\rM}'}$. By Lemma \ref{lem a35} Part $3$ and the fact that $\res_{\tilde{J}_{\rM}'}^{E_{\rM}^{\times}\tilde{J}_{\rM}}\ind_{E_{\rM}^{\times}J_{\rM}}^{E_{\rM}^{\times}\tilde{J}_{\rM}}\Lambda_{\rM}$ is multiplicity-free, its restriction must be irreducible, which ends the proof.
\end{proof}

\begin{lem}
\label{lem a36}
Let $\tau'$ be an irreducible $k$-representation of $(E_{\rM}^{\times}\tilde{J}_{\rM})'=E_{\rM}^{\times}\tilde{J}_{\rM}\cap\rM'$, then there is an irreducible $k$-representation $\tau$ of $E_{\rM}^{\times}\tilde{J}_{\rM}$ such that $\tau'$ is isomorphic to a quotient of $\tau\vert_{E_{\rM}^{\times}\tilde{J}_{\rM}\cap\rM'}$.
\end{lem}

\begin{proof}
The proof of Proposition $2.2$ of \cite{TA} can be applied here. The construction of $\tau$ will appear in the proof of Theorem \ref{thm a33}, for the convenience reason we state the part that we need here. Let $S_{\rM}$ be a subgroup of the centre $Z_{\rM}$, consisting of elements whose coefficients on diagonal are powers of $\mathfrak{p}_F$,where $\mathfrak{p}_F$ is an uniformiser of $F$. Then $S_{\rM}\cong\mathbb{Z}^{m}$, where $\rM$ is a product of $m$-blocks. Denote $S_0=S_{\rM}\cap(E_{\rM}^{\times}\tilde{J}_{\rM})'$ and $S_1$ a maximal subgroup of $S_{\rM}$ such that $S_1\cap S_0=\{1\}$, hence the rank of $S_1S_0$ equals to $m$, and the quotient $E_{\rM}^{\times}\tilde{J}_{\rM}\slash S_1(E_{\rM}^{\times}\tilde{J}_{\rM})'$ is compact. We extend $\tau'$ to $S_1(E_{\rM}^{\times}\tilde{J}_{\rM})'$ by acting $S_1$ trivially. The induction $\ind_{S_1(E_{\rM}^{\times}\tilde{J}_{\rM})'}^{E_{\rM}^{\times}\tilde{J}_{\rM}}\tau'$ is admissible, which has an irreducible sub-representation and denote it by $(V,\tau)$. Consider the map $f\rightarrow f(1)$ from $\tau$ to $\tau'$, which is a non-trivial morphism of $(E_{\rM}^{\times}\tilde{J}_{\rM})'$-representation. Hence $\tau'$ is a quotient of $\tau$.
\end{proof}

\begin{thm}
\label{thm a33}
Let $\tau_{\rM'}$ be an irreducible $k$-representation of $N_{\rM'}(\tilde{\lambda}_{\rM'})$ containing an irreducible component $\tilde{\lambda}_{\rM}'$ of $\tilde{\lambda}_{\rM}=\ind_{J_{\rM}}^{\tilde{J}_{\rM}}\lambda_{\rM}$. Then $\tau_{\rM'}$ is an extension of $\tilde{\lambda}_{\rM}'$.
\end{thm}

\begin{proof}
Let $N'$ denote $N_{\rM'}(\tilde{\lambda}_{\rM'})$. The induced $k$-representation $\tau'=\ind_{N'}^{(E_{\rM}^{\times}\tilde{J}_{\rM})'}\tau_{\rM'}$ is irreducible by Theorem \ref{Ltheorem 5}. Let $\tau$ be as in Lemma \ref{lem a36}, which is a sub-representation of $\ind_{S_1(E_{\rM}^{\times}\tilde{J}_{\rM})'}^{E_{\rM}^{\times}\tilde{J}_{\rM}}\tau'$, where $\tau'$ is extended to $S_1(E_{\rM}^{\times}\tilde{J}_{\rM})'$. By Mackey's theory, we have
$$\res_{\tilde{J}_{\rM}}^{E_{\rM}^{\times}\tilde{J}_{\rM}}\ind_{S_1(E_{\rM}^{\times}\tilde{J}_{\rM})'}^{E_{\rM}^{\times}\tilde{J}_{\rM}}\tau'\cong\bigoplus_{\alpha\in S_1(E_{\rM}^{\times}\tilde{J}_{\rM})'\backslash E_{\rM}^{\times}\tilde{J}_{\rM} \slash \tilde{J}_{\rM}} \ind_{\tilde{J}_{\rM}'}^{\tilde{J}_{\rM}}\res_{\tilde{J}_{\rM}'}^{S_1(E_{\rM}^{\times}\tilde{J}_{\rM})'}\alpha(\tau').$$
By Lemma \ref{lem a35} Part $3$ and the fact that $E_{\rM}^{\times}\tilde{J}_{\rM}$ normalises $\tilde{\lambda}_{\rM}$, we obtain that for any $\alpha$ as above, $\res_{\tilde{J}_{\rM}'}^{S_1(E_{\rM}^{\times}\tilde{J}_{\rM})'}\alpha(\tau')$ is a multiple of an irreducible component $\tilde{\lambda}_{\rM,\alpha}'$ of $\tilde{\lambda}_{\rM}\vert_{\tilde{J}_{\rM}'}$. Meanwhile, the induced representation $\ind_{\tilde{J}_{\rM}'}^{\tilde{J}_{\rM}}\tilde{\lambda}_{\rM,\alpha}'$ is a sub-representation of $\tilde{\lambda}_{\rM}\otimes\ind_{\tilde{J}_{\rM}'}^{\tilde{J}_{\rM}}\mathds{1}\cong\ind_{\tilde{J}_{\rM}'}^{\tilde{J}_{\rM}}\res_{\tilde{J}_{\rM}'}^{\tilde{J}_{\rM}}\tilde{\lambda}_{\rM}$. We deduce that, an irreducible sub-quotient of $\res_{\tilde{J}_{\rM}'}^{E_{\rM}^{\times}\tilde{J}_{\rM}}\ind_{S_1(E_{\rM}^{\times}\tilde{J}_{\rM})'}^{E_{\rM}^{\times}\tilde{J}_{\rM}}\tau'$ is isomorphic to an irreducible component of $\tilde{\lambda}_{\rM}\vert_{\tilde{J}_{\rM}'}$.

Now we come back to $\tau$. As in the proof of Lemma \ref{lem 9}, the restriction $\res_{\tilde{J}_{\rM}}^{E_{\rM}^{\times}\tilde{J}_{\rM}}\tau$ is embedding to a direct sum of finite length representations. Hence $\res_{\tilde{J}_{\rM}}^{E_{\rM}^{\times}\tilde{J}_{\rM}}\tau$ has an irreducible sub-representation $\tau_0$. As we has explained in the previous paragraph, $\tau_0\vert_{\tilde{J}_{\rM}'}$ contains an irreducible component of $\tilde{\lambda}_{\rM}\vert_{\tilde{J}_{\rM}'}$, hence after twisting an $k$-quasicharacter of $F^{\times}$ to $\tilde{\lambda}_{\rM}$, we can assume that $\tau_0\cong\tilde{\lambda}_{\rM}$ by Proposition \ref{Lprop 0.2}. Then $\tau$ is an extension of $\tilde{\lambda}_{\rM}$ to $E_{\rM}^{\times}\tilde{J}_{\rM}$ and is isomorphic to $\ind_{E_{\rM}^{\times}J_{\rM}}^{E_{\rM}^{\times}\tilde{J}_{\rM}}\Lambda_{\rM}$, for an extension $\Lambda_{\rM}$ of $\lambda_{\rM}$ by Lemma \ref{lem a35} (4). In other words, $\tau_{\rM'}$ is an irreducible component of $\res_{N'}^{E_{\rM}^{\times}\tilde{J}_{\rM}}\ind_{E_{\rM}^{\times}J_{\rM}}^{E_{\rM}^{\times}\tilde{J}_{\rM}}\Lambda_{\rM}$. We apply Proposition \ref{prop a34} and obtain the result.
\end{proof}

\begin{rem}
In other words, Theorem \ref{thm a33} states that an extended maximal simple $k$-types defined in Definition \ref{Ldefinition 200} is an extension of a maximal simple $k$-type.
\end{rem}

\begin{thm}
Let $(N_i,\tau_i)$ be two extended maximal simple $k$-types of $\rM'$ for $i=1,2$. Assume that there exists an element $g\in\rM'$ such that $g$ weakly intertwines $\tau_1$ with $\tau_2$ then they are in the same $\rM'$-conjugacy class.
\end{thm}

\begin{proof}
Let $(\tilde{J}_i',\tilde{\lambda}_i')$ be the maximal simple $k$-types contained in $(N_i,\tau_i)$ for $i=1,2$. The assumption implies that $\alpha g$ weakly intertwines $\tilde{\lambda}_1'$ to $\tilde{\lambda}_2'$, according to Theorem \ref{thm a30}, they are $\rM'$-conjugate. Now up to conjugate $(N_2,\tau_2)$ by an element in $\rM'$, we can deduce the problem to the case where $\tilde{J}_1'=\tilde{J}_2'$ and $\tilde{\lambda}_1'=\tilde{\lambda}_2'$. Then there exists an element $\beta\in N_1$ such that $\beta g$ weakly intertwines $\tilde{\lambda}_1'$ to itself. By Lemma \ref{Llemma 3}, $\beta g\in N_1$ hence $g\in N_1$, which implies that $\tau_2\cong\tau_1$.
\end{proof}

\appendix
\section{Geometric lemma of Bernstein and Zelevinsky}
\label{appendix A}
We need \cite[Theorem 5.2]{BeZe} in the $\ell$-modular setting. In fact, the proof in \cite{BeZe} is given by the theory of sheaves, which can be translated to a representation theoretical proof and be applied to our case. To be self-contained, we rewrite the proof following the same idea as \cite{BeZe}.

Let $\rG$ be a locally compact totally disconnected group, $\rP$, $\rM$, $\rU$, $\rQ$, $\rN$, $\rV$ are closed subgroups of $\rG$, and $\theta$, $\psi$ be $k$-characters of $\rU$ and $\rV$ respectively. Suppose that they verify conditions $(1)-(4)$ in \cite[\S 5.1]{BeZe}, and denote $\rP\backslash\rG$ by $\rX$. The numbering that we choose in condition $(3)$ is $\rZ_1,...,\rZ_k$ which are $\rQ$-orbits on $\rX$, and for an orbit $\rZ\subset \rX$, we choose $\overline{\omega}\in\rG$ and $\omega$ as in condition $(3)$ of \cite{BeZe}.

We introduce condition $(\ast)$:

$(\ast)$ The characters $\omega(\theta)$ and $\psi$ coincide when restricted to $\omega(\rU)\cap\rV$.

We define $\Phi_{\rZ}$ to be $0$ if $(\ast)$ does not hold, and define $\Phi_{\rZ}$ as in \cite[\S 5.1]{BeZe} if $(\ast)$ holds.

\begin{defn}
Let $\rM,\rU$ be closed subgroups of $\rG$, and $\rM\cap\rU=\{e\}$, and the subgroup $\rP=\rM\rU$ is closed in $\rG$. Let $\theta$ be a $k$-character of $\rU$ normalised by $\rM$. 
\begin{itemize}
\item Define functor $i_{\rU,\theta}:\mathrm{Rep}_{k}(\rM)\rightarrow\mathrm{Rep}_k(\rG)$, which maps $\rho\in\mathrm{Rep}_k(\rM)$, to $\ind_{\rP}^{\rG}\rho_{\rU,\theta}$, where $\rho_{\rU,\theta}\in\mathrm{Rep}_k(\rP)$, such that
$$\rho_{\rU,\theta}(mu)=\theta(u)\mathrm{mod}_{\rU}^{\frac{1}{2}}(m)\rho(m),$$
and $\mathrm{mod}$ is the module, and $u\in\rU,m\in\rM$.
\item Define functor $r_{\rU,\theta}:\mathrm{Rep}_k(\rG)\rightarrow\mathrm{Rep}_k(\rM)$, which maps $(\pi,W)\in\mathrm{Rep}_k(\rG)$ to $\mathrm{mod}_{\rU}^{-\frac{1}{2}}\cdot\res_{\rP}^{\rG}\pi\slash(\res_{\rP}^{\rG}\pi)(\rU,\theta)$, where $(\res_{\rP}^{\rG}\pi)(\rU,\theta)$ is generated by $\pi(u)w-\theta(u)w$, for $w\in W$.
\end{itemize}
\end{defn}

\begin{rem}
\label{rem 5.2.3}
By replacing $\ind$ to $\mathrm{Ind}$, we have $I_{\rU,\theta}$. Notice that $r_{\rU,\theta}$ is the left adjoint to $I_{\rU,\theta}$.
\end{rem}

\begin{prop}
\label{prop 5.2.2}
The functors $i_{\rU,\theta}$ and $r_{\rU,\theta}$ commute with inductive limits.
\end{prop}

\begin{proof}
The functor $r_{\rU,\theta}$ commutes with inductive limits since it has a right adjoint as in \ref{rem 5.2.3}.

For $i_{\rU,\theta}$, let $(\pi_{\alpha},\alpha\in\mathcal{C})$ be an inductive system, where $\mathcal{C}$ is a directed pre-ordered set. We prove that $i_{\rU,\theta}(\underrightarrow{\mathrm{lim}}\pi_{\alpha})\cong\underrightarrow{\mathrm{lim}}(i_{\rU,\theta}\pi_{\alpha}).$ The inductive limit $\underrightarrow{\mathrm{lim}}\pi_{\alpha}$ is defined as $\oplus_{\alpha\in\mathcal{C}}\pi_{\alpha}\slash\sim$, where $\sim$ is an equivalent relation as below. When $\alpha\prec\beta,x\in W_{\alpha},y\in W_{\beta}$, $x\sim y$ if $\phi_{\alpha,\beta}(x)=y$, where $W_{\alpha}$ is the space of $k$-representation $\pi_{\alpha}$, and $\phi_{\alpha,\beta}$ is the morphism from $\pi_{\alpha}$ to $\pi_{\beta}$ in the inductive system.

First, we prove that $i_{\rU,\theta}$ commutes with direct sum. By definition, $\oplus_{\alpha\in\mathcal{C}}i_{\rU,\theta}\pi_{\alpha}$ is a subrepresentation of $i_{\rU,\theta}\oplus_{\alpha\in\mathcal{C}}\pi_{\alpha}$, and the natural embedding is a morphism of $k$-representations of $\rG$. We will prove that the natural embedding is actually surjective. For any $f\in \pi:=i_{\rU,\theta}\oplus_{\alpha\in\mathcal{C}}\pi_{\alpha}$, there exists an open compact subgroup $K$ of $\rG$ such that $f$ is constant on each right $K$ coset of $\rM\rU\backslash\rG$. Furthermore, the function $f$ is non-trivial on finitely many right $K$ cosets. Hence there exists a finite index subset $J\subset\mathcal{C}$, such that $f(g)\in\oplus_{j\in J} W_{j}$, which means $f\in i_{\rU,\theta}\oplus_{j\in J}\pi_{j}$. Since $i_{\rU,\theta}$ commutes with finite direct sum, we finish this case.

The functor $i_{\rU,\theta}$ is exact, we have:
$$i_{\rU,\theta}(\underrightarrow{\mathrm{lim}}\pi_{\alpha})\cong i_{\rU,\theta}(\oplus_{\alpha\in\mathcal{C}}\pi_{\alpha})\slash i_{\rU,\theta}\langle x-y\rangle_{x\sim y}.$$
Notice that $\underrightarrow{\mathrm{lim}} i_{\rU,\theta}\pi_{\alpha}\cong\oplus i_{\rU,\theta}\pi_{\alpha}\slash\thicksim$, where $\thicksim$ is the equivalent relation as below. When $\alpha\prec\beta,f_{\alpha}\in V_{\alpha},f_{\beta}\in V_{\beta}$, where $V_{\alpha}$ is the representation space of $i_{\rU,\theta}\pi_{\alpha}$, then $f_{\alpha}\thicksim f_{\beta}$ if $i_{\rU,\theta}(\phi_{\alpha,\beta})(f_{\alpha})=f_{\beta}$, which is equivalent to say that $\phi_{\alpha,\beta}(f_{\alpha}(g))=f_{\beta}(g)$ for any $g\in\rG$. It is left to prove that the natural isomorphism from $\oplus_{\alpha\in\mathcal{C}}(i_{\rU,\theta}\pi_{\alpha})$ to $i_{\rU,\theta}(\oplus_{\alpha\in\mathcal{C}}\pi_{\alpha})$, induces an isomorphism from $\langle f_{\alpha}-f_{\beta}\rangle_{f_{\alpha}\thicksim f_{\beta}}$ to $i_{\rU,\theta}\langle x-y\rangle_{x\sim y}$. This can be checked directly through definition.
\end{proof}

\begin{thm}[Bernstein, Zelevinsky]
\label{thm 5.2}
Under the conditions above, the functor $\rF=r_{\rV,\psi}\circ i_{\rU,\theta}:\mathrm{Rep}_{k}(\rM)\rightarrow\mathrm{Rep}_{k}(\rN)$ is glued from the functor $\rZ$ runs through all $\rQ$-orbits on $\rX$. More precisely, if orbits are numerated so that all sets $\rY_i=\rZ_1\cup...\cup\rZ_i$ $(i=1,...,k)$ are open in $\rX$, then there exists a filtration $0=\rF_0\subset\rF_1\subset...\subset\rF_k=\rF$ such that $\rF_i\slash\rF_{i-1}\cong\Phi_{\rZ_i}$.
\end{thm}

The quotient space $\rX=\rP\backslash\rG$ is locally compact totally disconnected. Let $\rY$ be a $\rQ$-invariant open subset of $\rX$. We define the subfunctor $\rF_{\rY}\subset\rF$. Let $\rho$ be a $k$-representation of $\rM$, and $W$ be its representation space. We denote $i(W)$ the representation $k$-space of $i_{\rU,\theta}(\rho)$. Let $i_{\rY}(W)\subset i(W)$ the subspace consisting of functions which are equal to $0$ outside the set $\rP\rY$, and $\tau$, $\tau_{\rY}$ be the $k$-representations of $\rQ$ on $i(W)$ and $i_{\rY}(W)$. Put $\rF_{\rY}(\rho)=r_{\rV,\psi}(\tau_{\rY})$, which is a $k$-representation of $\rN$. The functor $\rF_{\rY}$ is a subfunctor of $\rF$ due to the exactness of $r_{\rV,\psi}$.

\begin{prop}
\label{prop 5.2.1}
Let $\rY$, $\rY'$ be two $\rQ$-invariant open subset in $\rX$, we have:
$$\rF_{\rY\cap\rY'}=\rF_{\rY}\cap\rF_{\rY'},\quad \rF_{\rY\cup\rY'}=\rF_{\rY}+\rF_{\rY'},\quad \rF_{\emptyset}=0,\quad \rF_{\rX}=\rF.$$
\end{prop}

\begin{proof}
Since $r_{\rV,\psi}$ is exact, it is sufficient to prove similar formulae for $\tau_{\rY}$. The only non-trivial one is the equality $\tau_{\rY\cup\rY'}=\rF_{\rY}+\rF_{\rY'}$. As in \S $1.3$ \cite{BeZe}, let $K$ be a compact open subgroup of $\rY\cup\rY'$, there exists $\varphi$ and $\varphi'$, which are idempotent $k$-function on $\rY$ and $\rY'$, such that $(\varphi+\varphi')\vert_{K}=1$. We deduce the result from this fact.
\end{proof}

Let $\rZ$ be any $\rQ$-invariant locally closed set in $\rX$, we define the functor 
$$\Phi_{\rZ}:\mathrm{Rep}_{k}(\rM)\rightarrow\mathrm{Rep}_{k}(\rN)$$
to be the functor $\rF_{\rY\cup\rZ}\slash\rF_{\rY}$, where $\rY$ can be any $\rQ$-invariant open set in $\rX$ such that $\rY\cup\rZ$ is open and $\rY\cap\rZ=\emptyset$. Let $\rZ_1,...,\rZ_k$ be the numeration of $\rQ$-orbits on $\rX$ as in \ref{thm 5.2}, and let $\rF_i=\rF_{\rY_i}$ $(i=1,...,k)$, which is a filtration of the functor $\rF$ be the definition. To prove Theorem \ref{thm 5.2}, it is sufficient to prove that $\rF_{\rZ_i}\cong\Phi_{\rZ_i}$. 

By replace $\rP$ to $\omega(\rP)$, we can assume that $\omega=\mathds{1}$. Now we consider the diagram in figure $\mathrm{BZ}$.
\begin{figure}
\centering
\begin{tikzpicture}
\node at (0,0){$\rP\cap\rG$};
\node at (-6,0) {$\rM$};
\node at (6,0){$\rN$};
\node at (-3,1) {$\rP$};
\node at (3,1) {$\rQ$};
\node at (0,2) {$\rG$};
\node at (-2,-1) {$\rM\cap\rQ$};
\node at (-4,-1) {$\rM\cap\rQ$};
\node at (2,-1) {$\rN\cap\rP$};
\node at (4,-1) {$\rN\cap\rP$};
\node at (0,-2) {$\rM\cap\rN$};
\node at (-3,0) {$\mathrm{I}$};
\node at (0,-1) {$\mathrm{II}$};
\node at (0,1) {$\mathrm{III}$};
\node at (3,0) {$\mathrm{IV}$};

\draw[->](-1.6,-0.8)--(-0.45,-0.225);
\draw[->](0.45,-0.225)--(1.6,-0.8);
\draw[->](-1.6,-1.2)--(-0.45,-1.775);
\draw[->](0.45,-1.775)--(1.6,-1.2);
\node at (-1.65,-0.4) {$\rU\cap\rQ$};
\node at (1.65,-0.4) {$\rV\cap\rP$};
\node at (-1.65,-1.6) {$\rV\cap\rM$};
\node at (1.65,-1.6) {$\rU\cap\rN$};

\draw[->](-5.75,0.089)--(-3.25,0.95);
\draw[->](-2.75,1.09)--(-0.25,1.92);
\draw[dashed,->](0.25,1.92)--(2.75,1.09);
\draw[->](3.25,0.91)--(5.75,0.08);
\draw[->](-2.75,0.91)--(-0.5,0.16);
\draw[->](0.5,0.16)--(2.75,0.91);
\node at (-4.75,0.7) {$\rU$};
\node at (-1.75,1.7) {$e$};
\node at (-1.5,0.7) {$e$};
\node at (1.5,0.7) {$e$};
\node at (4.75,0.7) {$\rV$};

\draw[->](-5.75,-0.2)--(-4.6,-0.85);
\draw[->](5.75,-0.2)--(4.6,-0.85);
\node at (-5.3,-0.6) {$e$};
\node at (5.3,-0.6) {$e$};

\draw[-stealth,decorate,decoration={snake,amplitude=2pt,pre length=1pt,post length=1pt}] (-3.4,-1) -- (-2.6,-1);
\draw[-stealth,decorate,decoration={snake,amplitude=2pt,pre length=1pt,post length=1pt}] (2.6,-1) -- (3.4,-1);
\node at (-3,-0.75) {$\varepsilon_1$};
\node at (3,-0.75) {$\varepsilon_2$};

\end{tikzpicture}
\caption{BZ}
\end{figure}

This is the same diagram as in \S $5.7$ \cite{BeZe}, in which a group a point $\rH$ means $\mathrm{Rep}_k(\rH)$, an arrow $\stackrel{\rH}{\nearrow}$ means the functor $i_{\rH,\theta}$, an arrow $\underset{\rH}{\searrow}$ means the functor $i_{\rH,\psi}$, and an arrow $\overset{\varepsilon}{\leadsto}$ means the functor $\varepsilon$ (consult \S $5.1$ \cite{BeZe} for the definition of $\varepsilon$). Notice that $\rG\dashrightarrow\rQ$ does not mean any functor, but the functor $\rP\rightarrow\rG\dashrightarrow\rQ$ is well-defined as explained above \ref{prop 5.2.1}. The composition functors along the highest path is of this diagram is $\rF_{\rZ}$, and if the condition $(\ast)$ holds, the composition functors along the lowest path is $\Phi_{\rZ}$. We prove Theorem \ref{thm 5.2} by showing that this diagram is commutative if condition $(\ast)$ holds, and $\rF_{\rZ}$ equals $0$ otherwise. Notice that parts $\mathrm{I},\mathrm{II},\mathrm{III},\mathrm{IV}$ are four cases of \ref{thm 5.2}, and we prove the statements through verifying them under the four cases respectively. 

Let $\rho$ be any $k$-representation of $\rM$, and $W$ is its representation space. We use $\pi$ to denote $\rF_{\rZ}(\rho)$, and $\tau$ to denote $\Phi_{\rZ}(\rho)$.

Case $\mathrm{I}$: $\rP=\rG,\rV=\{ e \}$. The $k$-representations $\pi$ and $\tau$ act on the same space $W$, and the quotient group $\rM\backslash(\rP\cap\rQ)$ is isomorphic to $(\rM\cap\rQ)\backslash(\rP\cap\rQ)$. We verify directly by definition that $\pi\cong\tau$.

Case $\mathrm{II}$: $\rP=\rG=\rQ$. The representation space of $\pi$ is still $W$. We have the equation:
$$r_{\rV,\psi}(W)\cong r_{\rV\cap\rM,\psi}(r_{\rV\cap\rU,\psi}(W)).$$
If $\theta\vert_{\rU\cap\rV}\neq\psi\vert_{\rU\cap\rV}$, then $\pi=0$ since $\rU\cap\rV=\rU\cap\rQ\cap\rV\cap\rP$ and $r_{\rV\cap\rU,\psi}(W)=0$. This means that after proving the diagrams of cases $\mathrm{I},\mathrm{III},\mathrm{IV}$ are commutative, the functor $\rF_{\rZ}$ equals $0$ if condition $(\ast)$ does not hold.

Now we assume that $(\ast)$ holds. The $k$-representations $\pi$ and $\tau$ act on the same space $W\slash W(\rV\cap\rM,\psi)$, because the fact $r_{\rV\cap\rU,\psi}(i_{\rV\cap\rU,\theta}(W))=W$ and the equation above. Notice that we have equations for $k$-character $\mathrm{mod}$:
$$\mathrm{mod}_{\rU}=\mathrm{mod}_{\rU\cap\rM}\cdot\mathrm{mod}_{\rU\cap\rV},\quad \mathrm{mod}_{\rV}=\mathrm{mod}_{\rV\cap\rN}\cdot\mathrm{mod}_{\rV\cap\rU},$$
from which we deduce that $\pi\cong\tau$ when condition $(\ast)$ holds.

Case $\mathrm{III}$: $\rU=\rV=\{ e \}$. Let $i(W)$ be the representation space of $i_{\rP}^{\rG}\rho$, then $\pi$ acts on a quotient space $W_1$ of $i(W)$. Let:
$$E=\{ f\in i(W)\vert f(\overline{\rP\rQ}\backslash \rP\rQ)=0 \},$$
$$E'=\{ f\in i(W)\vert f(\overline{\rP\rQ})=0 \},$$
then $W_1\cong E\slash E'$. The $k$-representation $\tau$ acts on $i(W)'$, which is the representation space of $i_{\rP\cap\rQ}^{\rQ}\rho$. By definition,
$$i(W)'=\{ h:\rQ\rightarrow W \vert h(pq)=\rho(p)h(q),p\in\rP\cap\rQ,q\in\rQ   \}.$$
We define a morphism $\gamma$ from $W_1$ to $i(W)'$, by sending $f$ to $f\vert_{\rQ}$, which respects $\rQ$-actions and is actually a bijection. For injectivity, let $f_1,f_2\in W_1$ and $f_1\vert_{\rQ}=f_2\vert_{\rQ}$, then $f_1-f_2$ is trivial on $\rP\rQ$, hence $f_1-f_2$ is trivial on $\overline{\rP\rQ}$ by the definition of $E$. This means $f_1-f_2=0$ in $W_1$. Now we prove $\gamma$ is surjective. Let $h\in i(W)'$, there exists an open compact subgroup $K'$ of $(\rP\cap\rQ)\backslash\rQ$ such that $h$ is constant on the right $K'$ cosets of $(\rP\cap\rQ)\backslash\rQ$, and denote $S$ the compact support of $h$. Let $K$ be an open compact subgroup of $\rP\backslash\rG$ such that $(\rP\cap\rQ)\backslash(\rQ\cap K)\subset K'$, and $S\cdot K\cap (\overline{\rP\rQ}\slash \rP\rQ)=\emptyset$. We define $f$ such that $f$ is constant on the right $K$ cosets of $\rP\backslash\rG$, and $f\vert_{(\rP\cap\rQ)\backslash\rQ}=h$. The function $f$ is smooth with compact support on the complement of $\overline{\rP\rQ}\slash \rP\rQ$, hence belongs to $E$, and $\gamma(f)=h$ as desired.

Case $\mathrm{IV}$: $\rU=\{ e\},\rQ=\rG$. We divide this case into two cases $\mathrm{IV}_1$ and $\mathrm{IV}_2$ as in the diagram of figure $\mathrm{Case} \mathrm{IV}$.

\begin{figure}
\begin{tikzpicture}
\node at (0,0) {$\rM$};
\node at (2,1) {$(\rM\cap\rN)\rV$};
\node at (2,-1) {$\rM\cap\rN$};
\node at (5,-1) {$\rM\cap\rN$};
\node at (8,0.5) {$\rN$};
\node at (5,2) {$\rQ=\rN\rV$};

\node at (2,0) {$\mathrm{IV}_2$};
\node at (5,0.5) {$\mathrm{IV}_1$};

\draw[->](0.17,0.15)--(1.25,0.75);
\draw[->](0.17,-0.15)--(1.45,-0.8);

\node at (0.5,0.5) {$e$};
\node at (0.4,-0.7) {$\rV\cap\rM$};

\draw[->](2.75,1.25)--(4.35,1.85);
\draw[->](5.65,1.85)--(7.85,0.66);
\draw[->](2.75,0.85)--(4.47,-0.84);
\draw[->](5.53,-0.84)--(7.85,0.34);
\node at (3.5,1.7) {$e$};
\node at (3.75,0.2) {$\rV$};
\node at (6.7,1.53) {$\rV$};
\node at (6.5,-0.1) {$e$};

\draw[-stealth,decorate,decoration={snake,amplitude=2pt,pre length=1pt,post length=1pt}] (2.6,-1) -- (4.4,-1);
\node at (3.5,-0.75) {$\varepsilon_2$};

\end{tikzpicture}
\caption{Case $\mathrm{IV}$}
\end{figure}

Case $\mathrm{IV}_1$: $\rU=\{ e\},\rQ=\rG,\rV\subset\rM=\rP$. The $k$-representation $\pi$ acts on $i(W)^{+}=i(W)\slash i(W)(\rV,\psi)$, where
$$i(W)_{\rV,\psi}=\langle vf-\psi(v)f, \forall f\in i(W),v\in\rV \rangle.$$
The $k$-representations $\tau$ acts on $i(W^{+})$, which is the smooth functions with compact support on $(\rM\cap\rN)\backslash\rN$ defined as below:
$$\{ h:\rN\rightarrow W\slash W(\rV,\psi)\vert f(mn)=\rho(m)f(n), \forall m\in\rM\cap\rN,n\in\rN \}.$$
There is a surjective projection from $i(W)$ to $i(W^{+})$, which projects $f(n)$ in $W^+=W\slash W_{\rV,\psi}$, for any $f\in i(W)$. In fact, let $h\in i(W^+)$, there exists an open compact subgroup $K$ of $\rP\backslash\rG\cong(\rM\cap\rN)\backslash\rN$, such that $f=\sum_{i=1}^m h_i$, $m\in\mathbb{N}$, where $h_i\in i(W^+)$ is nontrivial on one right $K$ coset $a_i K$ of $\rP\backslash\rG$. We have $h_i\equiv \overline{w_i}$ on $a_i K$, where $w_i\in W$ and $\overline{w_i}\in W^+$. Define $f=\sum_{i=1}^m f_i$, where $f_i\equiv w_i$ on $a_i K$, and equals $0$ otherwise. The function $f\in i(W)$, and the projection image is $h$.

It is clear that this projection induces a morphism from $i(W)^+$ to $i(W^+)$, and we prove this morphism is injective. Let $f,f'\in i(W)^+$, and $f=f'$ in $i(W^+)$. As in the proof above, there exists an open compact subgroup $K_0$ of $\rP\backslash\rG$, and $f_j\in i(W)^+$ such that $f_j$ is non-trivial on one right $K_0$ coset of $\rP\backslash\rG$ and $f-f'=\sum_{j=1}^sf_j$. Furthermore, the supports of $f_j$'s are two-two disjoint. Hence the image of $f_j$ on its support is contained in $W^+$, since $f_j$ is constant on its support, it equals $0$ in $i(W)^+$, whence $f-f'$ equals $0$ in $i(W)^+$. We conclude that this morphism is bijection, and the diagram case $\mathrm{IV}_1$ is commutative. 

Case $\mathrm{IV}_2$: $\rU=\{ e\},\rG=\rQ,\rN\subset\rM$. In this case:
$$\rX=\rN\rV'\backslash\rN\rV\cong\rV'\rV,$$
where $\rV'=\rV\cap\rM$. We choose one Haar measures $\mu$ of $\rX$ (the existence see \S$\mathrm{I}$, 2.8, \cite{V1}). Let $W^+$ denote the quotient $W\slash W(\rV',\psi)$ and $p$ the canonical projection $p: W\rightarrow W^+$. Let $i(W)$ be the space of $k-$representation $\tau= i_{\{e\},1}(\rho)$.

Define $\overline{A}$ a morphism of $k$-vector spaces from $i(W)$ to $W^+$ by:
$$\overline{A}f=\int_{\rV'\backslash\rV}\psi^{-1}(v)p(f(v))\mathrm{d}\mu(v).$$
This is well defined since the function $\psi^{-1}f$ is locally constant with compact support of $\rV'\slash\rV$, and the integral is in fact a finite sum. Since $\mu$ is stable by right translation, we have for any $v\in\rV$:
$$\overline{A}(\tau(v,f))=\psi(v)\overline(A)(f).$$
Hence $\overline{A}$ induces a morphism of $k$-vector spaces:
$$A:i(W)\slash i(W)(V,\psi)\rightarrow W^+.$$
Now we justify that $A\in\mathrm{Hom}_{k[\rN]}(\pi,\tau)$, where $k$-representations $\pi=r_{\rV,\psi}(\tau)$ equals $\rF(\rho)$, and $\tau=\varepsilon_2\cdot r_{\rV',\psi}(\rho)$ equals $\Phi(\rho)$. For any $n\in \rN$:
\begin{eqnarray}
A(\pi(n)f) &=&\mathrm{mod}_{\rV}^{-\frac{1}{2}}(n)\int_{\rV'\slash\rV}\psi^{-1}(v)p(f(vn))\mathrm{d}\mu(v)\\
&=&  \mathrm{mod}_{\rV}^{-\frac{1}{2}}(n)\sigma(n)\mathrm{mod}_{\rV'}^{\frac{1}{2}}(n)\varepsilon_2^{-1}\cdot\int_{\rV'\slash\rV}\psi^{-1}(v)p(f(n^{-1}vn))\mathrm{d}\mu(v)
\end{eqnarray}
By replacing $v'=n^{-1}vn$, the equation above equals to:
$$\sigma(n)\int_{\rV'\slash\rV}\psi^{-1}(v')p(f(v'))\mathrm{d}\mu(v')=\sigma(n)A(f).$$
Therefore $A$ belongs to $\Hom_{k[N]}(\pi,\tau)$, and hence a morphism from functor $\rF$ to $\Phi$. Now we prove that $A$ is an isomorphism.

Let $\rho'$ be the trivial representations of $\{e\}$ on $W$, then $i(W)'$ the space of $k$-representation $\ind_{\rV'}^{\rV}\rho'$ is isomorphic to $i(W)$ the space of $k$-representation $i_{\{e\},1}\rho$. Meanwhile, the diagram \ref{fig 6} $\mathrm{IV}_2$ is commutative, where $A$ indicates the morphism of $k$-vector spaces associated to the functor $A$. Hence it is sufficient to suppose that $\rN=\{e\},\rM=\rV'$. Replacing $\rho$ by $\psi^{-1}\rho$, we can suppose that $\psi=1$.

\begin{figure}
\label{fig 6}
\centering
\begin{tikzpicture}
\node at (-2,0) {$i(W)'\slash i(W)'(\rV,\psi)$};
\node at (1,0) {$W^+$};
\node at (-2,-1.5) {$i(W)\slash i(W)(\rV,\psi)$};
\node at (1,-1.5) {$W^+$};

\draw[->](-0.4,0)--(0.5,0);
\draw[->](-2,-0.2)--(-2,-1.2);
\draw[->](1,-0.2)--(1,-1.2);
\draw[->](-0.4,-1.5)--(0.5,-1.5);

\node at (-2.3,-0.7) {$\cong$};
\node at (1.3,-0.7) {$\cong$};
\node at (0.05,0.2) {$A$};
\node at (0.05,-1.3) {$A$};

\end{tikzpicture}
\caption{$\mathrm{IV}_2$}
\end{figure}

First of all, we consider $\rho=i_{\{e\},1}\mathds{1}=\ind_{e}^{\rV'}\mathds{1}$ the regular $k$-representation on $S(\rV')$, which is the space of locally constant functions with compact support on $\rV'$. Then $\tau=i_{\{e\},1}\rho$ is the regular $k$-representation of $\rV$ on $S(\rV)$ by the transitivity of induction functor. Any $k$-linear form on $r_{\rV',1}(S(\rV'))$ gives a Haar measure on $\rV'$, and conversely any Haar measure gives a $k$-linear form on $S(\rV')$, whose kernel is $S(\rV')(\rV',1)$, hence the two spaces is isomorphic, and the uniqueness of Haar measures implies that the dimension of $r_{\rV',1}(S(\rV'))$ equals one. We obtain the same result to $r_{\rV,1}(S(\rV'))$. Since in this case the morphism $A$ is non-trivial, then it is an isomorphism. The functors $i_{\{e\},1},r_{\rV,\psi},r_{\rV',\psi}$ commute with direct sum (as in \ref{prop 5.2.2}), and the morphism $A$ between $k$-vector spaces also commutes with direct sum, hence $A:\pi\rightarrow\tau$ is an isomorphism when $\rho$ is free, which means $\rho$ is a direct sum of regular $k$-representations of $\rV'$. Notice that any $\rho$ can be viewed as a module over Heck algebra, then $\rho$ is a quotient of some free $k$-representation. Hence $\rho$ has a free resolution. The exactness of $\rF$ and $\Phi$ implies that $A:\rF(\rho)\rightarrow\Phi(\rho)$ is an isomorphism for any $\rho$.


\pagestyle{empty}


\end{document}